\documentclass[11pt,a4paper]{article}

\usepackage[utf8]{inputenc}
\usepackage[T1]{fontenc}
\usepackage[english]{babel} 
\usepackage{amsmath}
\usepackage{amsfonts}
\usepackage{amsthm}
\usepackage{amssymb}
\usepackage{amsbsy}
\usepackage{makeidx}
\usepackage{graphicx}
\usepackage{lmodern}
\usepackage[left=2cm,right=2cm,top=2cm,bottom=2cm]{geometry}
\usepackage{enumitem}
\usepackage{natbib}
\usepackage[normalem]{ulem}
\usepackage{color}
\usepackage{comment}
\usepackage[dvipsnames]{xcolor}
\usepackage{graphicx}
\usepackage{framed}
\usepackage{tikz}
\usepackage{calrsfs}
\DeclareMathAlphabet{\pazocal}{OMS}{zplm}{m}{n}
\usepackage{stmaryrd} % for \{ and \}
\SetSymbolFont{stmry}{bold}{U}{stmry}{m}{n}

\usepackage{scalerel}
\usepackage[usestackEOL]{stackengine}
\def\dashint{\,\ThisStyle{\ensurestackMath{%
  \stackinset{c}{.2\LMpt}{c}{.5\LMpt}{\SavedStyle-}{\SavedStyle\phantom{\int}}}%
  \setbox0=\hbox{$\SavedStyle\int\,$}\kern-\wd0}\int}

\usepackage{fourier-orns}
\usepackage{mathtools}
\usepackage{relsize}
\usepackage{dsfont}
\usepackage{comment}

\usepackage{hyperref}
\hypersetup{
	linktocpage = true,
	colorlinks = true,
	linkcolor = {RoyalBlue},
	urlcolor = {RoyalBlue},
	citecolor = {Green}}

\renewcommand{\P}{\mathbb{P}}
\newcommand{\R}{\mathbb{R}}

\newcommand{\E}{\mathbb{E}}

\newcommand{\Apazo}{\pazocal{A}}

\newcommand{\Cpazo}{\pazocal{C}}

\newcommand{\Epazo}{\pazocal{E}}

\newcommand{\Hpazo}{\pazocal{H}}

\newcommand{\Jpazo}{\pazocal{J}}
\newcommand{\Kpazo}{\pazocal{K}}
\newcommand{\Lpazo}{\pazocal{L}}
\newcommand{\Mpazo}{\pazocal{M}}

\newcommand{\Ppazo}{\pazocal{P}}
\newcommand{\Qpazo}{\pazocal{Q}}
\newcommand{\Rpazo}{\pazocal{R}}
\newcommand{\Spazo}{\pazocal{S}}
\newcommand{\Tpazo}{\pazocal{T}}
\newcommand{\Upazo}{\pazocal{U}}
\newcommand{\Vpazo}{\pazocal{V}}

\newcommand{\Xpazo}{\pazocal{X}}

\newcommand{\Bcal}{\mathcal{B}}

\newcommand{\Hcal}{\mathcal{H}}

\newcommand{\Lcal}{\mathcal{L}}

\newcommand{\Pcal}{\mathcal{P}}

\newcommand{\Ucal}{\mathcal{U}}
\newcommand{\Vcal}{\mathcal{V}}

\newcommand{\Id}{\textnormal{Id}}

\newcommand{\D}{\textnormal{D}}

\newcommand{\Tan}{\textnormal{Tan}}

\newcommand{\supp}{\textnormal{supp}}

\newcommand{\AC}{\textnormal{AC}}

\newcommand{\Div}{\textnormal{div}}
\newcommand{\Hess}{\textnormal{Hess} \,}

\newcommand{\Adm}{\mathrm{Adm}}

\newcommand{\Erm}{\mathrm{E}}
\newcommand{\Lrm}{\mathrm{L}}

\newcommand{\dsf}{\textnormal{\textsf{d}}}

\newcommand{\Xsf}{\textnormal{\textsf{X}}}
\newcommand{\Ysf}{\textnormal{\textsf{Y}}}

\newcommand{\Usf}{\textnormal{\textsf{U}}}
\newcommand{\Vsf}{\textnormal{\textsf{V}}}
\newcommand{\Hsf}{\textnormal{\textsf{H}}}

\newcommand{\Bgamma}{\boldsymbol{\gamma}}

\newcommand{\Bnu}{\boldsymbol{\nu}}

\renewcommand{\epsilon}{\varepsilon}

\newcommand{\INTDom}[3]{\int_{#2} #1 \,\mathrm{d} #3}

\newcommand{\INTSeg}[4]{\int_{#3}^{#4} #1 \,\mathrm{d} #2}
\newcommand{\NormL}[3]{\| #1 \|_{L^{#2}(#3)}}
\newcommand{\NormLbis}[3]{\| #1 \|_{L^2_{#3}(#2)}}
\newcommand{\NormC}[3]{\left\| #1  \right\|_{C^{#2}(#3)}}

\newcommand{\derv}[2]{\frac{\textnormal{d} #1}{ \textnormal{d} #2}}

\newcommand{\deltaX}{\delta \hspace{-0.04cm}X}
\newcommand{\deltaPsi}{\delta \hspace{-0.015cm}\Psi}
\newcommand{\deltaY}{\delta Y}

\newcommand{\deltaLambda}{\delta \hspace{-0.02cm} \Lambda}
\newcommand{\deltau}{\delta \hspace{-0.005cm}u}

\newtheorem{Def}{Definition}[section]
\newtheorem{thm}[Def]{Theorem}
\newtheorem{prop}[Def]{Proposition}
\newtheorem{rmk}[Def]{Remark}
\newtheorem{lem}[Def]{Lemma}
\newtheorem{cor}[Def]{Corollary}

\newtheorem{taggedhypx}{Hypotheses}
\newenvironment{taggedhyp}[1]
    {\renewcommand\thetaggedhypx{#1}\taggedhypx}
    {\endtaggedhypx}

\title{Exponential Turnpike Theorems for Nonlinear Deterministic Meanfield Optimal Control Problems}

\author{Benoît Bonnet-Weill\footnote{Laboratoire des Signaux et Systèmes, Université Paris-Saclay, CNRS, CentraleSupélec, 91190 Gif-sur-Yvette, France. \textit{Email:} \texttt{benoit.bonnet-weill@centralesupelec.fr}}\,, Giovanni Colombo\footnote{Dipartimento di Matematica ``Tullio Levi-Civita'', Università degli Studi di Padova, 63 via Trieste, Padova, Italy.}\,, Denis Shishmintsev$^{\dagger}$\, and Emmanuel Trélat\footnote{Sorbonne Universit\'e, Universit\'e Paris Cit\'e, CNRS, Inria, Laboratoire Jacques-Louis Lions, LJLL, F-75005 Paris, France.}}

\date{}

\begin{document}

\maketitle

\begin{abstract}
In this article, we establish exponential turnpike theorems for a class of nonlinear deterministic meanfield optimal control problems. We carry out our analysis simultaneously in the so-called Lagrangian and Eulerian frameworks. In the Lagrangian setting, the problem is lifted to a Hilbert space of random variables, and we prove an exponential turnpike theorem by combining first-order optimality conditions, a second-order expansion of the lifted Hamiltonian, and an operator Riccati diagonalization argument. In the Eulerian setting, we derive intrinsic KKT conditions for the static constrained problem, and show how the Eulerian second-order hypotheses split into a horizontal part, transferred by unitary conjugation to the lifted space, and a vertical part which reduces to uniform pointwise stabilizability and detectability conditions on multiplication operators. This yields an exponential turnpike theorem in the Wasserstein space for optimal Pontryagin triples. Along the way, we %provide 
explicit the link between Wasserstein Hessians and their Lagrangian lifts, and provide several remarks clarifying the role of occupation measures, local Eulerian minimizers, and control constraints in our results.
\end{abstract}

{\footnotesize
\textbf{Keywords :} Meanfield Optimal Control, Optimal Transport, Lagrangian Lift, Pontryagin Maximum Principle, KKT Conditions, Exponential Turnpike Property, Riccati Theory.

\vspace{0.25cm}

\textbf{MSC2020 Subject Classification :} 49J21, 49K15, 49Q22, 93C25, 93D15.
}

\tableofcontents

%%%%%%%%%%%%%%%%%%%%%%%%%%%%%%%%%%%%%%%%%%%%%%%%%%%%%%%%%%%%%%%%%%%%%%%
%						   NEW SECTION AHEAD						   %
%%%%%%%%%%%%%%%%%%%%%%%%%%%%%%%%%%%%%%%%%%%%%%%%%%%%%%%%%%%%%%%%%%%%%%%

\section{Introduction}
\setcounter{equation}{0} \renewcommand{\theequation}{\thesection.\arabic{equation}}

The turnpike phenomenon is one of the most robust asymptotic signatures of optimal control. Roughly speaking, the latter asserts that optimal solutions of large horizon problems spend most of their time close to a distinguished point -- which is stationary for the limit ergodic problem --, up to exponentially thin boundary layers near the initial and terminal times. Since the seminal nonlinear finite-dimensional result of \cite{Trelat2015} and its Hilbertian counterpart \cite{Trelat2018}, turnpike theory has developed in many directions; see the recent survey \cite{TrelatZuazuaSurvey} and references therein 

The purpose of this article is to generalize this tried theory to deterministic meanfield optimal control problems of the form
\begin{equation*}
(\Ppazo_{\Erm}) ~~ \left\{
\begin{aligned}
\min_{(\mu,u_{\Erm})} & \bigg[ \INTSeg{\INTDom{L \Big( \mu(t),x,u_{\Erm}(t,x) \Big)}{\R^d}{\mu(t)(x)}}{t}{0}{T} + \varphi(\mu(T)) \bigg], \\
\textnormal{s.t.}~ & \left\{
\begin{aligned}
& \partial_t \mu(t) + \Div_x \Big( v(\mu(t),\cdot,u_{\Erm}(t,\cdot)) \mu(t) \Big) = 0, \\
& \mu(0) = \mu^0.
\end{aligned}
\right.
\end{aligned}
\right.
\end{equation*}
Therein $v : \Pcal_2(\R^d) \times \R^d \times U \to \R^d$ is a \textit{nonlocal vector field} describing the evolution of the system, that is controlled by means of a Lebesgue-Borel control signal $u_{\Erm} : [0,T] \times \R^d \to U$, whereas $L : \Pcal_2(\R^d) \times \R^d \times U \to \R$ and $\varphi : \Pcal_2(\R^d) \to \R$ are running and terminal costs, respectively. There is a very large literature on problems of the form $(\Ppazo_{\Erm})$, or small variations thereof, which traces back to the foundation of the theory of \textit{meanfield games} \cite{Lasry2007,Huang2006}. The latter was subsequently revisited through a control-theoretic viewpoint, first in a stochastic context \cite{Bensoussan2013,Carmona2013}, and then in a more deterministic setting \cite{FornasierPR2014,Fornasier2014}. To study meanfield optimal control problems, two general methodologies have been proposed during the past ten to fifteen years. \vspace{-0.15cm}
\begin{enumerate}[wide, labelindent=0pt]
\item[$\diamond$] On the one hand, one may study the problem in its so-called \textit{Eulerian formulation}, which operates directly at the level of probability measures by posing the problem in Wasserstein spaces and using geometric and functional analytic tools of optimal transport theory \cite{AGS,OTAM}. It is arguably the more intrinsic framework, and saw the development of deep existence results \cite{Cavagnari2022,Fornasier2019} along with the derivation of Pontryagin \cite{Bongini2017,Burger2021,SetValuedPMP,Pogodaev2016} and Hamilton-Jacobi-Bellman \cite{Aussedat2024,Badreddine2022,Marigonda2018} optimality conditions conditions. This framework is particularly well-adapted to investigate problems in which the controls do not depend on the space variable, which arise naturally in Follow-the-Leader \cite{FornasierPR2014} and Machine Learning \cite{E2019} models. \vspace{-0.15cm}
\item[$\diamond$] On the other hand, one may rephrase the problem using the so-called \textit{Lagrangian formulation}, which provides an equivalent reformulation as an optimal control problem in a Hilbert space of square-integrable random variables. This idea, which goes back to \cite{Cardaliaguet2010,Carmona2018}, allows for a cleaner treatment of problems in which the admissible controls are vector fields instead of time-dependent signals, due to the fact that such closed-loop controls become open-loop when lifted. For that reason, they recently gained a fair amount of steam in the field, again to prove Pontryagin optimality conditions \cite{Averboukh2025} and study Hamilton-Jacobi-Bellman equations \cite{Capuani2025,Jimenez2023}, but also to revisit Wasserstein geometry \cite{Bertucci2024,Bertucci2025}. 
\end{enumerate}
We stress that rigorous connections between the Eulerian and Lagrangian formalisms were proven in the ground-breaking work \cite{Cavagnari2022} -- see also the extension \cite{Averboukh2025} --, where it is shown that the corresponding value functions of both problems coincide, and that precise transfer mechanisms between trajectory-control pairs in both settings are available. \medskip

As the turnpike phenomenon is about optimal trajectories spending a long time in the vicinity of a stationary solution, we naturally consider the companion static problem
\begin{equation*}
(\Ppazo_{\Erm}^s) ~~ 
\left\{
\begin{aligned}
\min_{(\mu,u_{\Erm})} & \INTDom{L \Big( \mu,x,u_{\Erm}(x) \Big)}{\R^d}{\mu(x)} , \\
\textnormal{s.t.}~ & v(\mu,u_{\Erm}) = 0 ~~ \text{in $L^2_{\mu}(\R^d,\R^d)$}.
\end{aligned}
\right.
\end{equation*}
Therein, the functional equality should be interpreted in the sense that $v(\mu,x,u_{\Erm}(x)) = 0$ for $\mu$-almost every $x \in \R^d$, so that being an admissible pair means that the support of the candidate measure is effectively a collection of equilibrium points for the controlled vector field. Perhaps surprisingly, this is not the equality constraint one might formally derive from $(\Ppazo_{\Erm})$ by simply setting $\Div_x(v(\mu,u_{\Erm})) = 0$ in the sense of distributions. Although the latter makes perfect sense at the PDE level -- and is highly reminiscent of the usual Beckmann problem \cite[Section 4.2]{OTAM} --, its admissible measures are stationary distributions instead of collections of equilibrium points, and thus harder to approach via flow techniques. We stress nonetheless that the equality constraint we consider is already fairly challenging to handle, and not covered by existing works on optimization in measure spaces.  

%Passing from one framework to the other is classical at the level of admissible trajectory-control pairs and cost functions, see e.g. \cite{Averboukh2025,Cavagnari2022}, but the second-order analysis needed for turnpike estimates is subtler and requires a careful comparison between lifted Hessians and Wasserstein Hessians.

\paragraph*{Overview of contributions.} In this paper, we prove exponential turnpike estimates for smooth constrained problems with nonlinear meanfield interactions of the form $(\Ppazo_{\Erm})$, under second-order assumptions expressed on the stationary Hamiltonian. %Our first contribution is a systematic review of the relationship between first- and second-order Eulerian and Lagrangian differential structures. At the level of first order, we recall the classical lift formula for Wasserstein gradients derived in \cite{Gangbo2019} and subsequently simplified in \cite{Jimenez2023}. At second order, we revisit a representation formula for lifted Hessians proposed in \cite{Chassagneux2022} in terms of the local and nonlocal coefficients of the Wasserstein Hessian. This leads to a convenient operator-theoretic viewpoint: the lifted Hessian is the sum of a multiplication operator and an integral operator, and on the horizontal subspace generated by the stationary realization, it is unitarily conjugated to the intrinsic Wasserstein Hessian.
Our first contribution is a systematic review of the relationship between first- and second-order Eulerian and Lagrangian differential structures. At the level of first-order calculus, we simply recall the classical lift formula for Wasserstein gradients derived in \cite{Gangbo2019} and subsequently simplified in \cite{Jimenez2023}. At second-order, we revisit a representation formula for lifted Wasserstein Hessians proposed in \cite{Chassagneux2022} in terms of the local and nonlocal coefficients of the Wasserstein Hessian \cite{Choi2017,villani1}. While studying second-order calculus in Wasserstein spaces and its connections with Hilbertian lifts is evidently not new, our derivation of exponential turnpike estimates requires a subtle analysis of the relationships between Hessians operators in the Eulerian and Lagrangian frameworks, and in particular of their respective spectra. This leads us to exhibiting the following deep operator-theoretic connection: the lifted Hessian is the sum of a multiplication operator and an integral operator, and on the horizontal subspace generated by a stationary realization, it is unitarily conjugated to the intrinsic Wasserstein Hessian.

Our second contribution consists in new first-order optimality conditions for the static Eulerian problem $(\Ppazo_{\Erm}^s)$, exposed in Theorem \ref{thm:KKTEulerian}. To this end, we first recollect classical KKT conditions for the static Lagrangian problem in the lifted Hilbert space, and justify their transferability to the Eulerian constrained problem. The Eulerian statement is tailored for infinite-dimensional constraints, and with the view of serving as a linearisation point for the Wasserstein PMP, and is therefore naturally expressed in terms of a state-costate measure in the spirit of \cite{PMPWassConst,SetValuedPMP,PMPWass} (see also \cite{Lanzetti2024bis}). A key point here is that the relevant notion of local Eulerian minimizer only involves proximity of the state measure in Wasserstein distance. This avoids imposing an artificial topology on Eulerian controls, whose canonical reconstruction from Lagrangian ones via disintegration is not stable in general. We stress that a stronger local topology is available at the level of occupation measures, but it is not needed for the results of the present paper.

Our third and main contribution is the proof of exponential turnpike theorems in both Lagrangian and Eulerian formulations. In Theorem \ref{thm:TurnpikeLagrangian}, we start by establishing an abstract Hilbertian turnpike theorem in the Lagrangian setting. The latter is prescribed by the Hessian of the lifted Hamiltonian at a stationary Pontryagin triple, and its proofs carefully leverages linearization techniques along with results of algebraic Riccati theory in infinite dimension inspired by \cite{Trelat2018}. In Theorem \ref{thm:TurnpikeEulerian}, we then formulate intrinsic Eulerian turnpike estimates for $(\Ppazo_{\Erm})$, which hinge upon hypotheses formulated directly at the level of the Wasserstein Hessian of a suitable relaxed Hamiltonian. Those can be split into two parts: a horizontal contribution, which transfers to the lifted space by unitary conjugation, and a vertical one, which reduces to uniform pointwise stabilizability and detectability conditions for the multiplication operators acting on the orthogonal complement. This decomposition explains precisely which directions are seen by the Eulerian Hessian, and which additional finite-dimensional conditions are needed to control singular orthogonal perturbations in the lifted space.

\paragraph*{Related works.} For meanfield control problems, turnpike results are presently available essentially in linear-quadratic or closely related settings; see for instance \cite{Gugat2024,SunYong2024,BayraktarJian2025,BayraktarJian2026}, as well as the recent particle-to-hydrodynamic perspective in \cite{HertyZhou2025}. By contrast, the genuinely nonlinear deterministic case seems to have remained largely open. 

\paragraph*{Organization of the paper.} Section \ref{section:Preliminaries} collects the basic facts on Wasserstein calculus, liftings and meanfield Pontryagin theory that we shall employ to derive turnpike estimates. Section \ref{section:LagrangianEulerian} contains the static KKT analysis in the Lagrangian and Eulerian settings. In Section \ref{section:Turnpike}, we prove the exponential turnpike theorem in the lifted Hilbert framework and then transfer it to Wasserstein space. Appendix \ref{section:AppendixHess} contains the proof of the second-order lift formula. We also include several remarks on the role of control constraints and on the intrinsic tangent-space viewpoint. 

%%%%%%%%%%%%%%%%%%%%%%%%%%%%%%%%%%%%%%%%%%%%%%%%%%%%%%%%%%%%%%%%%%%%%%%	
%  							NEW SECTION AHEAD	     		 		  %
%%%%%%%%%%%%%%%%%%%%%%%%%%%%%%%%%%%%%%%%%%%%%%%%%%%%%%%%%%%%%%%%%%%%%%%

\section{Preliminaries}
\label{section:Preliminaries}

\setcounter{equation}{0} \renewcommand{\theequation}{\thesection.\arabic{equation}}

In this section, we recollect basic material on optimal transport, integration and probability, and Wasserstein calculus, for which we broadly refer to the monographs \cite{villani1}, \cite{AnalysisBanachSpaces} and \cite{Chassagneux2022}. 

%%%%%%%%%%%%%%%%%%%%%%%%%%%%%%%%%%%%%%%%%%%%%%%%%%%%%%%%%%%%%%%%%%%%%%%	

\paragraph*{Function spaces, integration and optimal transport.} 

 Given two separable Hilbert spaces $(\Xsf,\langle\cdot,\cdot\rangle_{\Xsf})$,$(\Ysf,\langle\cdot,\cdot\rangle_{\Ysf})$, we denote by $\Lpazo(\Xsf,\Ysf)$ the space of bounded linear operators from $\Xsf$ to $\Ysf$. Given two subspaces $\Usf \subset \Xsf$ and $\Vsf \subset \Ysf$, we also write $C^0_b(\Usf,\Vsf)$ and $C^1_c(\Usf,\Vsf)$ for the spaces of continuous bounded maps and continuously differentiable maps with compact support. Given a real number $T > 0$ and a Polish space $(\Xpazo,\dsf_{\Xpazo}(\cdot,\cdot))$, we denote by $\AC^2([0,T],\Xpazo)$ the space of all $\Lcal^1$-measurable maps $x : [0,T] \to \Xpazo$ satisfying
\begin{equation*}
\dsf_{\Xpazo}(x(t_1),x(t_2)) \leq \INTSeg{m(s)}{s}{t_1}{t_2}
\end{equation*}
for all times $0 \leq t_1 < t_2 \leq T$ and some $m(\cdot) \in L^2([0,T],\R_+)$. 
 
Given a probability measure $\mu \in \Pcal(\Xsf)$, we write $L^2_{\mu}(\Xsf,\Ysf)$ for the usual Hilbert space of square-integrable maps from $\Xsf$ to $\Ysf$, endowed with its canonical norm. When there is no confusion, we shall identify a map with its equivalence class, and work with its Borel representative whenever needed \cite[Corollary 6.5.6]{Bogachev}. %may omit the target space from the notation. 
In what follows, we shall use the notation $\Gamma(\mu,\nu)$ for the set of \textit{transport plans} between two probability measures $\mu,\nu \in \Pcal(\Xsf)$, defined by 
\begin{equation*}
\Gamma(\mu,\nu) := \Big\{ \gamma \in \Pcal(\Xsf^2) ~\,\mathrm{s.t.}~ \pi^1_{\sharp}\gamma = \mu ~~\text{and}~~ \pi^2_{\sharp}\gamma = \nu \Big\}
\end{equation*}
where $\pi^1,\pi^2 : \Xsf \times \Xsf \to \Xsf$ are the projections onto the first and second component, and ``$\sharp$'' stands for the usual image measure operation. We denote by $\Pcal_2(\Xsf)$ the subset of probability measures with a finite moment of order 2, and recall the definition of the Wasserstein distance
\begin{equation*}
W_2(\mu,\nu) := \inf \bigg\{\bigg( \INTDom{|x-y|^2}{\Xsf^2}{\gamma(x,y)} \bigg)^{1/2} ~\mathrm{s.t.}~ \gamma \in \Gamma(\mu,\nu) \bigg\}
\end{equation*}
between two measures $\mu,\nu \in \Pcal_2(\R^d)$. Below, we let $\Gamma_o(\mu,\nu)$ stand for the corresponding set of optimal plans, which is nonempty by the direct method of the calculus of variations \cite[Chapter 4]{villani1}. 

\paragraph{Random variables and Lagrangian lifts.} Throughout the paper, we let $(\Omega,\Apazo,\P)$ be a standard atomless probability space, namely $\Omega$ is (isomorphic to) a Polish space, $\Apazo$ is the $\P$-completion of the native Borel $\sigma$-algebra $\Bcal(\Omega)$, and $\P$ has no atoms. We shall denote by $L^2_{\P}(\Omega,\Xsf)$ the space of square-integrable $\Xsf$-valued random variables on $(\Omega,\Apazo,\P)$. Given a random variable $X \in L^2_{\P}(\Omega,\R^d)$, we write $\sigma(X)$ for its induced $\sigma$-algebra, and recall that the Hilbert space $L^2_{\P}(\Omega,\R^d;\sigma(X))$ of $X$-adapted square-integrable random variables can be characterized by  
\begin{equation}
\label{eq:AdaptedRandomVariable}
L^2_{\P}(\Omega,\R^d;\sigma(X)) = \overline{\Big\{ \xi \circ X ~\,\textnormal{s.t.}~ \xi : \R^d \to \R^d ~~\text{is bounded and Borel} \Big\}}^{L^2_{\P}(\Omega,\R^d)}.
\end{equation}
Recall then that the conditional expectation $\E[\cdot|\sigma(X)] : L^2_{\P}(\Omega,\R^d) \to L^2_{\P}(\Omega,\R^d;\sigma(X))$ with respect to $\sigma(X)$ is defined as the orthogonal projection onto $L^2_{\P}(\Omega,\R^d;\sigma(X))$. In particular given $H \in L^2_{\P}(\Omega,\R^d)$, one has that $H = \E[H|\sigma(X)] + H^{\perp}_X$ where
\begin{equation}
\label{eq:CondExpDecomposition}
\langle\xi \circ X , H^{\perp}_X \rangle_{L^2_{\P}(\Omega,\R^d)} = 0 
\end{equation}
for every bounded Borel map $\xi : \R^d \to \R^d$. 

Given a measure $\mu \in \Pcal_2(\Xsf)$, there exists by the generalization of Skorokhod's theorem proven in \cite{Berti2007} (and recollected in \cite[Proposition 2.1]{Cavagnari2022}) a random variable $X \in L^2_{\P}(\Omega,\Xsf)$ such that $X_{\sharp}\P = \mu$. This induces the equivalence relation 
\begin{equation*}
X \sim Y \qquad \text{if and only if} \qquad X_{\sharp} \P = Y_{\sharp} \P 
\end{equation*}
over $L^2_{\P}(\Omega,\R^d)$, through which the definition of the Wasserstein distance can be recast as 
\begin{equation*}
W_2(\mu,\nu) := \bigg\{ \|X-Y\|_{L^2_{\P}(\Omega,\R^d)} ~\,\textnormal{s.t.}~ X_{\sharp} \P = \mu ~~\text{and}~~ Y_{\sharp}\P = \nu \bigg\}. 
\end{equation*}
This yields in particular the identification of $(\Pcal_2(\R^d),W_2(\cdot,\cdot))$ as the quotient $L^2_{\P}(\Omega,\R^d) / \hspace{-0.1cm} \sim$. 
 %This simple observation is the backbone of the Eulerian-Lagrangian correspondence used all along the paper. In particular, the Wasserstein differential structure of $\Phi$ is transported to the lifted Hilbert space through composition operators, while conversely the lifted formulation provides a convenient Hilbertian realization of intrinsic objects on $\Pcal_2(\R^d)$.

%%%%%%%%%%%%%%%%%%%%%%%%%%%%%%%%%%%%%%%%%%%%%%%%%%%%%%%%%%%%%%%%%%%%%%%	

\paragraph*{First- and second-order calculus over $\Pcal_2(\R^d)$ and $L^2_{\P}(\Omega,\R^d)$.}

In the following definition, we explicit the notions of first- and second-order derivatives for a map over $(\Pcal_2(\R^d),W_2(\cdot,\cdot))$ that we shall use in the sequel. Therein, given two measures $\mu,\nu \in \Pcal_2(\R^d)$, we let 
\begin{equation}
\label{eq:WeightedWass}
W_{2,\gamma}(\mu,\nu) := \bigg( \INTDom{|x-y|^2}{\R^{2d}}{\gamma(x,y)} \bigg)^{1/2}
\end{equation}
stand for the weighted Wasserstein metric along an admissible plan $\gamma \in \Gamma(\mu,\nu)$. Note in particular that $W_{2,\gamma}(\mu,\nu) = W_2(\mu,\nu)$ whenever $\gamma \in \Gamma_o(\mu,\nu)$. 

\begin{Def}[Fréchet differentiable functions over $\Pcal_2(\R^d)$]
\label{def:WassGrad}
We say that a mapping $\Phi : \Pcal_2(\R^d) \to \R$ is \emph{Fréchet differentiable} at $\mu \in \Pcal_2(\R^d)$ if there exists an element $\nabla_{\mu} \Phi(\mu) \in \Tan_{\mu} \Pcal_2(\R^d)$ such that 
\begin{equation}
\label{eq:FréchetDiff}
\Phi(\nu) = \Phi(\mu) + \INTDom{\langle \nabla_{\mu} \Phi(\mu)(x) , y-x \rangle}{\R^{2d}}{\gamma(x,y)} + \mathrm{o}(W_{2,\gamma}(\mu,\nu))
\end{equation}
for each $\nu \in \Pcal_2(\R^d)$ and every admissible plan $\gamma \in \Gamma(\mu,\nu)$. More generally, if $\Phi : \Pcal_2(\R^d) \to \R^d$ for some $m \geq 1$, then we say that it is Fréchet differentiable at $\mu \in \Pcal_2(\R^d)$ provided that its components are Fréchet differentiable, and we then let $\nabla_{\mu} \Phi(\mu) := (\nabla_{\mu} \Phi_i(\mu))_{1 \leq i \leq m} \in L^2(\R^d,\R^{d \times m};\mu)$. 
\end{Def}

Before moving on, observe that every mapping $\Phi : \Pcal_2(\R^d) \to \R$ that is Fréchet differentiable at $\mu\in\Pcal_2(\R^d)$ is clearly continuous at that point, since 
\begin{equation*}
|\Phi(\nu) - \Phi(\mu)| \leq W_2(\mu,\nu) \Big( \NormLbis{\nabla_{\mu}\Phi(\mu)}{\R^d}{\mu} + \mathrm{o}(1) \Big)
\end{equation*}
by choosing optimal plans $\gamma \in \Gamma_o(\mu,\nu)$ in the identity \eqref{eq:FréchetDiff}.

\begin{rmk}[On the definition of Wasserstein derivative]
While the above definition of Wasserstein gradients may seem more restrictive than its classical geometric counterpart involving optimal plans introduced in \cite{AGS,AmbrosioGangbo}, it was shown e.g. in \cite[Proposition 3.6]{SetValuedPMP} (see also \cite[Section 3.2]{Jimenez2023} and \cite{Lanzetti2025}) that they are in fact equivalent.
\end{rmk}

Recall that a map $\widetilde{\Phi} : L^2_{\P}(\Omega,\R^d) \to \R$ is said to be law-invariant if it is of the form $\widetilde{\Phi}(X) = \Phi(X_{\sharp}\P)$ for some functional $\Phi : \Pcal_2(\R^d) \to \R$, called the Eulerian representative of $\widetilde{\Phi}$. It is now commonly known that the Fréchet differentiability of a law-invariant functional over $L^2_{\P}(\Omega,\R^d)$ is equivalent to that of its Eulerian representative on $\Pcal_2(\R^d)$ in the sense of Definition \ref{def:WassGrad}. We refer to the excellent works \cite{Gangbo2019,Jimenez2023} for more details about this fairly nontrivial result. Below, we provide a statement and short proof of the direct implication of this claim, as it happens to follow elementarily from Definition \ref{def:WassGrad} and we will primarily use it in the sequel.

\begin{prop}[First-order differentiability and gradient formula]
\label{prop:WassGrad}
Suppose that $\Phi : \Pcal_2(\R^d) \to \R$ is Fréchet differentiable in the sense of Definition \ref{def:WassGrad}. Then, its lift $\widetilde{\Phi} : L^2_{\P}(\Omega,\R^d) \to \R$ is Fréchet differentiable in the usual sense, and it holds that
\begin{equation*}
\nabla \widetilde{\Phi}(X) = \nabla_{\mu} \Phi(X_{\sharp}\P) \circ X
\end{equation*}
for every $X \in L^2_{\P}(\Omega,\R^d)$.  
\end{prop}

\begin{proof}
Given any pair of random variables $X,H \in L^2_{\P}(\Omega,\R^d)$, let $\gamma := (X,X+H)_{\sharp} \P$ and simply note that, from the Fréchet differentiability of $\Phi : \Pcal_2(\R^d) \to \R$ at $X_{\sharp} \P \in \Pcal_2(\R^d)$ combined with the definition \eqref{eq:WeightedWass} of the weighted Wasserstein metric, it follows that
\begin{equation*}
\begin{aligned}
\widetilde{\Phi}(X+H) & = \Phi((X+H)_{\sharp} \P) \\
& = \Phi(X_{\sharp} \P) + \INTDom{\langle \nabla_{\mu} \Phi(X_{\sharp} \P)(x) , y-x \rangle}{\R^{2d}}{\gamma(x,y)} + \mathrm{o}(W_{2,\gamma}(\mu,\nu)) \\
& = \widetilde{\Phi}(X) + \INTDom{\Big\langle \nabla_{\mu} \Phi(X_{\sharp} \P)(X(\omega)) , H(\omega) \Big\rangle}{\Omega}{\P(\omega)} + \mathrm{o} \Big( \NormL{H}{2}{\Omega,\R^d;\P} \Big). 
\end{aligned}
\end{equation*}
Since $X,H \in L^2_{\P}(\Omega,\R^d)$ were arbitrary, this concludes the proof. 
\end{proof}

%The above result provides us with the crucial information that derivatives of rearrangement-invariant maps at some $X \in L^2_{\P}(\Omega,\R^d)$ belongs to the subspace of admissible directions 
%%
%\begin{equation}
%\Trm_X L^2_{\P}(\Omega,\R^d) := \Big\{ \xi \circ X ~\,\textnormal{s.t.}~ \xi \in \Tan_{X_{\sharp} \P} \Pcal_2(\R^d) \Big\}
%\end{equation}
%%
%which was already identified to be the good lifting space for tangent vectors in \cite{Gangbo2019}. 
In the sequel, we will also need a notion of second-order derivative over $\Pcal_2(\R^d)$. We adopt here the terminology of \cite{Chassagneux2022}, but also refer the reader e.g. to \cite{Chow2019} for an equivalent and more geometric-flavoured definition of these objects in terms of Lie derivatives. 

\begin{Def}[Twice continuously differentiable functions over $\Pcal_2(\R^d)$]
\label{def:WassHess}
We say that a map $\Phi : \Pcal_2(\R^d) \to \R$ is \emph{twice continuously differentiable} if the mapping $(\mu,x) \in \Pcal_2(\R^d) \times \R^d \to \nabla_{\mu} \Phi(\mu)(x) \in \R^d$ is continuous and Fréchet differentiable, with the second-order variations 
\begin{equation*}
(\mu,x) \in \Pcal_2(\R^d) \times \R^d \mapsto \D_x \nabla_{\mu} \Phi(\mu)(x) \in \R^{d \times d}
\end{equation*}
and 
\begin{equation*}
(\mu,x,y) \in \Pcal_2(\R^d) \times \R^d \times \R^d \mapsto \nabla_{\mu}^2 \Phi(\mu)(x,y) \in \R^{d \times d}
\end{equation*}
being both continuous. In what follows, we write that $\Phi \in C_b^2(\Pcal_2(\R^d),\R)$ provided that it is twice continuously differentiable with bounded second-order derivatives, and define its \textnormal{Wasserstein Hessian} at any $\mu \in \Pcal_2(\R^d)$ as the bilinear form $\Hess(\Phi)(\mu) : L^2_{\mu}(\R^d,\R^d) \times L^2_{\mu}(\R^d,\R^d) \to \R$ given by 
\begin{equation*}
\Hess \Phi(\mu)(\xi,\zeta) := \INTDom{\Big\langle \D_x \nabla_{\mu} \Phi(\mu)(x) \xi(x) , \zeta(x) \Big\rangle}{\R^d}{\mu(x)} + \INTDom{\INTDom{\Big\langle \nabla_{\mu}^2 \Phi(\mu)(x,y) \xi(y) , \zeta(x) \Big\rangle}{\R^d}{\mu(y)}}{\R^d}{\mu(x)}
\end{equation*}
for all $\xi,\zeta \in L^2_{\mu}(\R^d,\R^d)$. 
\end{Def}

Observe that following the above definition, the Wasserstein Hessian of a function can be equivalently represented by a linear operator $\Hess \Phi(\mu) : L^2_{\mu}(\R^d,\R^d) \to L^2_{\mu}(\R^d,\R^d)$, which is given by 
\begin{equation*}
\Big( \Hess \Phi(\mu) \, \xi \Big)(x) := \D_x \nabla_{\mu} \Phi(\mu)(x) \xi(x) + \INTDom{\nabla_{\mu}^2 \Phi(\mu)(x,y) \xi(y)}{\R^d}{\mu(y)}
\end{equation*}
for each $\xi \in L^2_{\mu}(\R^d,\R^d)$ and $\mu$-almost every $x \in \R^d$. In particular, note that the latter is the sum of a multiplication operator and a kernel operator. As alluded to e.g. in \cite[Section 3.1]{Chassagneux2022}, it can be shown that $\Phi \in C_b^2(\Pcal_2(\R^d),\R)$ if and only if its lift $\widetilde{\Phi} : L^2_{\P}(\Omega,\R^d) \to \R$ is an element of $C^2_b(L^2_{\P}(\Omega,\R^d),\R)$. Again, we provide a statement and a proof of the direct implication of this claim, for the sake of completeness. 

\begin{prop}[Second-order differentiability and Hessian formula]
\label{prop:WassHess}
Let $\Phi \in C_b^2(\Pcal_2(\R^d),\R)$ in the sense of Definition \ref{def:WassHess}. Then, its lift $\widetilde{\Phi} : L^2_{\P}(\Omega,\R^d) \to \R$ is twice continuously Fréchet differentiable with bounded second-order derivatives in the usual sense, with 
\begin{equation}
\label{eq:HessianFormula}
\begin{aligned}
\nabla^2 \widetilde{\Phi}(X)(G,H) & = \INTDom{\Big\langle \D_x \nabla_{\mu}\Phi(X_{\sharp} \P)(X(\omega))  G(\omega) , H(\omega) \Big\rangle}{\Omega}{\P(\omega)} \\
& \hspace{0.45cm} + \INTDom{\INTDom{\Big\langle \nabla^2_{\mu} \Phi(X_{\sharp} \P)(X(\omega),X(\theta)) G(\theta), H(\omega) \Big\rangle}{\Omega}{\P(\theta)}}{\Omega}{\P(\omega)}
\end{aligned}
\end{equation}
for all $X,G,H \in L^2_{\P}(\Omega,\R^d)$. 
\end{prop}

\begin{proof}
Being a combination of fairly standard but technical arguments, this proof is deferred to Appendix \ref{section:AppendixHess} below.
\end{proof}

\begin{lem}[Composition operators and the representation of lifted Hessian]
\label{lem:LiftedConjugation}
Let $X \in L^2_{\P}(\Omega,\R^d)$ be given, set $\mu := X_{\sharp}\P \in \Pcal_2(\R^d)$, and define the composition operator $\Cpazo_X : L^2_{\mu}(\R^d,\R^d) \to L^2_{\P}(\Omega,\R^d)$ by 
\begin{equation*}
\Cpazo_X \xi := \xi \circ X
\end{equation*}
for each $\xi \in L^2_{\mu}(\R^d,\R^d)$. Then, $\Cpazo_X$ is an isometry onto the closed subspace $L^2_{\P}(\Omega,\R^d;\sigma(X))$, and its adjoint $\Cpazo_X^{\star} : L^2_{\P}(\Omega,\R^d) \to L^2_{\mu}(\R^d,\R^d)$ is characterized by the identity
\begin{equation}
\label{eq:AdjointComposition}
\INTDom{\langle (\Cpazo_X^{\star} H)(x), \xi(x) \rangle}{\R^d}{\mu(x)} = \INTDom{\langle H(\omega), \xi \circ X(\omega) \rangle}{\Omega}{\P(\omega)}
\end{equation}
for every $H \in L^2_{\P}(\Omega,\R^d)$ and $\xi \in L^2_{\mu}(\R^d,\R^d)$. In particular, the operator $\Cpazo_X \Cpazo_X^{\star} \in \Lpazo(L^2_{\P}(\Omega,\R^d))$ coincides with the orthogonal projection onto $L^2_{\P}(\Omega,\R^d;\sigma(X))$, namely
\begin{equation*}
\Cpazo_X \Cpazo^{\star}_X H = \mathbb{E}[H\,\vert \sigma(X)]. 
\end{equation*}
Lastly, if $\Phi \in C_b^2(\Pcal_2(\R^d),\R)$, then the operator representing its Hessian $\nabla^2 \widetilde{\Phi}(X)$ decomposes into
\begin{equation}
\label{eq:LiftedHessianStructure}
\nabla^2 \widetilde{\Phi}(X) = \Mpazo_X + \Cpazo_X \Kpazo_{\mu} \Cpazo_X^{\star},
\end{equation}
where
\begin{equation*}
(\Mpazo_X H)(\omega) = \D_x \nabla_{\mu} \Phi(\mu)(X(\omega)) H(\omega) \qquad \text{and} \qquad
(\Kpazo_{\mu} \, \xi)(x) = \INTDom{\nabla_{\mu}^2 \Phi(\mu)(x,y)\xi(y)}{\R^d}{\mu(y)}.
\end{equation*}
In particular, the Hilbert subspace $L^2_{\P}(\Omega,\R^d;\sigma(X))$ is invariant under $\nabla^2 \widetilde{\Phi}(X)$, and
\begin{equation}
\label{eq:HessianUnitaryConjugation}
\nabla^2 \widetilde{\Phi}(X)_{\mid L^2_{\P}(\Omega,\R^d;\sigma(X))} = \Cpazo_X \, \Hess \Phi(\mu) \, \Cpazo_X^{\star}.
\end{equation}
\end{lem}

\begin{proof}
The isometry property of $\Cpazo_X$ follows directly from the change of variable formula
\begin{equation*}
\INTDom{|(\Cpazo_X \xi)(\omega)|^2}{\Omega}{\P(\omega)} = \INTDom{|\Cpazo_X \xi \circ X(\omega)|^2}{\Omega}{\P(\omega)} = \INTDom{|\xi(x)|^2}{\R^d}{\mu(x)}.
\end{equation*}
The characterization \eqref{eq:AdjointComposition} of $\Cpazo_X^{\star}$ is simply the definition of the adjoint, while the identification $\Cpazo_X \Cpazo_X^{\star} = \E[\cdot|\sigma(X)]$ follows from the uniqueness of orthogonal projections onto closed subspaces. Concerning the decomposition \eqref{eq:LiftedHessianStructure}, is a just a rewriting of \eqref{eq:HessianFormula}. Indeed, the first term therein is precisely 
\begin{equation*}
\INTDom{\Big\langle \D_x \nabla_{\mu}\Phi(X_{\sharp} \P)(X(\omega))  G(\omega) , H(\omega) \Big\rangle}{\Omega}{\P(\omega)} = \big\langle \Mpazo_X G, H \big\rangle_{L^2_{\P}(\Omega,\R^d)},
\end{equation*}
whereas due to \eqref{eq:AdjointComposition}, the second term can be recast as 
\begin{equation*}
\begin{aligned}
& \INTDom{\INTDom{\Big\langle \nabla^2_{\mu} \Phi(X_{\sharp} \P)(X(\omega),X(\theta)) G(\theta), H(\omega) \Big\rangle}{\Omega}{\P(\theta)}}{\Omega}{\P(\omega)} \\
& \hspace{0.45cm} = \INTDom{\INTDom{\Big\langle \nabla_{\mu}^2 \Phi(\mu)(x,y) (\Cpazo_X^{\star} G)(y) , (\Cpazo_X^{\star} H)(x) \Big\rangle}{\R^d}{\mu(y)}}{\R^d}{\mu(x)} \\
& \hspace{0.45cm} = \big\langle \Cpazo_X \Kpazo_{\mu} \Cpazo_X^{\star} G , H\big\rangle_{L^2_{\P}(\Omega,\R^d)}.
\end{aligned}
\end{equation*}
Finally, when $G = \xi \circ X$ for some bounded Borel map $\xi : \R^d \to \R^d$, one clearly has $\Cpazo_X^{\star} G = \xi$, and \eqref{eq:HessianUnitaryConjugation} follows immediately from the definition of $\Hess \Phi(\mu)$ and taking the closure in $L^2_{\P}(\Omega,\R^d)$.
\end{proof}

%%%%%%%%%%%%%%%%%%%%%%%%%%%%%%%%%%%%%%%%%%%%%%%%%%%%%%%%%%%%%%%%%%%%%%%	

\paragraph*{Meanfield optimal control.} We now recall necessary facts about the variational problems considered throughout the paper. Given some $\mu^0 \in \Pcal_2(\R^d)$, we consider in the sequel the following \textit{meanfield optimal control problem} in Eulerian form
\begin{equation}
(\Ppazo_{\Erm}) ~~ \left\{
\begin{aligned}
\min_{(\mu,u_{\Erm})} & \bigg[ \INTSeg{\INTDom{L \Big( \mu(t),x,u_{\Erm}(t,x) \Big)}{\R^d}{\mu(t)(x)}}{t}{0}{T} + \varphi(\mu(T)) \bigg], \\
\textnormal{s.t.}~ & \left\{
\begin{aligned}
& \partial_t \mu(t) + \Div_x \Big( v(\mu(t),\cdot,u_{\Erm}(t,\cdot)) \mu(t) \Big) = 0, \\
& \mu(0) = \mu^0,
\end{aligned}
\right.
\end{aligned}
\right.
\end{equation}
where the dynamics is understood in the sense of distributions 
\begin{equation}
\label{eq:DistribCE}
\INTSeg{\INTDom{\bigg( \partial_t \varphi(t,x) + \Big\langle\nabla_x \varphi(t,x) , v(\mu(t), x,u_{\Erm}(t,x)) \Big\rangle \bigg)}{\R^d}{\mu(t)(x)}}{t}{0}{T} = 0
\end{equation}
for each $\varphi \in C^1_c((0,T) \times \R^d)$, and the controls belongs to the (trajectory dependent) admissible set 
\begin{equation*}
\Ucal_{\Erm}[\mu(\cdot)] := \Bigg\{u_{\Erm} : [0,T] \times \R^d \to U ~\,\textnormal{s.t.}~ \text{$\Lcal^1 \hspace{-0.05cm} \times \hspace{-0.05cm} \Bcal(\R^d)$-measurable\, and}~ \INTSeg{\NormLbis{u_{\Erm}(t)}{\R^d,\R^d}{\mu(t)}^2}{t}{0}{T} < +\infty \Bigg\}.
\end{equation*}
Given now a random variable $X^0 \in L^2_{\P}(\Omega,\R^d)$ satisfying the initial assignment condition $X^0_{\sharp} \P = \mu^0$, the Lagrangian lift of $(\Ppazo_{\Erm})$ is given by 
\begin{equation}
(\Ppazo_{\Lrm}) ~~
\left\{
\begin{aligned}
\min_{(X,u_{\Lrm})} & \bigg[ \INTSeg{\INTDom{L \Big( X(t)_{\sharp}\P, X(t,\omega),u_{\Lrm}(t,\omega) \Big)}{\Omega}{\P(\omega)}}{t}{0}{T} + \varphi(X(T)_{\sharp}\P) \bigg], \\
\textnormal{s.t.}~ & \left\{
\begin{aligned}
& \dot{X}(t,\omega) = v \Big( X(t)_{\sharp}\P, X(t,\omega),u_{\Lrm}(t,\omega) \Big), \\
& X(0) = X^0,
\end{aligned}
\right.
\end{aligned}
\right.
\end{equation}
where the dynamics is understood in the strong sense -- or equivalently $\P$-almost everywhere in $\omega \in \Omega$, see e.g. \cite[Appendix A.1]{Cavagnari2022} -- and the controls belong to the admissible set 
\begin{equation*}
\Ucal_{\Lrm} := \Bigg\{u_{\Lrm} : [0,T] \times \Omega \to U ~\,\textnormal{s.t.}~ \text{$\Lcal^1 \times \P$-measurable\, and}~ \INTSeg{\NormLbis{u_{\Lrm}(t)}{\Omega,\R^d}{\P}^2}{t}{0}{T} < +\infty \Bigg\}.
\end{equation*}
For later use, we introduce the following admissible classes of trajectory-control pairs
\begin{equation*}
\Adm(\Ppazo_{\Lrm}) = \bigg\{ (X(\cdot),u_{\Lrm}(\cdot)) \in \AC^2([0,T],L^2_{\P}(\Omega,\R^d)) \times \Ucal_{\Lrm} ~\; \text{admissible for for $(\Ppazo_{\Lrm})$} \bigg\}
\end{equation*}
and 
\begin{equation*}
\Adm(\Ppazo_{\Erm}) = \bigg\{ (\mu(\cdot),u_{\Erm}(\cdot)) \in \AC^2([0,T],\Pcal_2(\R^d)) \times \Ucal_{\Erm}[\mu(\cdot)] ~\; \text{admissible for $(\Ppazo_{\Erm})$} \bigg\}.
\end{equation*}
Throughout the article, we will work under the following basic assumptions. 

\begin{taggedhyp}{\textnormal{(H)}}
\label{hyp:H}
The data of the control problem satisfy the following assumptions.
\begin{enumerate}
\item[$(i)$] The control set $U \subset \R^m$ is a nonempty subspace.
\item[$(ii)$] The mappings $v : \Pcal_2(\R^d) \times \R^d \times U \to \R^d$ and $L : \Pcal_2(\R^d) \times \R^d \times U \to \R$ are continuous, twice continuously differentiable with respect to $(x,u) \in \R^d \times U$, twice continuously Fréchet differentiable with respect to $\mu \in \Pcal_2(\R^d)$ in the sense of Definition \ref{def:WassHess}, and all first- and second-order derivatives are bounded on bounded subsets of the relevant spaces.
\item[$(iii)$] The dynamics is affine with respect to the control variable, namely $v(\mu,x,(1-\alpha)u + \alpha v) = (1-\alpha)v(\mu,x,u) + \alpha v(\mu,x,v)$ for each $(\mu,x) \in \Pcal_2(\R^d) \times \R^d$, every $u,v \in U$ and all $\alpha \in [0,1]$. 
%$v(\mu,x,u) = v_0(\mu,x) + B(\mu,x)u$.
%
\item[$(iv)$] For every $(\mu,x) \in \Pcal_2(\R^d) \times \R^d$, the map $u \in U \mapsto L(\mu,x,u) \in \R$ is strictly convex %, and strictly convex on neighbourhoods of the stationary controls considered in Section \ref{section:Turnpike}.
\item[$(v)$] The terminal cost $\varphi : \Pcal_2(\R^d) \to \R$ is twice continuously differentiable in the sense of Definition \ref{def:WassHess} with bounded derivatives. % satisfies $\varphi \in \widetilde{\Cpazo}^2_b(\Pcal_2(\R^d),\R)$.
%
%\item[$(vi)$] The state equations in \eqref{eq:LagrangianDynamicProblem} and \eqref{eq:EulerianDynamicProblem} are well posed for admissible controls, and the Pontryagin maximum principles quoted below apply.
\end{enumerate}
\end{taggedhyp}

\begin{rmk}[Concerning Hypotheses \ref{hyp:H}]
Note that Hypotheses \ref{hyp:H}-$(ii)$ and \ref{hyp:H}-$(v)$ are purposely stated in a slightly redundant way in order to cover simultaneously the differential statements of Sections \ref{section:Preliminaries} and \ref{section:LagrangianEulerian}, and the second-order turnpike analysis of Section \ref{section:Turnpike}. In concrete examples, and in particular in linear-quadratic models, these assumptions are easily checked directly on the coefficients.
\end{rmk}

In our subsequent developments, we shall need the following pointwise Hamiltonian 
\begin{equation}
\label{eq:HamiltonianClassical}
H(\mu,x,p,u) = \langle p , v(\mu,x,u) \rangle - L(\mu,x,u)
\end{equation}
for every $(\mu,x,p,u) \in \Pcal_2(\R^d) \times \R^{2d} \times U$, along with its Lagrangian lift defined by 
\begin{equation}
\label{eq:HamiltonianLag}
\widetilde{\Hpazo}(X,\Psi,u_{\Lrm}) = \INTDom{H \Big( X_{\sharp}\P,X(\omega),\Psi(\omega),u_{\Lrm}(\omega) \Big)}{\Omega}{\P(\omega)}
\end{equation}
for each $(X,\Psi,u_{\Lrm}) \in L^2_{\P}(\Omega,\R^{2d} \times U)$. The ensuing turnpike analysis only invokes Pontryagin principles that are already available in the literature, see \cite{Averboukh2025}, respectively in the lifted Hilbert setting and in the Wasserstein space. For reading convenience, we state them in the form needed below.

\begin{thm}[Pontryagin Maximum Principle for $(\Ppazo_{\Lrm})$]
\label{thm:PMPLagrangian}
Let $(X^*(\cdot),u_{\Lrm}^*(\cdot)) \in \Adm(\Ppazo_{\Lrm})$ be a strong local minimizer, namely there exists $\epsilon > 0$ such that every $(X(\cdot),u_{\Lrm}(\cdot)) \in \Adm(\Ppazo_{\Lrm})$ satisfying 
\begin{equation*}
\sup_{t \in [0,T]} \NormLbis{X^*(t) - X(t)}{\Omega,\R^d}{\P} + \INTSeg{\NormLbis{u^*_{\Lrm}(t) - u_{\Lrm}(t)}{\Omega,U}{\P}^2}{t}{0}{T} \leq \epsilon
\end{equation*}
produces a higher cost. Then, there exists $\Psi^*(\cdot) \in \AC^2([0,T],L^2_{\P}(\Omega,\R^d))$ such that the following holds.
\begin{enumerate}
\item[$(i)$] The pair $(X^*(\cdot),\Psi^*(\cdot)) \in \AC^2([0,T],L^2_{\P}(\Omega,\R^{2d}))$ solves the boundary-value problem
\begin{equation}
\label{eq:HamiltonianDynLag}
\left\{
\begin{aligned}
& \dot{X}^*(t) = \nabla_{\Psi} \widetilde{\Hpazo}(X^*(t),\Psi^*(t),u_{\Lrm}^*(t)), \qquad\quad X^*(0) = X^0, \\
& \dot{\Psi}^*(t) = -\nabla_X \widetilde{\Hpazo}(X^*(t),\Psi^*(t),u_{\Lrm}^*(t)), \qquad\, \Psi^*(T) = -\nabla \widetilde{\varphi}(X^*(T)),
\end{aligned}
\right.
\end{equation}
for $\Lcal^1$-almost every $t \in [0,T]$.
\item[$(ii)$] The maximization condition
\begin{equation}
\label{eq:MaximizationLag}
H \Big( X^*(t)_{\sharp}\P, X^*(t,\omega),\Psi^*(t,\omega),u_{\Lrm}^*(t,\omega) \Big) = \max_{u \in U} H \Big( X^*(t)_{\sharp}\P, X^*(t,\omega),\Psi^*(t,\omega),u \Big)
\end{equation}
holds for $\Lcal^1 \times \P$-almost every $(t,\omega) \in [0,T] \times \Omega$.
\end{enumerate}
\end{thm}

\begin{thm}[Pontryagin Maximum Principle for $(\Ppazo_{\Erm})$]
\label{thm:PMPEulerian}
Let $(\mu^*(\cdot),u_{\Erm}^*(\cdot)) \in \Adm(\Ppazo_{\Erm})$ be a local $W_2$-minimizer for $(\Ppazo_{\Erm})$, namely there exists $\epsilon > 0$ such that every $(\mu(\cdot),u_{\Erm}(\cdot)) \in \Adm(\Ppazo_{\Erm})$ satisfying
\begin{equation*}
\sup_{t \in [0,T]} W_2(\mu(t),\mu^*(t)) \leq \epsilon
\end{equation*}
produces a higher cost. Then, there exists $\nu^*(\cdot) \in \AC^2([0,T],\Pcal_2(\R^{2d}))$ such that the following holds.
\begin{enumerate}
\item[$(i)$] The curve $\nu^*(\cdot) \in \AC^2([0,T],\Pcal_2(\R^{2d}))$ solves the boundary-value problem
\begin{equation}
\label{eq:HamiltonianDynEul}
\left\{
\begin{aligned}
& \partial_t \nu^*(t) + \Div_{(x,p)} \big( \Vcal^*(t)\nu^*(t) \big) = 0, \\
& \pi^1_{\sharp} \nu^*(t) = \mu^*(t) \quad \textnormal{for all $t \in [0,T]$}, \\
& \nu^*(T) = \Big( \Id,-\nabla_{\mu}\varphi(\mu^*(T)) \Big)_{\raisebox{4pt}{\scriptsize$\sharp$}} \, \mu^*(T),
\end{aligned}
\right.
\end{equation}
where
\begin{equation*}%\label{eq:PMPVectorFieldEulerian}
\Vcal^*(t,x,p) = \begin{pmatrix}
v \Big( \mu^*(t),x,u_{\Erm}^*(t,x) \Big) \\
- \D_x v \Big( \mu^*(t),x,u_{\Erm}^*(t,x) \Big)^{\raisebox{-4pt}{\scriptsize $\top$}} p - \INTDom{\D_{\mu} v \Big( \mu^*(t),y,u_{\Erm}^*(t,y) \Big)(x)^{\top} q}{\R^{2d}}{\nu^*(t)(y,q)}
\end{pmatrix}
\end{equation*}
for  $\Lcal^1$-almost every $t \in [0,T]$ and $\nu^*(t)$-almost every $(x,p) \in \R^{2d}$.
\item[$(ii)$] The maximization condition
\begin{equation}
\label{eq:MaximizationEul}
H \Big( \mu^*(t),x,p,u_{\Erm}^*(t,x) \Big) = \max_{u \in U} H \big( \mu^*(t),x,p,u \big)
\end{equation}
holds for $\Lcal^1$-almost every $t \in [0,T]$ and $\nu^*(t)$-almost every $(x,p) \in \R^{2d}$.
\end{enumerate}
\end{thm}

We close this preliminary section by reformulating the main results of \cite[Section 8]{Averboukh2025}, relating Lagrangian and Eulerian Pontryagin triples and the underlying PMPs.  

\begin{prop}[Realization of Eulerian optimal pairs]
\label{prop:EulerianRealization}
Under Hypotheses \ref{hyp:H}, every Eulerian Pontryagin triple $(\nu(\cdot),u_{\Erm}(\cdot))$ satisfying \eqref{eq:HamiltonianDynEul}-\eqref{eq:MaximizationEul} is realized by a Lagrangian Pontryagin triple $(X(\cdot),\Psi(\cdot),u_{\Lrm}(\cdot))$ satisfying \eqref{eq:HamiltonianDynLag}-\eqref{eq:MaximizationLag}, in the sense that
\begin{equation*}
(X(t),\Psi(t))_{\sharp}\P = \nu(t)
\qquad \textnormal{and} \qquad
u_{\Lrm}(t,\omega) = u_{\Erm}(t,X(t,\omega))
\end{equation*}
for $\Lcal^1 \times \P$-almost every $(t,\omega) \in [0,T] \times \Omega$. Moreover, if $(\mu(\cdot),u_{\Erm}(\cdot))$ is a local $W_2$-minimizer for $(\Ppazo_{\Erm})$, then $(X(\cdot),u_{\Lrm}(\cdot))$ may be chosen as a strong local minimizer for $(\Ppazo_{\Lrm})$.
\end{prop}

%\begin{proof}
%This follows from the Eulerian-Lagrangian correspondence developed in \cite{Averboukh2025,Cavagnari2022}: one first realizes the initial law and the state-costate flow on the reference probability space, then pulls back the Eulerian control by composition with the realized state variable. The optimality statement follows from the equivalence of the costs under the graph lifting $(x,u) \mapsto (\Id,u)_{\sharp}\mu$.
%\end{proof}

\section{First-order optimality conditions for the static problems}
\label{section:LagrangianEulerian}

\setcounter{equation}{0} \renewcommand{\theequation}{\thesection.\arabic{equation}}

In this section, we derive first-order KKT optimality conditions for the static Lagrangian and Eulerian problems, that will both be needed to establish the turnpike theorems of Section \ref{section:Turnpike}. While the Lagrangian KKT conditions of Section \ref{subsection:LagrangianKKT} are well-known, they do play a pivotal role in proving the Eulerian ones from Section \ref{subsection:EulerianKKT}, and are thus tailored to this purpose. The latter are incidentally quite novel, and improve upon those e.g. of \cite{Lanzetti2025} in that they allow to treat infinite-dimensional constraints whose target space depends on the measure. 

%%%%%%%%%%%%%%%%%%%%%%%%%%%%%%%%%%%%%%%%%%%%%%%%%%%%%%%%%%%%%%%%%%%%%%%

\subsection{KKT conditions for the static Lagrangian problem}
\label{subsection:LagrangianKKT}

We start by studying the static Lagrangian problem, which we recall writes as
\begin{equation*}%\label{eq:LagrangianStatic}
\tag{$\Ppazo_{\mathrm{L}}^s$}
\left\{ 
\begin{aligned}
\min_{(X,u_{\Lrm})} & \INTDom{L \Big(X_{\sharp} \P,X(\omega),u_{\Lrm}(\omega) \Big)}{\Omega}{\P(\omega)}, \\
\text{s.t.}~\, &
v \big( X_{\sharp} \P,X,u_{\Lrm} \big) = 0 ~~ \text{in $L^2_{\P}(\Omega,\R^d)$}.
\end{aligned}
\right.
\end{equation*}
The latter can be equivalently recast as the following standard equality-constrained optimization problem 
\begin{equation*}
\left\{
\begin{aligned}
\min_{(X,u_{\Lrm})} & \, \widetilde{\Jpazo}(X,u_{\Lrm}), \\
\text{s.t.} ~\, & \widetilde{\Vpazo}(X,u_{\Lrm}) = 0, 
\end{aligned}
\right.
\end{equation*}
where we introduced the lifted cost function $\widetilde{\Jpazo}: L^2_{\P}(\Omega,\R^d \times U) \to \R$ and vector field $\widetilde{\Vpazo}: L^2_{\P}(\Omega,\R^d \times U) \to L^2_{\P}(\Omega,\R^d)$, given by
\begin{equation*}
\widetilde{\Jpazo}(X,u_{\Lrm}) := \INTDom{L\Big( X_{\sharp} \P,X(\omega),u_{\Lrm}(\omega) \Big)}{\Omega}{\P(\omega)} \qquad \text{and} \qquad \widetilde{\Vpazo}(X,u_{\Lrm}) := v \big( X_{\sharp} \P,X,u_{\Lrm} \big) ,
\end{equation*}
for every $(X,u_{\Lrm}) \in L^2_{\P}(\Omega,\R^d \times U)$. In what follows, we let 
\begin{equation*}
\Adm(\Ppazo^s_{\Lrm}) := \Big\{ (X,u_{\Lrm}) \in L^2_{\P}(\Omega,\R^d \times U) ~\,\textnormal{s.t.}~ \widetilde{\Vpazo}(X,u_{\Lrm}) = 0 \Big\}
\end{equation*}
stand for the collection of admissible pairs of $(\Ppazo_{\Lrm}^s)$.

\begin{thm}[KKT conditions for the static Lagrangian problem]
\label{thm:KKTLagrangian}
Suppose that Hypotheses \ref{hyp:H} hold and let $(\bar{X}^s,\bar{u}_{\Lrm}^s) \in \Adm(\Ppazo_{\Lrm}^s)$ be a strong local minimizer at which
\begin{equation*}
\D \widetilde{\Vpazo}(\bar{X}^s,\bar{u}_{\Lrm}^s) : L^2_{\P}(\Omega,\R^d \times U) \to L^2_{\P}(\Omega,\R^d)
\end{equation*}
is surjective. Then, there exists a multiplier $\bar{\Psi}^s \in L^2_{\P}(\Omega,\R^d)$ such that
\begin{equation*}
\nabla \widetilde{\Jpazo}(\bar{X}^s,\bar{u}_{\Lrm}^s) - \D \widetilde{\Vpazo}(\bar{X}^s,\bar{u}_{\Lrm}^s)^{\star}\bar{\Psi}^s = 0.
\end{equation*}
Equivalently, the stationary Pontryagin triple $(\bar{X}^s,\bar{\Psi}^s,\bar{u}_{\Lrm}^s) \in L^2_{\P}(\Omega,\R^{2d} \times U)$ satisfies
\begin{equation}
\label{cor:LagrangianKKTHamil}
\nabla \widetilde{\Hpazo}(\bar{X}^s,\bar{\Psi}^s,\bar{u}_{\Lrm}^s) = 0
\end{equation}
in $L^2_{\P}(\Omega,\R^{2d} \times U)$.
\end{thm}

\begin{lem}[Differential of the lifted constraints]
\label{lem:DifferentialConstraints}
For every $(X,u_{\Lrm}) \in \Adm(\Ppazo_{\Lrm}^s)$, it holds that
\begin{equation*}%\label{eq:DiffLiftedConstraints}
\begin{aligned}
\Big( \D \widetilde{\Vpazo}(X,u_{\Lrm})(H,v_{\hspace{0.01cm}\Lrm}) \Big)(\omega)
& = \D_{(x,u)} v \Big( X_{\sharp}\P,X(\omega),u_{\Lrm}(\omega) \Big)(H(\omega),v_{\hspace{0.01cm}\Lrm}(\omega)) \\
& \hspace{0.45cm} + \INTDom{\D_{\mu} v \Big( X_{\sharp}\P,X(\omega),u_{\Lrm}(\omega) \Big)(X(\theta)) H(\theta)}{\Omega}{\P(\theta)}
\end{aligned}
\end{equation*}
for all $(H,v_{\hspace{0.01cm}\Lrm}) \in L^2_{\P}(\Omega,\R^d \times U)$. In particular, if there exists $c > 0$ such that
\begin{equation}
\label{eq:QuantitativeSurjectivityLag}
\Big| \D_{(x,u)} v \Big( X_{\sharp}\P,X(\omega),u_{\Lrm}(\omega) \Big)(y,v) \Big| \geq c \, \Vert (y,v) \Vert_{\R^d \times U}
\end{equation}
for almost every $\omega \in \Omega$ and every $(y,v) \in \R^d \times U$, then $\D \widetilde{\Vpazo}(X,u_{\Lrm})$ is surjective provided it is injective.
\end{lem}

\begin{proof}
The differential formula is an immediate consequence of Proposition \ref{prop:WassGrad} applied componentwisely to the lifted map $(X,u_{\Lrm}) \in L^2_{\P}(\Omega,\R^d \times U) \mapsto \mapsto v(X_{\sharp}\P,X,u_{\Lrm}) \in L^2_{\P}(\Omega,\R^d)$. Indeed, the first term comes from the pointwise differentiation in the random variables $(X,u_{\Lrm})$, whereas the second term corresponds to the Wasserstein differential in the law variable $X_{\sharp}\P \in \Pcal_2(\R^d)$.

To prove the surjectivity criterion, consider the decomposition
$\D \widetilde{\Vpazo}(X,u_{\Lrm}) = \widetilde{\Mpazo}_{X,u_{\Lrm}} + \widetilde{\Kpazo}_{X,u_{\Lrm}}$
where
\begin{equation*}
(\widetilde{\Mpazo}_{X,u_{\Lrm}}(H,v_{\hspace{0.01cm}\Lrm}))(\omega) = \D_{(x,u)} v \Big( X_{\sharp}\P,X(\omega),u_{\Lrm}(\omega) \Big)(H(\omega),v_{\hspace{0.01cm}\Lrm}(\omega))
\end{equation*}
and
\begin{equation*}
(\widetilde{\Kpazo}_{X,u_{\Lrm}}(H,v_{\hspace{0.01cm}\Lrm}))(\omega) = \INTDom{\D_{\mu} v \Big( X_{\sharp}\P,X(\omega),u_{\Lrm}(\omega) \Big)(X(\theta)) H(\theta)}{\Omega}{\P(\theta)}.
\end{equation*}
Under Hypotheses \ref{hyp:H}, the operator $\widetilde{\Kpazo}_{X,u_{\Lrm}}$ is Hilbert-Schmidt, hence compact \cite[Chapter 6]{Brezis}. On the other hand, \eqref{eq:QuantitativeSurjectivityLag} implies that $\widetilde{\Mpazo}_{X,u_{\Lrm}}$ is boundedly invertible, so in particular
\begin{equation*}
\D \widetilde{\Vpazo}(X,u_{\Lrm}) = \widetilde{\Mpazo}_{X,u_{\Lrm}} \Big( \Id + \widetilde{\Mpazo}_{X,u_{\Lrm}}^{-1} \widetilde{\Kpazo}_{X,u_{\Lrm}} \Big),
\end{equation*}
where $\widetilde{\Mpazo}_{X,u_{\Lrm}}^{-1} \widetilde{\Kpazo}_{X,u_{\Lrm}}$ is compact. If we assume that $\D \widetilde{\Vpazo}(X,u_{\Lrm})$ is injective, then $\Id + \widetilde{\Mpazo}_{X,u_{\Lrm}}^{-1} \widetilde{\Kpazo}_{X,u_{\Lrm}}$ is injective as well. It is therefore surjective by the Fredholm alternative on Hilbert spaces \cite[Theorem 6.6]{Brezis}, and so is $\D \widetilde{\Vpazo}(X,u_{\Lrm})$.
\end{proof}

\begin{proof}[Proof of Theorem \ref{thm:KKTLagrangian}]
This is just the standard multiplier rule for equality-constrained optimization in Hilbert spaces. Since $(\bar{X}^s,\bar{u}_{\Lrm}^s)$ is a strong local minimizer and the differential of the constraint map is surjective by Lemma \ref{lem:DifferentialConstraints}, the standard Banach-space KKT theorem applies, see for instance \cite[Chapter 3]{BonnansShapiro}. Hence, there exists $\bar{\Psi}^s \in L^2_{\P}(\Omega,\R^d)$ such that
\begin{equation*}
\nabla \widetilde{\Jpazo}(\bar{X}^s,\bar{u}_{\Lrm}^s) - \D \widetilde{\Vpazo}(\bar{X}^s,\bar{u}_{\Lrm}^s)^{\star}\bar{\Psi}^s = 0.
\end{equation*}
Recalling the definition of the lifted Hamiltonian \eqref{eq:HamiltonianLag}, this identity is equivalent to \eqref{cor:LagrangianKKTHamil}.
\end{proof}

\begin{rmk}[Quantitative surjectivity and observability]
\label{rmk:QuantitativeSurjectivity}
The sufficient condition \eqref{eq:QuantitativeSurjectivityLag} can be interpreted as a pointwise exact controllability estimate for the linearized constraint. In the control-affine case $v(\mu,x,u) = A(\mu,x) + B(\mu,x)u$, which we recall is structurally needed to apply the Eulerian-Lagrangian machinery \cite{Cavagnari2022}, it therefore amounts to imposing a uniform right-inverse estimate on the control matrix $B(\mu,x)$ along the stationary support. This is the finite-dimensional analogue of the observability or Hautus-type conditions that appear in Riccati theory and in turnpike estimates for distributed systems, see e.g. \cite{Tucsnak2009,Trelat2018} or \cite{Zabczyk2020} for the general Hilbertian framework. 
\end{rmk}

%%%%%%%%%%%%%%%%%%%%%%%%%%%%%%%%%%%%%%%%%%%%%%%%%%%%%%%%%%%%%%%%%%%%%%%

\subsection{KKT conditions for the static Eulerian problem}
\label{subsection:EulerianKKT}

We now shift our focus to studying the static version of the Eulerian optimal control problem 
\begin{equation*}
\label{eq:EulerianStatic}
(\Ppazo_{\mathrm{E}}^s) ~~ \left\{ 
\begin{aligned}
\min_{(\mu,u_{\mathrm{E}})} & \INTDom{L \Big(\mu,x,u_{\mathrm{E}}(x) \Big)}{\R^d}{\mu(x)}, \\
\text{s.t.}~ &
v(\mu,u_{\mathrm{E}}) = 0 ~~ \text{in $L^2_{\mu}(\R^d,\R^d)$},
\end{aligned}
\right.
\end{equation*}
where the minimization runs over the set of pairs $(\mu,u_{\mathrm{E}}) \in \Pcal_2(\R^d) \times L^2_{\mu}(\R^d,U)$. Similarly to what we did for $(\Ppazo_{\Lrm}^s)$ in Section \ref{subsection:LagrangianKKT}, we will study the following equivalent reformulation of the problem
\begin{equation*}
\left\{ 
\begin{aligned}
\min_{(\mu,u_{\mathrm{E}})} &~ \Jpazo \Big( (\Id,u_{\mathrm{E}})_{\sharp}\mu \Big) \\
\text{s.t.}~ &
v(\mu,u_{\mathrm{E}}) = 0 ~~ \text{in $L^2_{\mu}(\R^d,\R^d)$}.
\end{aligned}
\right.
\end{equation*}
For reasons that will become apparent below, the latter is expressed in terms of the relaxed cost 
\begin{equation}
\label{eq:RelaxedCost}
\Jpazo(\sigma) := \INTDom{L \big(\pi^1_{\sharp}\sigma,x,u \big)}{\R^d \times U}{\sigma(x,u)} 
\end{equation}
defined for each $\sigma \in \Pcal_2(\R^d \times U)$, and we also consider the admissible set of $(\Ppazo_{\mathrm{E}}^s)$, defined by 
\begin{equation*}
\Adm(\Ppazo_{\mathrm{E}}^s) := \bigg\{ (\mu,u_{\mathrm{E}}) \in \Pcal_2(\R^d) \times L^2_{\mu}(\R^d,\R^d) ~\,\textnormal{s.t.}~ v(\mu,u_{\mathrm{E}}) = 0 ~\, \text{in $L^2_{\mu}(\R^d,\R^d)$} \bigg\}.
\end{equation*}
In line with Lemma \ref{lem:DifferentialConstraints}, we make a small notational abuse and given some $(\mu,u_{\mathrm{E}}) \in \Adm(\Ppazo_{\mathrm{E}}^s)$, we denote by $\D v(\mu,u_{\mathrm{E}}) : L^2_{\mu}(\R^d,\R^d \times U) \to L^2_{\mu}(\R^d,\R^d)$ the operator 
\begin{equation}
\label{eq:IntrinsicEulerianDiff}
\Big( \D v(\mu,u_{\mathrm{E}})(\xi,v_{\mathrm{E}}) \Big)(x) := \D_{(x,u)}v \Big(\mu,x,u_{\mathrm{E}}(x) \Big)(\xi(x),v_{\mathrm{E}}(x)) 
+ \INTDom{\D_{\mu} v \Big( \mu,y,u_{\mathrm{E}}(y) \Big)(x)\xi(y)}{\R^d}{\mu(y)}
\end{equation}
defined for all $(\xi,v_{\mathrm{E}}) \in L^2_{\mu}(\R^d,\R^d \times U)$. As highlighted by the following theorem, and carefully justified within its proof, this convention stems from the fact that $\D v(\mu,u_{\mathrm{E}})$ does play the role of an intrinsic differential for the application $(\mu,u_{\mathrm{E}}) \in \Pcal_2(\R^d) \times L^2_{\mu}(\R^d,U) \mapsto v(\mu,u_{\mathrm{E}}) \in L^2_{\mu}(\R^d,\R^d)$ encoding the constraints of the static Eulerian problem.

\begin{thm}[First-order conditions for the static Eulerian problem]
\label{thm:KKTEulerian}
Suppose that Hypotheses \ref{hyp:H} hold, and let $(\bar{\mu}^s,\bar{u}_{\mathrm{E}}^s) \in \Adm(\Ppazo_{\mathrm{E}}^s)$ be a local $W_2$-minimizer of $(\Ppazo_{\mathrm{E}}^s)$, namely there is an $\epsilon >0$ such that 
\begin{equation*}
\Jpazo \Big( (\Id,\bar{u}_{\mathrm{E}}^s)_{\sharp}\bar{\mu}^s \Big) \leq \Jpazo \Big( (\Id,u_{\mathrm{E}})_{\sharp}\mu \Big)
\end{equation*}
for every $(\mu,u_{\Erm}) \in \Adm(\Ppazo^s_{\Erm})$ with $W_2(\mu,\bar{\mu}^s) \leq \epsilon$. In addition, suppose that the operator $\D v(\bar{\mu}^s,\bar{u}_{\mathrm{E}}^s) : L^2_{\bar{\mu}^s}(\R^d,\R^d \times U) \to L^2_{\bar{\mu}^s}(\R^d,\R^d)$ is injective, and that there exists a constant $c > 0$ for which
\begin{equation}
\label{eq:ClosedRange}
\inf_{x \in \supp(\bar{\mu}^s)} \Big| \, \D_{(x,u)} v \Big(\bar{\mu}^s,x,\bar{u}_{\mathrm{E}}^s(x) \Big)  (y,v) \, \Big|  \geq c \, \| (y,v) \|_{\R^d \times U}
\end{equation}
for all $(y,v) \in \R^d \times U$. Then, there exists a state-costate measure $\bar{\nu}^s \in \Pcal_2(\R^{2d})$ satisfying $\pi^1_{\sharp} \bar{\nu}^s = \bar{\mu}^s$, and such that
\begin{multline}
\label{eq:EulerianKKTStateCostate}
\int_{\R^{2d}} \bigg\langle \nabla_{\sigma} \Jpazo \Big((\Id,\bar{u}^s_{\mathrm{E}})_{\sharp} \bar{\mu}^s \Big)(x,\bar{u}_{\mathrm{E}}^s(x)) - \D_{(x,u)} v \Big(\bar{\mu}^s,x,\bar{u}_{\mathrm{E}}^s(x) \Big)^{\raisebox{-4pt}{\scriptsize$\top$}} p  \\
 - \INTDom{\D_{\mu} v \Big(\bar{\mu}^s,x,\bar{u}_{\mathrm{E}}^s(x) \Big)(y)^{\top} q}{\R^{2d}}{\bar{\nu}^s(y,q)} , \Xi(x,p) \bigg\rangle \, \mathrm{d} \bar{\nu}^s(x,p) = 0
\end{multline}
for every $\Xi \in L^2_{\bar{\nu}^s}(\R^{2d},\R^d \times U)$. Equivalently, there exists a multiplier $\bar{\lambda}^s \in L^2_{\bar{\mu}^s}(\R^d,\R^d)$ such that 
\begin{equation}
\label{eq:EulerianKKTMultiplier}
\nabla_{\sigma} \Jpazo \Big((\Id,\bar{u}^s_{\mathrm{E}})_{\sharp} \bar{\mu}^s \Big)(\Id,\bar{u}_{\mathrm{E}}^s) - \D v(\bar{\mu}^s,\bar{u}_{\mathrm{E}}^s)^{\star} \bar{\lambda}^s = 0
\end{equation}
in $L^2_{\bar{\mu}^s}(\R^d,\R^d \times U)$. 
\end{thm}

\begin{rmk}[Concerning the Eulerian KKT theorem]
\label{rmk:EulerianKKT}
Note that Theorem \ref{thm:KKTEulerian} handles genuinely infinite-dimensional and highly nonlinear constraints, since the map $(\mu,u_{\Erm}) \mapsto v(\mu,u_{\Erm})$ is valued in the whole space $L^2_{\mu}(\R^d,\R^d)$, which itself depends on $\mu \in \Pcal_2(\R^d)$. This is precisely why the proof proceeds using the lifted Hilbert formulation of Section \ref{subsection:LagrangianKKT}, whose role is to disentangle the control and space variables. Another possible approach would have been to work at the level of occupation measures, but we opted for the methodology Lagrangian-to-Eulerian due to its direct connection with existing results at the level of dynamical problems. %Another important point is that the local minimality assumption is only imposed with respect to the state measure, and not with respect to a direct topology on Eulerian controls. This asymmetry is unavoidable in general, since as we shall see below, the barycentric projection which allows to build a candidate Eulerian control from a Lagrangian one is not continuous in any natural norm topology. A stronger and more symmetric notion of local minimality could be formulated directly at the level of occupation measures $(\Id,u_{\Erm})_{\sharp}\mu \in \Pcal_2(\R^d \times U)$, but the present argument does not require it.
\end{rmk}

The proof of Theorem \ref{thm:KKTEulerian} hinges upon the core relationships between the Lagrangian and Eulerian static problems exposed in the following lemma, which are directly inspired by \cite{Averboukh2025,Cavagnari2022}.

\begin{lem}[Transformations between Lagrangian and Eulerian problems]
\label{lem:LagrangianEulerian}
Under Hypotheses \ref{hyp:H}, the following statements hold. 
\begin{enumerate}
\item[$(a)$] For every $(\mu,u_{\mathrm{E}}) \in \Adm(\Ppazo_{\Erm}^s)$, there exists $(X,u_{\Lrm}) \in \Adm(\Ppazo_{\Lrm}^s)$ such that 
$X_{\sharp} \P = \mu$ and $u_{\Lrm} = u_{\mathrm{E}} \circ X$.
In particular, it holds that $\Jpazo( (\Id,u_{\mathrm{E}})_{\sharp}\mu) = \widetilde{\Jpazo}(X,u_{\Lrm})$.
\item[$(b)$] For every $(X,u_{\Lrm}) \in \Adm(\Ppazo_{\Lrm}^s)$, there exists $(\mu,u_{\mathrm{E}}) \in \Adm(\Ppazo_{\mathrm{E}}^s)$ such that 
$\Jpazo \big( (\Id,u_{\mathrm{E}})_{\sharp}\mu \big) \leq \widetilde{\Jpazo}(X,u_{\Lrm})$.
\end{enumerate}
Consequently, every local $W_2$-minimizer of $(\Ppazo_{\mathrm{E}}^s)$ is realized by a strong local minimizer of $(\Ppazo_{\Lrm}^s)$. 
\end{lem}

\begin{proof}
We start by establishing item $(a)$. From the generalization of Skorokhod's theorem proven in \cite{Berti2007} and recollected in \cite[Proposition 2.1]{Cavagnari2022}, there exists for each $(\mu,u_{\mathrm{E}}) \in \Adm(\Ppazo_{\mathrm{E}}^s)$ a map $X \in L^2_{\P}(\Omega,\R^d)$ such that $X_{\sharp} \P = \mu$. Setting now $u_{\Lrm} := u_{\mathrm{E}} \circ X$, it can be easily checked that $u_{\Lrm} \in L^2_{\P}(\Omega,U)$, and 
\begin{equation*}
\begin{aligned}
\P \Big(\Big\{ \omega \in \Omega ~\textnormal{s.t.}~ v(X_{\sharp}\P,X(\omega),u_{\Lrm}(\omega)) = 0 \Big\} \Big) & = \P \Big(\Big\{ \omega \in \Omega ~\textnormal{s.t.}~ v(\mu,X(\omega),u_{\mathrm{E}}(X(\omega))) = 0 \Big\} \Big) \\
& = \P\Big( X^{-1} \Big( \Big\{ x \in \R^d ~\,\text{s.t.}~ v(\mu,x,u_{\mathrm{E}}(x)) = 0 \Big\} \Big) \Big) = 1
\end{aligned}
\end{equation*}
since $v(\mu,u_{\mathrm{E}}) = 0$ in $L^2_{\mu}(\R^d,\R^d)$ by assumption. Thus, we have shown that $(X,u_{\Lrm}) \in \Adm(\Ppazo_{\Lrm}^s)$. It may then be straightforwardly verified that $\Jpazo( (\Id,u_{\mathrm{E}})_{\sharp}\mu) = \widetilde{\Jpazo}(X,u_{\Lrm})$, which settles item $(a)$.  

We focus now on item $(b)$, and start given any $(X,u_{\Lrm}) \in \Adm(\Ppazo_{\Lrm}^s)$ by defining $\mu := X_{\sharp} \P \in \Pcal_2(\R^d)$. Then, by the disintegration theorem (see e.g. \cite[Theorem 5.3.1]{AGS}), there exists a $\mu$-almost uniquely determined Borel family of measures $\{\P_x\}_{x \in \R^d} \subset \Pcal(\Omega)$ such that  
\begin{equation*}
\INTDom{H(\omega)}{\Omega}{\P(\omega)} = \INTDom{\INTDom{H(\omega)}{\Omega_x}{\P_x(\omega)}}{\R^d}{\mu(x)}
\end{equation*}
for each bounded Borel map $H : \Omega \to \R$, where $\Omega_x := X^{-1}(\{x\}) \subset \Omega$. At this stage, we set
\begin{equation*}
u_{\mathrm{E}}(x) := \INTDom{u_{\Lrm}(\omega)}{\Omega_x}{\P_x(\omega)}, 
\end{equation*}
and note that the latter is Borel by standard measurability arguments, valued in $U$ by convexity of the integral (see e.g. \cite[Proposition 1.2.12]{AnalysisBanachSpaces}), and such that  
\begin{equation*}
\begin{aligned}
\NormLbis{u_{\mathrm{E}}}{\R^d,U}{\mu}^2 & = \INTDom{\bigg|\INTDom{u_{\Lrm}(\omega)}{\Omega_x}{\P_x(\omega)} \, \bigg|^2}{\R^d}{\mu(x)} \\
& \leq \INTDom{\INTDom{|u_{\Lrm}(\omega)|^2}{\Omega_x}{\P_x(\omega)}}{\R^d}{\mu(x)} 
= \NormLbis{u_{\Lrm}}{\Omega,\R^d}{\P} < +\infty
\end{aligned}
\end{equation*}
by Jensen's inequality. Concerning the admissibility of the pair $(\mu,u_{\mathrm{E}}) \in \Pcal_2(\R^d) \times L^2_{\mu}(\R^d,U)$ for $(\Ppazo_{\mathrm{E}}^s)$, note that for each $\phi \in C^0_c(\R^d,\R^d)$, it holds that
\begin{equation*}
\begin{aligned}
\INTDom{\Big\langle v(\mu,x,u_{\mathrm{E}}(x)) , \phi(x) \Big\rangle}{\R^d}{\mu(x)} & = \INTDom{\bigg\langle v \bigg( \mu,x,\INTDom{u_{\Lrm}(\omega)}{\Omega_x}{\P_x(\omega)} \bigg) , \phi(x) \bigg\rangle}{\R^d}{\mu(x)} \\
& = \INTDom{\INTDom{\Big\langle v(\mu,x,u_{\Lrm}(\omega)) , \phi(x) \Big\rangle}{\Omega_x}{\P_x(\omega)}}{\R^d}{\mu(x)} \\
& = \INTDom{\INTDom{\Big\langle v(X_{\sharp} \P,X(\omega),u_{\Lrm}(\omega)) , \phi(X(\omega)) \Big\rangle}{\Omega_x}{\P_x(\omega)}}{\R^d}{\mu(x)} \\
& = \INTDom{\Big\langle v(X_{\sharp} \P,X(\omega),u_{\Lrm}(\omega)) , \phi(X(\omega)) \Big\rangle}{\Omega}{\P(\omega)} = 0,
\end{aligned}
\end{equation*}
by the affinity assumption posited in Hypothesis \ref{hyp:H}-$(iii)$. Up to a standard density argument, this directly implies that $v(\mu,u_{\mathrm{E}}) = 0$ in $L^2_{\mu}(\R^d,\R^d)$. Whence, we have shown that $(\mu,u_{\mathrm{E}}) \in \Adm(\Ppazo_{\mathrm{E}}^s)$. Regarding the cost inequality, the latter follows simply from the convexity posited in Hypotheses \ref{hyp:H}-$(iv)$, since
\begin{equation*}
\begin{aligned}
\Jpazo(X,u_{\Lrm})  & = \INTDom{L \Big( X_{\sharp} \P,X(\omega),u_{\Lrm}(\omega) \Big)}{\Omega}{\P(\omega)} \\
& = \INTDom{\INTDom{L \Big( X_{\sharp} \P,X(\omega),u_{\Lrm}(\omega) \Big)}{\Omega_x}{\P_x(\omega)}}{\R^d}{\mu(x)} \\
& \geq \INTDom{L \bigg( \mu ,x ,\INTDom{u_{\Lrm}(\omega)}{\Omega_x}{\P_x(\omega)} \bigg)}{\R^d}{\mu(x)} = \Jpazo\Big( (\Id,u_{\mathrm{E}})_{\sharp} \mu \Big)
\end{aligned}
\end{equation*}
by Jensen's inequality, which achieves the proof of item $(b)$. 

To conclude the proof, let $(\bar{\mu}^s,\bar{u}_{\mathrm{E}}^s) \in \Adm(\Ppazo_{\mathrm{E}}^s)$ be a local $W_2$-minimizer for $(\Ppazo_{\mathrm{E}}^s)$, which we recall means that there is some $\epsilon >0$ such that  
$\Jpazo \big( (\Id,\bar{u}_{\mathrm{E}}^s)_{\sharp}\bar{\mu}^s \big) \leq \Jpazo\big((\Id,u_{\mathrm{E}})_{\sharp}\mu \big)$
for every $(\mu,u_{\mathrm{E}}) \in \Adm(\Ppazo_{\mathrm{E}}^s)$ satisfying $W_2(\bar{\mu}^s,\mu) \leq \epsilon$. Denote then by $(\bar{X}^s,\bar{u}^s_{\Lrm}) \in \Adm(\Ppazo_{\Lrm}^s)$ one of the Lagrangian pairs given by item $(a)$, and consider any $(X,u_{\Lrm}) \in \Adm(\Ppazo_{\Lrm}^s)$ satisfying $\NormLbis{(X,u_{\Lrm}) -(\bar{X}^s,\bar{u}_{\Lrm}^s)}{\Omega,\R^d}{\P} \leq \epsilon$. Letting then $(\mu,u_{\mathrm{E}}) \in \Adm(\Ppazo_{\mathrm{E}}^s)$ be the Eulerian pair built from the latter in item $(b)$, it follows by construction that $W_2(\bar{\mu}^s,\mu) \leq \NormLbis{X -\bar{X}^s}{\Omega,\R^d}{\P} \leq \epsilon$, which in turn yields
\begin{equation*}
\widetilde{\Jpazo}(X,u_{\Lrm}) \geq \Jpazo \Big( (\Id,u_{\mathrm{E}})_{\sharp} \mu \Big) 
 \geq \Jpazo \Big( (\Id,\bar{u}_{\mathrm{E}}^s)_{\sharp}\bar{\mu}^s \Big) 
 = \widetilde{\Jpazo}(\bar{X}^s,\bar{u}^s_{\Lrm}). 
\end{equation*}
In conclusion, every Lagrangian pair $(\bar{X}^s,\bar{u}^s_{\Lrm}) \in \Adm(\Ppazo_{\Lrm}^s)$ realizing a local $W_2$-minimizer $(\bar{\mu}^s,\bar{u}_{\mathrm{E}}^s) \in \Adm(\Ppazo_{\mathrm{E}}^s)$ is itself a strong local minimizer.  
\end{proof}

\begin{rmk}[$W_2$-local minimizers versus occupation measures]
\label{rmk:MuLocalminimizers}
The previous lemma shows that local minimality with respect to the sole measure variable is sufficient for the static Eulerian-to-Lagrangian transfer used in this paper, see also \cite{Averboukh2025}. This choice is deliberate, as in general, there is no useful continuity estimate allowing one to control the distance between the Eulerian controls reconstructed by disintegration from the distance between two neighboring Lagrangian controls. If one wishes to keep track of both the measure and the control variable in a single topology, a natural alternative is to work at the level of occupation measures $(\Id,u_{\Erm})_{\sharp}\mu \in \Pcal_2(\R^d \times U)$. We shall not pursue this relaxation viewpoint here, since the $W_2$-local notion is sufficient for the KKT analysis and the turnpike argument.
\end{rmk}

These preliminary relationships being laid out, we are ready to attack the proof of Theorem \ref{thm:KKTEulerian}. 

\begin{proof}[Proof of Theorem \ref{thm:KKTEulerian}]
Given a $W_2$-local minimizer $(\bar{\mu}^s,\bar{u}_{\mathrm{E}}^s) \in \Adm(\Ppazo_{\mathrm{E}}^s)$ for the static Eulerian problem, there exists by Lemma \ref{lem:LagrangianEulerian} a strong local minimizer $(\bar{X}^s,\bar{u}_{\Lrm}^s) \in \Adm(\Ppazo^s_{\Lrm})$ such that 
\begin{equation}
\label{eq:EulerProofReal}
\bar{X}^s_{\sharp} \P = \bar{\mu}^s \qquad \text{and} \qquad \bar{u}_{\Lrm}^s = \bar{u}_{\mathrm{E}}^s \circ \bar{X}^s.  
\end{equation}
Besides, we have shown in Lemma \ref{lem:DifferentialConstraints} that 
\begin{equation}
\label{eq:EulerianProofDiff}
\begin{aligned}
\Big( \D \widetilde{\Vpazo}(\bar{X}^s,\bar{u}_{\Lrm}^s) (H,v_{\hspace{0.01cm}\Lrm}) \Big)(\omega) & = \D_{(x,u)} v \Big( \bar{X}^s_{\sharp} \P ,\bar{X}^s(\omega),\bar{u}_{\mathrm{E}}^s(\bar{X}^s(\omega)) \Big)(H(\omega),v_{\hspace{0.01cm}\Lrm}(\omega)) \\
& \hspace{0.45cm} + \INTDom{\D_{\mu} v \Big( \bar{X}^s_{\sharp} \P ,\bar{X}^s(\omega),\bar{u}_{\mathrm{E}}^s(\bar{X}^s(\omega))\Big)(X(\theta)) H(\theta)}{\Omega}{\P(\theta)}
\end{aligned}
\end{equation}
for every $(H,v_{\hspace{0.01cm}\Lrm}) \in L^2_{\P}(\Omega,\R^d \times U)$. We start by proving that the operator $\D \widetilde{\Vpazo}(\bar{X}^s,\bar{u}^s_{\Lrm}) : L^2_{\P}(\Omega,\R^d \times U) \to L^2_{\P}(\Omega,\R^d)$ is surjective under our working assumptions. To this end, observe at first that
\begin{equation}
\label{eq:EulerianProofSigma}
\bar{\sigma}^s := (\bar{X}^s,\bar{u}_{\Lrm}^s)_{\sharp} \P = (\Id,\bar{u}_{\mathrm{E}}^s)_{\sharp} \bar{\mu}^s
\end{equation}
by construction, which implies in particular 
\begin{multline*}
\inf_{\omega \in \Omega} \Big| \, \D_{(x,u)} v \Big( \bar{X}^s_{\sharp} \P,\bar{X}^s(\omega),\bar{u}^s_{\Lrm}(\omega) \Big)(y,v) \, \Big| = \inf_{(x,u) \in \supp(\bar{\sigma}^s)} \Big| \, \D_{(x,u)} v( \bar{\mu}^s,x,u)  (y,v) \, \Big| \\
=\inf_{x \in \supp(\bar{\mu}^s)} \Big| \, \D_{(x,u)} v \Big( \bar{\mu}^s,x,\bar{u}_{\mathrm{E}}^s(x) \Big)  (y,v) \, \Big| 
\geq c \, \|(y,v)\|_{\R^d \times U}
\end{multline*}
where the last line follows from \eqref{eq:ClosedRange}. Concerning the injectivity, suppose by contradiction that there exists $(H,v_{\hspace{0.01cm}\Lrm}) \in \mathrm{Ker}(\D \widetilde{\Vpazo}(\bar{X}^s,\bar{u}^s_{\Lrm})) \setminus \{0\}$ and let  $(\xi,v_{\mathrm{E}}) \in L^2_{\bar{\mu}^s}(\R^d,\R^d \times U)$ be given by 
\begin{equation*}
(\xi(x),v_{\mathrm{E}}(x)) := \INTDom{(H(\omega),v_{\hspace{0.01cm}\Lrm}(\omega))}{\Omega_x}{\P_x(\omega)}
\end{equation*}
for $\bar{\mu}^s$-almost every $x \in \R^d$, where $\{\Omega_x\}_{x \in \R^d}$ and $\{\P_x\}_{x \in \R^d} \subset \Pcal(\Omega)$ are defined by disintegration as in the proof of Lemma \ref{lem:LagrangianEulerian} above. Then, for every Borel map $\zeta \in L^2_{\bar{\mu}^s}(\R^d,\R^d)$, it can be checked that 
\begin{equation*}
\begin{aligned}
& \Big\langle \D v(\bar{\mu}^s,\bar{u}_{\mathrm{E}}^s) (\xi,v_{\mathrm{E}}) , \zeta \Big\rangle_{L^2_{\bar{\mu}^s}(\R^d,\R^d)} \\
& =  \INTDom{\Big\langle \D_{(x,u)}v \Big(\bar{\mu}^s,x,\bar{u}^s_{\mathrm{E}}(x) \Big)(\xi(x),v_{\mathrm{E}}(x)) , \zeta(x) \Big\rangle}{\R^d}{\bar{\mu}^s(x)} \\
& \hspace{0.5cm} + \INTDom{\INTDom{\Big\langle\D_{\mu} v \Big(\bar{\mu}^s,x,\bar{u}^s_{\mathrm{E}}(x) \Big)(y) \xi(y) , \zeta(x) \Big\rangle}{\R^d}{\bar{\mu}^s(y)}}{\R^d}{\bar{\mu}^s(x)}  \\
& = \INTDom{\INTDom{\Big\langle \D_{(x,u)} v \Big( \bar{X}^s_{\sharp} \P,\bar{X}^s(\omega),\bar{u}^s_{\mathrm{E}} \circ \bar{X}^s(\omega) \Big) (H(\omega),v_{\hspace{0.01cm}\Lrm}(\omega)) , \zeta \circ X(\omega) \Big\rangle}{\Omega_x}{\P_x(\omega)}}{\R^d}{\bar{\mu}^s(x)} \\
& \hspace{0.5cm} + \INTDom{\Bigg( \INTDom{\INTDom{\Big\langle\D_{\mu} v \Big( \bar{X}^s_{\sharp} \P,\bar{X}^s(\omega),\bar{u}^s_{\mathrm{E}} \circ \bar{X}^s(\omega) \Big) (X(\theta))H(\theta) , \zeta \circ X(\omega) \Big\rangle}{\Omega_y}{\P_y(\theta)}}{\R^d}{\bar{\mu}^s(y)} \Bigg)}{\Omega}{\P(\omega)} \\
& = \Big\langle \D \widetilde{\Vpazo}(\bar{X}^s,\bar{u}_{\Lrm}^s) (H,v_{\hspace{0.01cm}\Lrm}) , \zeta \circ \bar{X}^s \Big\rangle_{L^2_{\P}(\Omega,\R^d)} 
 = 0
\end{aligned}
\end{equation*}
which contradicts the injectivity of $\D v(\bar{\mu}^s,\bar{u}_{\mathrm{E}}^s) : L^2_{\bar{\mu}}(\R^d,\R^d \times U) \to L^2_{\bar{\mu}}(\R^d,\R^d)$. Thus, both conditions entail the surjectivity of $\D \widetilde{\Vpazo}(\bar{X}^s,\bar{u}^s_{\Lrm}) : L^2_{\P}(\Omega,\R^d \times U) \to L^2_{\P}(\Omega,\R^d)$ by Lemma \ref{lem:DifferentialConstraints} above, and it follows from Theorem \ref{thm:KKTLagrangian} that there exists a multiplier $\bar{\Psi}^s \in L^2_{\P}(\Omega,\R^d)$ for which
\begin{equation}
\label{eq:EulerianProofLag}
\nabla \widetilde{\Jpazo}(\bar{X}^s,\bar{u}_{\Lrm}^s) - \D \widetilde{\Vpazo}(\bar{X}^s,\bar{u}_{\Lrm}^s)^{\star}\bar{\Psi}^s = 0
\end{equation}
in $L^2_{\P}(\Omega,\R^d \times U)$. Note that this may be equivalently written as
\begin{equation}
\label{eq:EulerianProofTrans0}
\INTDom{\bigg\langle \nabla \widetilde{\Jpazo}(\bar{X}^s,\bar{u}_{\Lrm}^s)(\omega) - \Big( \D \widetilde{\Vpazo}(\bar{X}^s,\bar{u}_{\Lrm}^s)^{\star}\bar{\Psi}^s \Big)(\omega) , (H(\omega),v_{\hspace{0.01cm}\Lrm}(\omega)) \bigg\rangle}{\Omega}{\P(\omega)} = 0
\end{equation}
for every $(H,v_{\hspace{0.01cm}\Lrm}) \in L^2_{\P}(\Omega,\R^d \times U)$.

To transfer the Lagrangian stationarity condition \eqref{eq:EulerianProofLag} back to the Eulerian framework, we introduce the state-costate measure defined by $\bar{\nu}^s := (\bar{X}^s,\bar{\Psi}^s)_{\sharp} \P \in \Pcal_2(\R^{2d})$, and notice that 
\begin{equation}
\label{eq:EulerianProofTrans1}
\nabla \widetilde{\Jpazo}(\bar{X}^s,\bar{u}_{\Lrm}^s) = \nabla_{\sigma} \Jpazo(\bar{\sigma}^s) \circ (\bar{X}^s,\bar{u}_{\Lrm}^s) 
= \nabla_{\sigma} \Jpazo \Big( (\Id,\bar{u}_{\mathrm{E}}^s)_{\sharp} \bar{\mu}^s \Big) \circ (\Id,\bar{u}_{\mathrm{E}}^s) \circ \bar{X}^s
\end{equation}
owing to the representation result of Proposition \ref{prop:WassGrad} combined with \eqref{eq:EulerProofReal} and \eqref{eq:EulerianProofSigma}. Likewise
\begin{equation}
\label{eq:EulerianProofTrans2}
\begin{aligned}
& \Big( \D \widetilde{\Vpazo}(\bar{X}^s,\bar{u}_{\Lrm}^s)^{\star}\bar{\Psi}^s \Big)(\omega) \\
& \hspace{0.45cm} = \D_{(x,u)} v \Big( \bar{\mu}^s , \bar{X}^s(\omega),\bar{u}_{\Lrm}^s(\omega) \Big)^{\raisebox{-4pt}{\scriptsize$\top$}} \bar{\Psi}^s(\omega) 
 + \INTDom{\D_{\mu} v \Big( \bar{\mu}^s , \bar{X}^s(\theta),\bar{u}_{\Lrm}^s(\theta) \Big)(\bar{X}^s(\omega))^{\top} \bar{\Psi}^s(\theta)}{\Omega}{\P(\theta)} \\
& \hspace{0.45cm} = \D_{(x,u)} v \Big( \bar{\mu}^s , \bar{X}^s(\omega),\bar{u}_{\mathrm{E}}^s \circ \bar{X}^s(\omega) \Big)^{\raisebox{-4pt}{\scriptsize$\top$}} \bar{\Psi}^s(\omega) 
 + \INTDom{\D_{\mu} v \Big( \bar{\mu}^s , \bar{X}^s(\theta),\bar{u}_{\mathrm{E}}^s \circ \bar{X}^s(\theta) \Big)(\bar{X}^s(\omega))^{\top} \bar{\Psi}^s(\theta)}{\Omega}{\P(\theta)}
\end{aligned}
\end{equation}
thanks to the expression provided in \eqref{eq:EulerianProofDiff}. At this stage, it follows from taking particular couples of the form $(H,v_{\hspace{0.01cm}\Lrm}) := \Xi \circ (\bar{X}^s,\bar{\Psi}^s) \in L^2_{\P}(\Omega,\R^d \times U)$ with $\Xi \in C^0_b(\R^{2d},\R^d \times U)$ while leveraging the information from \eqref{eq:EulerianProofTrans1} that
\begin{equation}
\label{eq:EulerianProofTrans3}
\begin{aligned}
& \INTDom{\Big\langle \nabla \widetilde{\Jpazo}(\bar{X}^s,\bar{u}_{\Lrm}^s)(\omega) , (H(\omega),v_{\hspace{0.01cm}\Lrm}(\omega)) \Big\rangle}{\Omega}{\P(\omega)} \\
& = \INTDom{\bigg\langle \nabla_{\sigma} \Jpazo \Big( (\Id,\bar{u}_{\mathrm{E}}^s)_{\sharp} \bar{\mu}^s \Big) \Big(\bar{X}^s(\omega),\bar{u}_{\mathrm{E}}^s \circ \bar{X}^s(\omega)\Big) \, , \, \Xi \circ \Big(\bar{X}^s(\omega),\bar{\Psi}^s(\omega) \Big) \bigg\rangle}{\Omega}{\P(\omega)} \\
& = \INTDom{\Big\langle \nabla_{\sigma} \Jpazo \Big( (\Id,\bar{u}_{\mathrm{E}}^s)_{\sharp} \bar{\mu}^s \Big)(x,\bar{u}_{\mathrm{E}}^s(x)) , \Xi(x,p) \Big\rangle}{\R^{2d}}{\bar{\nu}^s(x,p)}  
\end{aligned}
\end{equation}
and similarly using \eqref{eq:EulerianProofTrans2}, %together with the barycentric projection $\bar{\lambda}^s \in L^2_{\bar{\mu}^s}(\R^d,\R^d)$ introduced below
one has that
\begin{equation}
\label{eq:EulerianProofTrans4}
\begin{aligned}
& \INTDom{\Big\langle  \Big( \D \widetilde{\Vpazo}(\bar{X}^s,\bar{u}_{\Lrm}^s)^{\star}\bar{\Psi}^s \Big)(\omega) , (H(\omega),v_{\hspace{0.01cm}\Lrm}(\omega)) \Big\rangle}{\Omega}{\P(\omega)} \\
& = \INTDom{\bigg\langle \D_{(x,u)} v \Big( \bar{\mu}^s , \bar{X}^s(\omega),\bar{u}_{\mathrm{E}}^s \circ \bar{X}^s(\omega) \Big)^{\hspace{-0.1cm} \top} \bar{\Psi}^s(\omega) \, , \, \Xi \circ \Big(\bar{X}^s(\omega),\bar{\Psi}^s(\omega) \Big) \bigg\rangle}{\Omega}{\P(\omega)} \\
& \hspace{0.45cm} + \INTDom{\INTDom{\bigg\langle \D_{\mu} v \Big( \bar{\mu}^s , \bar{X}^s(\theta),\bar{u}_{\mathrm{E}}^s \circ \bar{X}^s(\theta) \Big)(\bar{X}^s(\omega))^{\top} \bar{\Psi}^s(\theta) \, , \, \Xi \circ \Big(\bar{X}^s(\omega),\bar{\Psi}^s(\omega) \Big) \bigg\rangle}{\Omega}{\P(\theta)}}{\Omega}{\P(\omega)} \\
& = \INTDom{\bigg\langle \D_{(x,u)} v \Big( \bar{\mu}^s , x,\bar{u}_{\mathrm{E}}^s(x) \Big)^{\hspace{-0.1cm} \top} p + \INTDom{\D_{\mu} v \Big( \bar{\mu}^s , y,\bar{u}_{\mathrm{E}}^s(y) \Big)(x)^{\top} q}{\R^{2d}}{\bar{\nu}^s(y,q)} \, , \, \Xi(x,p) \bigg\rangle}{\R^{2d}}{\bar{\nu}^s(x,p)}. 
\end{aligned}
\end{equation}
Merging together \eqref{eq:EulerianProofTrans3} and \eqref{eq:EulerianProofTrans4} in \eqref{eq:EulerianProofTrans0}, we recover the distributional identity \eqref{eq:EulerianKKTStateCostate}. 

To obtain the multiplier version \eqref{eq:EulerianKKTMultiplier} of the Eulerian KKT conditions, we let $\{\bar{\nu}^s_x\}_{x \in \R^d} \subset \Pcal_2(\R^d)$ be the disintegration of $\bar{\nu}^s \in \Pcal_2(\R^{2d})$ onto its first marginal, and 
$\bar{\lambda}^s(x) := \INTDom{p}{\R^d}{\bar{\nu}^s_x(p)}$
be its barycentric projection \cite[Definition 5.4.2]{AGS}, that is defined for $\bar{\mu}^s$-almost every $x \in \R^d$. It can be verified that $\bar{\lambda}^s \in L^2_{\bar{\mu}^s}(\R^d,\R^d)$, and recalling the definition \eqref{eq:IntrinsicEulerianDiff} of the intrinsic differential while choosing test functions of the form $\Xi(x,p) := \xi(x)$ with $\xi \in C^0_b(\R^d,\R^d \times U)$ in \eqref{eq:EulerianKKTStateCostate}, we finally get that 
\begin{equation*}
\INTDom{\bigg\langle \nabla_{\sigma} \Jpazo \Big( (\Id,\bar{u}_{\mathrm{E}}^s)_{\sharp} \bar{\mu}^s \Big)(x,\bar{u}_{\mathrm{E}}^s(x)) - \Big(\D v (\bar{\mu}^s,\bar{u}_{\mathrm{E}}^s)^{\star} \bar{\lambda}^s \Big)(x) , \xi(x) \bigg\rangle}{\R^d}{\bar{\mu}^s(x)} = 0,
\end{equation*}
which amounts to \eqref{eq:EulerianKKTMultiplier} since $\xi \in C^0_b(\R^d,\R^d \times U)$ is arbitrary. 
\end{proof}

The proof of our Eulerian turnpike theorem discussed in Section \ref{section:Turnpike} extensively relies on the Hamiltonian reformulation of the KKT conditions of Theorem \ref{thm:KKTEulerian}. To this end, we introduce the mapping 
\begin{equation*}%\label{eq:HamtiltonianExtended}
\Hcal(\Bnu) := \INTDom{H\big(\pi^1_{\sharp} \Bnu, x,p,u \big)}{\R^{2d} \times U}{\Bnu(x,p,u)}. 
\end{equation*}
defined for every $\Bnu \in \Pcal_2(\R^{2d} \times U)$. Under the working assumptions listed in Hypotheses \ref{hyp:H}, it can be shown via arguments similar to those detailed in Appendix \ref{section:AppendixHess} below that $ \Hcal : \Pcal_2(\R^{2d} \times U) \to \R$ is twice continuously Fréchet differentiable in the sense of Definition \ref{def:WassHess}, with bounded derivatives. 

\begin{cor}[Hamiltonian version of the Eulerian KKT conditions]
\label{cor:EulerianKKTHamil}
Let $(\bar{\mu}^s,\bar{u}_{\mathrm{E}}^s) \in \Adm(\Ppazo_{\mathrm{E}}^s)$ and suppose that $\bar{\nu}^s \in \Pcal_2(\R^{2d})$ is such that the stationarity condition \eqref{eq:EulerianKKTStateCostate} from Theorem \ref{thm:KKTEulerian} holds. Then, the measure $\bar{\Bnu}^s := (\pi^1,\pi^2,\bar{u}_{\mathrm{E}}^s \circ \pi^1)_{\sharp} \bar{\nu}^s \in \Pcal_2(\R^{2d} \times U)$ satisfies
$\nabla_{\Bnu} \Hcal(\bar{\Bnu}^s) = 0$.
\end{cor}

\begin{proof}
By a slight modification of the results in \cite[Section 5]{SetValuedPMP}, it can be shown that 
\begin{equation}
\label{eq:HamiltonianKKT0}
\nabla_{\Bnu} \Hcal(\bar{\Bnu}^s)(x,p,u) = 
\begin{pmatrix}
\nabla_x H(\bar{\mu}^s,x,p,u) + \INTDom{\nabla_{\mu} H(\bar{\mu}^s,y,q,v)(x)}{\R^d}{\bar{\Bnu}^s(y,q,v)} \\
\nabla_p H(\bar{\mu}^s,x,p,u) \\
\nabla_u H(\bar{\mu}^s,x,p,u)
\end{pmatrix}. 
\end{equation}
Recalling the definition \eqref{eq:HamiltonianClassical} of the local Hamiltonian $H : \Pcal_2(\R^d) \times \R^{2d} \times U \to \R$, we straightforwardly get that 
\begin{equation}
\label{eq:HamiltonianKKT1}
\begin{pmatrix}
\nabla_x H(\bar{\mu}^s,x,p,u) \\
\nabla_p H(\bar{\mu}^s,x,p,u) \\
\nabla_u H(\bar{\mu}^s,x,p,u)
\end{pmatrix}
=  
\begin{pmatrix}
\D_{(x,u)} v(\bar{\mu}^s,x,u)^{\top}p - \nabla_{(x,u)} L(\bar{\mu}^s,x,u) \\
v(\bar{\mu}^s,x,u) 
\end{pmatrix}
\end{equation}
up to a row permutation, and likewise
\begin{equation}
\label{eq:HamiltonianKKT2}
\begin{aligned}
\INTDom{\nabla_{\mu} H(\bar{\mu}^s,y,q,v)(x)}{\R^{2d} \times U}{\bar{\Bnu}^s(y,q,v)} & = \INTDom{\D_{\mu} v \Big(\bar{\mu}^s,y,\bar{u}_{\mathrm{E}}^s(y) \Big)(x)^{\top} q}{\R^{2d}}{\bar{\nu}^s(y,q)} \\
& \hspace{1.25cm} - \INTDom{\nabla_{\mu} L\Big(\bar{\mu}^s,y,\bar{u}_{\mathrm{E}}^s(y)\Big)(x)}{\R^d}{\bar{\mu}^s(y)}. 
\end{aligned}
\end{equation}
Recalling now the definition of the relaxed cost function $\Jpazo : \Pcal_2(\R^d \times U) \to \R$ given in \eqref{eq:RelaxedCost}, it can be checked by the arguments invoked above and excerpted from \cite[Section 5]{SetValuedPMP} that 
\begin{equation}
\label{eq:HamiltonianKKT3}
\nabla_{\sigma} \Jpazo \Big( (\Id,\bar{u}_{\Erm}^s)_{\sharp} \bar{\mu}^s \Big)(x,\bar{u}^s_{\Erm}(x)) = \nabla_{(x,u)} L \Big( \bar{\mu}^s,x,\bar{u}^s_{\Erm}(x) \Big) + \INTDom{\nabla_{\mu} L \Big( \bar{\mu}^s,y,\bar{u}^s_{\Erm}(y) \Big)(x)}{\R^d}{\bar{\mu}^s(y)}
\end{equation}
for $\bar{\mu}^s$-almost every $x \in \R^d$. Thus, combining \eqref{eq:HamiltonianKKT0}, \eqref{eq:HamiltonianKKT1}, \eqref{eq:HamiltonianKKT2} and \eqref{eq:HamiltonianKKT3}, we infer that for every $\Xi := (\xi,\zeta,v_{\Erm}) \in C^0_b(\R^{2d} \times U,\R^{2d} \times U)$, there holds
\begin{equation*}
\begin{aligned}
& \INTDom{\Big\langle \nabla_{\Bnu} \Hcal(\bar{\Bnu}^s) (x,p,u) , \Xi(x,p,u) \Big\rangle}{\R^{2d} \times U}{\bar{\Bnu}^s(x,p,u)} \\
& = \INTDom{\Big\langle v \Big(\bar{\mu}^s,x,\bar{u}_{\Erm}^s(x) \Big) , \zeta(x,p,\bar{u}_{\Erm}^s(x)) \Big\rangle}{\R^{2d}}{\bar{\nu}^s(x,p)} \\
& \hspace{0.45cm} + \int_{\R^{2d}} \bigg\langle \nabla_{\sigma} \Jpazo \Big((\Id,\bar{u}^s_{\mathrm{E}})_{\sharp} \bar{\mu}^s \Big)(x,\bar{u}_{\mathrm{E}}^s(x)) - \D_{(x,u)} v \Big(\bar{\mu}^s,x,\bar{u}_{\mathrm{E}}^s(x) \Big)^{\raisebox{-4pt}{\scriptsize$\top$}} p \\
& \hspace{5cm} - \INTDom{\D_{\mu} v \Big(\bar{\mu}^s,y,\bar{u}_{\mathrm{E}}^s(y) \Big)(x)^{\top} q}{\R^{2d}}{\bar{\nu}^s(y,q)} , (\xi,v_{\Erm})(x,p,\bar{u}_{\Erm}^s(x)) \bigg\rangle \, \mathrm{d} \bar{\nu}^s(x,p) \\
& = 0 
\end{aligned}
\end{equation*}
since $(\bar{\mu}^s,\bar{u}_{\mathrm{E}}^s) \in \Adm(\Ppazo_{\mathrm{E}}^s)$ and $\bar{\nu}^s \in \Pcal_2(\R^{2d})$ satisfies \eqref{eq:EulerianKKTMultiplier}, and the conclusion follows by density. 
\end{proof}

%%%%%%%%%%%%%%%%%%%%%%%%%%%%%%%%%%%%%%%%%%%%%%%%%%%%%%%%%%%%%%%%%%%%%%%	
%  							NEW SECTION AHEAD	     		 		  %
%%%%%%%%%%%%%%%%%%%%%%%%%%%%%%%%%%%%%%%%%%%%%%%%%%%%%%%%%%%%%%%%%%%%%%%

\section{Exponential turnpike theorems}
\label{section:Turnpike}

\setcounter{equation}{0} \renewcommand{\theequation}{\thesection.\arabic{equation}}

In this section, we discuss the main exponential turnpike results of the paper. To this end, we first prove in Section \ref{subsection:LagrangianTurnpike} a Hilbertian turnpike theorem for the Lagrangian problem $(\Ppazo_{\Lrm})$. Then, in Section \ref{subsection:EulerianTurnpike}, we formulate intrinsic second-order hypotheses on the Eulerian data, and show that they imply the necessary Lagrangian assumptions after lifting, thanks to the horizontal-vertical decomposition of the lifted Hessian discussed in Lemma \ref{lem:LiftedConjugation}. The Eulerian turnpike theorem for $(\Ppazo_{\Erm})$ is finally obtained in Theorem \ref{thm:TurnpikeEulerian}, by combining this implication with the realization procedure from Section \ref{section:Preliminaries}.

%%%%%%%%%%%%%%%%%%%%%%%%%%%%%%%%%%%%%%%%%%%%%%%%%%%%%%%%%%%%%%%%%%%%%%%

\subsection{Turnpike in the Lagrangian framework}
\label{subsection:LagrangianTurnpike}

In this subsection, we derive our first main result which takes the form of an exponential turnpike theorem for the Lagrangian meanfield optimal control problem, that we recall is given by 
\begin{equation*}
(\Ppazo_{\Lrm}) ~~ \left\{
\begin{aligned}
\min_{(X,u_{\Lrm})} & \bigg[ \INTSeg{\INTDom{L \Big( X(t)_{\sharp} \P , X(t,\omega),u_{\Lrm}(t,\omega) \Big)}{\Omega}{\P(\omega)}}{t}{0}{T} + \varphi(X(T)_{\sharp} \P) \bigg], \\
\text{s.t.} ~ & \left\{
\begin{aligned}
& \dot{X}(t,\omega) = v \Big( X(t)_{\sharp} \P,X(t,\omega),u_{\Lrm}(t,\omega) \Big), \\
& X(0)_{\sharp} \P = \mu^0. 
\end{aligned}
\right.
\end{aligned}
\right.
\end{equation*}
In the sequel, given some triple $(\bar{X}^s,\bar{\Psi}^s,\bar{u}_{\mathrm{L}}^s) \in L^2_{\P}(\Omega,\R^{2d} \times U)$, we shall represent the evaluation of the Hessian of the lifted Hamiltonian defined in \eqref{eq:HamiltonianLag} in terms of the operator matrix  
\begin{equation}
\label{eq:LiftedBlockStructureLag}
\nabla^2 \widetilde{\Hpazo}(\bar{X}^s,\bar{\Psi}^s,\bar{u}_{\mathrm{L}}^s) :=  \begin{pmatrix}
\widetilde{\Hpazo}_{XX} & \widetilde{\Hpazo}_{X \Psi} & \widetilde{\Hpazo}_{X u} \\
\widetilde{\Hpazo}_{\Psi X} & \widetilde{\Hpazo}_{\Psi \Psi} & \widetilde{\Hpazo}_{\Psi u} \\
\widetilde{\Hpazo}_{u X} & \widetilde{\Hpazo}_{u \Psi} & \widetilde{\Hpazo}_{uu}
\end{pmatrix} 
\in \Lpazo \Big(L^2_{\P}(\Omega,\R^{2d} \times U) \Big).
\end{equation}
With these notations in hand, we may now formulate the Hilbertian hypotheses and the corresponding Lagrangian turnpike theorem.

\begin{taggedhyp}{\textnormal{(LT)}} \hfill
\label{hyp:LT}
\begin{enumerate}
\item[$(i)$] The operator $\widetilde{\Hpazo}_{uu} \in \Lpazo(L^2_{\P}(\Omega,U))$ is boundedly invertible.  
\item[$(ii)$] There exists $\widetilde{\Cpazo} \in \Lpazo(L^2_{\P}(\Omega,\R^d))$ such that $\widetilde{\Hpazo}_{X u} \widetilde{\Hpazo}_{uu}^{-1} \widetilde{\Hpazo}_{u X} - \widetilde{\Hpazo}_{XX} = \widetilde{\Cpazo}^{\star} \widetilde{\Cpazo}$.
\item[$(iii)$] Denoting by 
\begin{equation*}
\widetilde{\Apazo} := \widetilde{\Hpazo}_{\Psi X} - \widetilde{\Hpazo}_{\Psi u} \widetilde{\Hpazo}_{uu}^{-1} \widetilde{\Hpazo}_{u X} \in \Lpazo(L^2_{\P}(\Omega,\R^d)), 
\end{equation*}
the pairs $(\widetilde{\Apazo},\widetilde{\Hpazo}_{\Psi u})$ and $(\widetilde{\Apazo}^{\star},\widetilde{\Cpazo}^{\star})$ are exponentially stabilizable. 
\end{enumerate}
\end{taggedhyp} 

\begin{thm}[Exponential turnpike in the Lagrangian framework]
\label{thm:TurnpikeLagrangian}
Let $(\bar{X}^s,\bar{\Psi}^s,\bar{u}_{\mathrm{L}}^s) \in L^2_{\P}(\Omega,\R^{2d} \times U)$ be a stationary triple for $(\Ppazo^s_{\Lrm})$ satisfying Hypotheses \ref{hyp:LT}. 

Then, there exist positive constants $\epsilon,\alpha,c > 0$ such that, under the smallness condition
\begin{equation}
\label{eq:SmallnessLag}
\NormLbis{X^0 - \bar{X}^s}{\Omega,\R^d}{\P} + \NormLbis{\nabla \widetilde{\varphi}(\bar{X}^s) + \bar{\Psi}^s}{\Omega,\R^d}{\P} 
+ \|\nabla^2 \widetilde{\varphi}(\bar{X}^s) \|_{\Lpazo(L^2_{\P}(\Omega,\R^d))} \leq \epsilon,
%
%+ \sup_{H \in L^2_{\P}(\Omega,\R^d)} \frac{\langle -\nabla^2 \widetilde{\varphi}(\bar{X}^s)H,H \rangle_{L^2_{\P}(\Omega,\R^d)}}{\NormLbis{H}{\Omega,\R^d}{\P}^2} 
\end{equation}
every optimal Pontryagin triple $(X^*(\cdot),\Psi^*(\cdot),u^*_{\Lrm}(\cdot)) \in \AC^2([0,T],L^2_{\P}(\Omega,\R^{2d})) \times \Ucal_{\Lrm}$ for $(\Ppazo_{\Lrm})$ satisfies the exponential turnpike property 
\begin{equation}
\label{eq:LagrangianTurnpike}
\NormLbis{X^*(t) - \bar{X}^s}{\Omega,\R^d}{\P} + \NormLbis{\Psi^*(t) - \bar{\Psi}^s}{\Omega,\R^d}{\P} + \NormLbis{u^*_{\Lrm}(t) - \bar{u}_{\Lrm}^s}{\Omega,U}{\P} \leq c \Big( e^{-\alpha t} + e^{-\alpha(T-t)} \Big)
\end{equation}
for all times $t \in [0,T]$. 
\end{thm}

Before attacking the proof of this result, a few comments are in order.

\begin{rmk}[On the size of terminal contributions]
\label{rmk:TerminalSmallness}
The smallness of the Hessian $\nabla^2 \widetilde{\varphi}(\bar{X}^s)$ in \eqref{eq:SmallnessLag} is a sufficient condition ensuring that the linearization of terminal transversality relation can be absorbed uniformly with respect to the time horizon. It is likely not optimal, and could be replaced in specific situations by sharper endpoint estimates or by suitable terminal constraints. We choose keep this formulation because it is both robust and transparent at the level of the diagonalization argument. We also stress that our turnpike theorems hold true without smallness assumptions for linear-quadratic problems, as the linearization arguments become trivial in this context, see e.g. \cite{Trelat2015,Trelat2018}.   
\end{rmk} 

\begin{rmk}[What about control constraints?]
We wish to bring attention to the fact that our theorem remains valid when $U \subset \R^m$ is a proper closed convex set, provided the stationary control $\bar{u}_{\Lrm}^s \in L^2_{\P}(\Omega,U)$ stays at a positive essential distance from $\partial U$. This guarantees that the maximizing control remains on an interior arc in the turnpike regime. Without such an interiority condition, uniform exponential turnpike estimates may fail even in one-dimensional linear-quadratic constrained problems. Indeed, given some constants $a,b >0$ and $x_0 > 0$, consider the basic optimal control problem 
\begin{equation*}
\left\{
\begin{aligned}
\min_{(x,u)} & \INTSeg{\big( x^2(t) + u^2(t) \big)}{t}{0}{T} \\
\mathrm{s.t.} & \left\{
\begin{aligned}
& \dot{x}(t) = a x(t) + b u(t), \\
& x(0) = x^0, 
\end{aligned}
\right. 
\end{aligned}
\right.
\end{equation*}
whose companion static problem reads 
\begin{equation*}
\left\{
\begin{aligned}
\min_{(x,u)} & ~ (x^2 + u^2) \\
\mathrm{s.t.} & ~ax + bu = 0. 
\end{aligned}
\right.
\end{equation*}
If we assume  e.g. that $U = \R$, then the turnpike point is given by $(\bar{x}^s,\bar{u}^s) = (0,0)$ and the linear-quadratic exponential turnpike theorem e.g. of \cite{Trelat2015} applies. If on the other hand we let $U := \R_+$, it can be checked quite easily that the unique optimal curve for the dynamic problem satisfies $x^*(t) \geq e^{at} x_0$ for all times $t \in [0,T]$, and thus drifts away exponentially fast from the turnpike $(\bar{x}^s,\bar{u}^s) = (0,0)$. 
\end{rmk}

The proof of the latter result is heavily inspired by that of \cite{Trelat2018}, and generalizes it by incorporating a terminal cost, in the spirit of \cite{Trelat2015}, and by making explicit the operator-theoretic ingredients that will later be transferred to the Eulerian setting.

\begin{proof}[Proof of Theorem \ref{thm:TurnpikeLagrangian}]
We divide the argument into three steps.
In Step 1, we rewrite the optimality system for $(\Ppazo_{\Lrm})$ as a semilinear perturbation of a hyperbolic linear Hamiltonian system around $(\bar{X}^s,\bar{\Psi}^s,\bar{u}^s_{\Lrm})$, while keeping track of the remainder terms in a form that is suitable for bootstrap.
In Step 2, we diagonalize the leading-order linear par of the system by means of an operator Riccati equation, and obtain a stable/unstable splitting of the dynamics.
In Step 3, we combine a nonlinear variation-of-constant estimate with endpoint controls (including the terminal transversality condition) to close the bootstrap and derive the horizon-uniform exponential turnpike bound.

\medskip
\noindent\textit{Step 1 -- Semilinear perturbation and elimination of the control.} 
The goal of this first step is to linearize the dynamics of an optimal triple around the turnpike point $(\bar{X}^s,\bar{\Psi}^s,\bar{u}^s_{\Lrm}) \in L^2_{\P}(\Omega,\R^{2d} \times U)$. We keep the remainder terms under an explicit $\mathrm{o}$-form throughout the proof, because this makes the bootstrap argument transparent and matches the nonlinear structure of \cite{Trelat2015,Trelat2018}. In concrete applications, these remainders can often be replaced by quadratic $\mathrm{O}$-bounds, once a priori smallness has been established.

To begin with, note that under the regularity assumptions listed in Hypotheses \ref{hyp:H}, there exists a constant $\Delta_0 > 0$ such that, for any $0 < \Delta < \Delta_0$, the following expansion
\begin{equation}
\label{eq:TaylorHamiltonianLag}
\begin{aligned}
\nabla \widetilde{\Hpazo} \Big(\bar{X}^s  + \deltaX(t), \bar{\Psi}^s + \deltaPsi(t), \bar{u}^s_{\Lrm} + \deltau(t)\Big) = 
\begin{pmatrix}
\widetilde{\Hpazo}_{X X} & \widetilde{\Hpazo}_{X \Psi} & \widetilde{\Hpazo}_{X u} \\
\widetilde{\Hpazo}_{\Psi X} & \widetilde{\Hpazo}_{\Psi \Psi} & \widetilde{\Hpazo}_{\Psi u} \\
\widetilde{\Hpazo}_{u X} & \widetilde{\Hpazo}_{u \Psi} & \widetilde{\Hpazo}_{uu}
\end{pmatrix} 
\begin{pmatrix}
\deltaX(t) \\
\deltaPsi(t) \\
\deltau(t)
\end{pmatrix} 
+ \mathrm{o}(
\deltaX(t),
\deltaPsi(t),
\deltau(t))
\end{aligned}
\end{equation}
holds at times $t \in [0,T]$ for any $(\deltaX(\cdot),\deltaPsi(\cdot),\deltau(\cdot)) \in L^{\infty}([0,T],L^2_{\P}(\Omega,\R^{2d} \times U))$ satisfying
\begin{equation*}
\| (\deltaX(t),\deltaPsi(t),\deltau(t)) \|_{L^2_{\P}(\Omega,\R^{2d} \times U)} \leq \Delta.
\end{equation*}
In \eqref{eq:TaylorHamiltonianLag}, we notably used the fact that $\nabla \widetilde{\Hpazo}(\bar{X}^s, \bar{\Psi}^s,\bar{u}^s_{\Lrm}) = 0$ as a consequence of Theorem \ref{thm:KKTLagrangian}. Moreover, by the uniform boundedness of the second derivatives of the lifted Hamiltonian on bounded sets, there exists a nondecreasing function $\rho : [0,\Delta_0] \to \R_+$ satisfying $\rho(r) \to 0$ as $r \to 0$, and such that for any perturbation triple satisfying $\sup_{t \in [0,T]} \| (\deltaX(t),\deltaPsi(t),\deltau(t)) \|_{L^2_{\P}(\Omega,\R^{2d} \times U)} \leq \Delta$, all remainder terms denoted by $\mathrm{o}(\deltaX(t),\deltaPsi(t),\deltau(t))$ satisfy
\begin{equation}
\label{eq:ModulusRemainder}
\NormLbis{\mathrm{o}(\deltaX(t),\deltaPsi(t),\deltau(t))}{\Omega,\R^{2d} \times U}{\P} \leq \rho(\Delta) \big\| (\deltaX(t),\deltaPsi(t),\deltau(t)) \big\|_{L^2_{\P}(\Omega,\R^{2d} \times U)}
\end{equation}
uniformly with respect to $t \in [0,T]$. We keep the compact $\mathrm{o}$-notation in the sequel, but we will systematically absorb it using \eqref{eq:ModulusRemainder}.
Then, given an optimal Pontryagin triple $(X^*(\cdot),\Psi^*(\cdot),u^*_{\Lrm}(\cdot)) \in \AC^2([0,T],L^2_{\P}(\Omega,\R^{2d})) \times \Ucal_{\Lrm}$ for $(\Ppazo_{\Lrm})$, we consider the perturbation curves given by 
\begin{equation*}
\deltaX(t) := X^*(t)-\bar{X}^s, \qquad 
\deltaPsi(t) := \Psi^*(t)-\bar{\Psi}^s \qquad \text{and} \qquad  \deltau(t) := u^*_{\Lrm}(t)- \bar{u}^s_{\Lrm}.
\end{equation*}
for all times $t \in [0,T]$. From there on, we make the \emph{a priori} bootstrap assumption that
\begin{equation}
\label{eq:AprioriEstLag}
\Delta := \sup_{t \in [0,T]} \big\| (\deltaX(t),\deltaPsi(t),\deltau(t)) \|_{L^2_{\P}(\Omega,\R^{2d} \times U)} + \INTSeg{\big\| (\deltaX(t),\deltaPsi(t))\|_{L^2_{\P}(\Omega,\R^{2d})}^2}{t}{0}{T} \leq \Delta_0,
\end{equation}
a fact which we will be able to guarantee a posteriori by a suitable choice of $\epsilon>0$ in the smallness condition \eqref{eq:SmallnessLag}. At this stage, we may infer from the third equation in \eqref{eq:TaylorHamiltonianLag} above combined with the bounded invertibility of $\widetilde{\Hpazo}_{uu} \in \Lpazo(L^2_{\P}(\Omega,U))$ that
\begin{equation}
\label{eq:varU}
\deltau(t) = -\widetilde{\Hpazo}_{uu}^{-1} \Big( \widetilde{\Hpazo}_{u X} \, \deltaX(t) + \widetilde{\Hpazo}_{u \Psi} \, \deltaPsi(t) \Big) + \mathrm{o}(\deltaX(t), \deltaPsi(t))
\end{equation}
for all times $t \in [0,T]$. Then, by leveraging again the linearization formula \eqref{eq:TaylorHamiltonianLag} together with the Legendre identity \eqref{eq:varU} and the Hamiltonian formulation of the Lagrangian PMP and KKT conditions detailed in Theorem \ref{thm:PMPLagrangian} and Theorem \ref{thm:KKTLagrangian} respectively, one may show that the perturbation curves $t \in [0,T] \mapsto (\deltaX(t),\deltaPsi(t)) \in L^2_{\P}(\Omega,\R^{2d})$ follow the coupled dynamics 
\begin{equation}
\label{eq:PerturbedSystem}
\derv{}{t} \begin{pmatrix}
\deltaX(t) \\
\deltaPsi(t)
\end{pmatrix} 
= 
\begin{pmatrix}
\widetilde{\Apazo} & -\widetilde{\Hpazo}_{\Psi u} \widetilde{\Hpazo}_{uu}^{-1} \widetilde{\Hpazo}_{u \Psi} \\
\widetilde{\Cpazo}^{\star} \widetilde{\Cpazo} & -\widetilde{\Apazo}^{\star}
\end{pmatrix} \begin{pmatrix}
\deltaX(t) \\
\deltaPsi(t)
\end{pmatrix} + \mathrm{o}(\deltaX(t), \deltaPsi(t)),
\end{equation}
in the mild sense \cite{Pazy2012}, and comply with the boundary conditions
\begin{equation}
\label{eq:PerturbedBoundary}
\begin{pmatrix}
\deltaX(0) \\
\deltaPsi(T) 
\end{pmatrix}
=
\begin{pmatrix}
X^0 - \bar{X}^s \\
- \nabla \widetilde{\varphi}(\bar{X}^s) - \bar{\Psi}^s  -\nabla^2 \widetilde{\varphi}(\bar{X}^s) \deltaX(T) 
\end{pmatrix} + \mathrm{o}(\deltaX(T)).
\end{equation}

\medskip
\noindent\textit{Step 2 -- Diagonalization of the leading-order linear part.}
We now diagonalize the linear Hamiltonian matrix in \eqref{eq:PerturbedSystem}. This is the operator-theoretic analogue of the stable/unstable splitting used in \cite{Trelat2018}, and relies on a Riccati equation and an associated Lyapunov equation.
To this end, we denote by $\Ppazo \in \Lpazo(L^2_{\P}(\Omega, \R^d))$ the unique nonnegative self-adjoint solution of the algebraic Riccati equation
\begin{equation}
\label{eq:RiccatiOperator}
\widetilde{\Apazo}^{\star} \Ppazo + \Ppazo \widetilde{\Apazo} + \widetilde{\Cpazo}^{\star} \widetilde{\Cpazo} + \Ppazo \widetilde{\Hpazo}_{\Psi u} \widetilde{\Hpazo}_{uu}^{-1} \widetilde{\Hpazo}_{u \Psi} \Ppazo = 0,
\end{equation}
whose well-posedness under Hypotheses \ref{hyp:LT} follows e.g. from \cite[Theorem 9.5]{Zabczyk2020}. Moreover, the same theorem guarantees that, under the exponential stabilizability of $(\widetilde{\Apazo}, \widetilde{\Hpazo}_{\Psi u})$, the operator $(\widetilde{\Apazo} + \widetilde{\Hpazo}_{\Psi u} \widetilde{\Hpazo}_{uu}^{-1} \widetilde{\Hpazo}_{u \Psi} \Ppazo) \in\Lpazo(L^2_{\P}(\Omega,\R^d))$ generates an exponentially stable $C^0$-semigroup $(\Spazo(t))_{t \geq 0} \subset \Lpazo(L^2_{\P}(\Omega,\R^d))$ satisfying 
\begin{equation}
\label{eq:ExponentialStabilitySemi}
\| \Spazo(t) \|_{\Lpazo(L^2_{\P}(\Omega, \R^d))} \leq M e^{-\beta t}
\end{equation} 
for all times $t \geq 0$ and some constants $M, \beta > 0$. In particular, the spectral bound of its generator is strictly negative, namely
\begin{equation*}
\sup \bigg\{\mathrm{Re}(\lambda) ~\,\text{s.t.}~ \lambda \in \sigma \Big(\widetilde{\Apazo} + \widetilde{\Hpazo}_{\Psi u} \widetilde{\Hpazo}_{uu}^{-1} \widetilde{\Hpazo}_{u \Psi} \Ppazo \Big) \bigg\} \leq - \beta < 0. 
\end{equation*}
We consider then the self-adjoint operator $\Epazo \in \Lpazo(L^2_{\P}(\Omega,\R^d))$ given by 
\begin{equation*}
\Epazo := -\INTSeg{\Spazo(t)^{\star} \widetilde{\Hpazo}_{\Psi u} \widetilde{\Hpazo}_{uu}^{-1} \widetilde{\Hpazo}_{u \Psi} \Spazo(t)}{t}{0}{+\infty},
\end{equation*}
and note that since the semigroup $(\Spazo(t))_{t \geq 0}$ is exponentially stable, it follows e.g. from \cite[Theorem 2.7]{Zabczyk2020} that the latter solves the algebraic Lyapunov equation
\begin{equation}
\label{eq:LyapunovOperator}
\left( \widetilde{\Apazo} + \widetilde{\Hpazo}_{\Psi u} \widetilde{\Hpazo}_{uu}^{-1} \widetilde{\Hpazo}_{u \Psi} \Ppazo \right) \Epazo + \Epazo \left( \widetilde{\Apazo} + \widetilde{\Hpazo}_{\Psi u} \widetilde{\Hpazo}_{uu}^{-1} \widetilde{\Hpazo}_{u \Psi} \Ppazo \right)^{\star} + \widetilde{\Hpazo}_{\Psi u} \widetilde{\Hpazo}_{uu}^{-1} \widetilde{\Hpazo}_{u \Psi} = 0.
\end{equation}
We are now in position to diagonalize the operator matrix
\begin{equation*}
\begin{pmatrix}
\widetilde{\Apazo} & -\widetilde{\Hpazo}_{\Psi u} \widetilde{\Hpazo}_{uu}^{-1} \widetilde{\Hpazo}_{u \Psi} \\
\widetilde{\Cpazo}^{\star} \widetilde{\Cpazo} & -\widetilde{\Apazo}^{\star},
\end{pmatrix},
\end{equation*}
which amounts to boundedly decoupling the linearized state and costate dynamics. To this end, following the ideas of \cite[Lemma 1]{Trelat2018}, we introduce the following bounded linear operator 
\begin{equation}
\label{eq:ChangeVariables}
\Tpazo := \begin{pmatrix}
\Id + \Epazo \Ppazo & \Epazo \\
\Ppazo & \Id
\end{pmatrix} \in \Lpazo(L^2_{\P}(\Omega,\R^{2d})) 
\end{equation}
whose inverse is given by 
\begin{equation}
\label{eq:ChangeVariablesInv}
\Tpazo^{-1} = \begin{pmatrix}
\Id & -\Epazo \\
-\Ppazo & \Id + \Ppazo \Epazo
\end{pmatrix} \in \Lpazo(L^2_{\P}(\Omega,\R^{2d})).
\end{equation}
Then, it may be checked through direct computations relying on \eqref{eq:RiccatiOperator} and \eqref{eq:LyapunovOperator}  that
\begin{equation}
\label{eq:diagonalization}
\Tpazo
 \begin{pmatrix}
\widetilde{\Apazo} & -\widetilde{\Hpazo}_{\Psi u} \widetilde{\Hpazo}_{uu}^{-1} \widetilde{\Hpazo}_{u \Psi} \\
\widetilde{\Cpazo}^{\star} \widetilde{\Cpazo} & -\widetilde{\Apazo}^{\star}
\end{pmatrix} 
\Tpazo^{-1} = 
\begin{pmatrix}
\widetilde{\Apazo} + \widetilde{\Hpazo}_{\Psi u} \widetilde{\Hpazo}_{uu}^{-1} \widetilde{\Hpazo}_{u \Psi} \Ppazo & 0 \\
0 & -(\widetilde{\Apazo} + \widetilde{\Hpazo}_{\Psi u} \widetilde{\Hpazo}_{uu}^{-1} \widetilde{\Hpazo}_{u \Psi} \Ppazo)^{\star}
\end{pmatrix},
\end{equation}
which leads us to consider the auxiliary perturbation curves $(\deltaY(\cdot),\deltaLambda(\cdot)) \in \AC^2([0,T],L^2_{\P}(\Omega,\R^{2d}))$ defined by 
\begin{equation}
\label{eq:ChangeVariablesFormula}
\begin{pmatrix}
\deltaY(t) \\
\deltaLambda(t)
\end{pmatrix} := \Tpazo \begin{pmatrix}
\deltaX(t) \\
\deltaPsi(t)
\end{pmatrix}
\end{equation}
for all times $t \in [0,T]$. By combining \eqref{eq:PerturbedSystem} and \eqref{eq:PerturbedBoundary} with \eqref{eq:diagonalization}, it may be finally checked that the latter pair solves the decoupled dynamics
\begin{equation}
\label{eq:DiagonalizedDynamics}
\derv{}{t}{} \begin{pmatrix}
\deltaY(t) \\
\deltaLambda(t)
\end{pmatrix} = \begin{pmatrix}
\widetilde{\Apazo} + \widetilde{\Hpazo}_{\Psi u} \widetilde{\Hpazo}_{uu}^{-1} \widetilde{\Hpazo}_{u \Psi} \Ppazo & 0 \\
0 & -(\widetilde{\Apazo} + \widetilde{\Hpazo}_{\Psi u} \widetilde{\Hpazo}_{uu}^{-1} \widetilde{\Hpazo}_{u \Psi} \Ppazo)^{\star}
\end{pmatrix} \begin{pmatrix}
\deltaY(t) \\
\deltaLambda(t)
\end{pmatrix} + \mathrm{o}(\deltaY(t), \deltaLambda(t))
\end{equation}
in the mild sense, complemented by the boundary conditions
\begin{equation}
\label{eq:DiagonalizedBoundary}
\begin{pmatrix}
\deltaY(0) \\
\deltaLambda(T)
\end{pmatrix}
= 
\begin{pmatrix}
(\Id + \Epazo \Ppazo)\deltaX(0) + \Epazo \deltaPsi(0) \\
\Ppazo \deltaX(T) + \deltaPsi(T)
\end{pmatrix}. 
\end{equation}

\medskip
\noindent\textit{Step 3 -- Exponential estimate, endpoint control, and closure of the bootstrap.} 
In this final step, we combine the hyperbolicity of the diagonalized dynamics \eqref{eq:DiagonalizedDynamics}, the quantitative control of the nonlinear remainders provided by \eqref{eq:ModulusRemainder}, and the terminal transversality relation to derive a horizon-uniform exponential estimate, and close the a priori bootstrap bound \eqref{eq:AprioriEstLag}.

To begin with, note that by the very definition of mild solution to a semilinear evolution equation, the unique integral curves of \eqref{eq:DiagonalizedDynamics} are given explicitly by 
\begin{equation*}
\left\{
\begin{aligned}
\deltaY(t) & = \Spazo(t) \deltaY(0) + \INTSeg{\Spazo(t-s) \, \mathrm{o}(\deltaY(s),\deltaLambda(s))}{s}{0}{t}, \\
\deltaLambda(t) & =  \Spazo^*(T-t) \deltaLambda(T) - \INTSeg{\Spazo^*(s-t) \, \mathrm{o}(\deltaY(s),\deltaLambda(s))}{s}{t}{T}, 
\end{aligned}
\right.
\end{equation*}
where $(\Spazo(t))_{t \geq 0} \subset \Lpazo(L^2_{\P}(\Omega,\R^d))$ is the semigroup generated by $\widetilde{\Apazo} + \widetilde{\Hpazo}_{\Psi u} \widetilde{\Hpazo}_{uu}^{-1} \widetilde{\Hpazo}_{u \Psi} \Ppazo$.
Let us denote by $\Rpazo(t) = \mathrm{o}(\deltaY(t),\deltaLambda(t))$ the remainder term in \eqref{eq:DiagonalizedDynamics}.
By the boundedness of the change of variables \eqref{eq:ChangeVariables}-\eqref{eq:ChangeVariablesInv}, the Legendre reduction \eqref{eq:varU}, and the modulus estimate \eqref{eq:ModulusRemainder}, we may use (up to enlarging it by a multiplicative constant) the same modulus $\rho : [0,\Delta_0] \to \R_+$  to write
\begin{equation*}%\label{eq:RemainderYLambda}
\NormLbis{\Rpazo(t)}{\Omega,\R^{2d}}{\P} \leq \rho(\Delta) \Big( \NormLbis{\deltaY(t)}{\Omega,\R^d}{\P} + \NormLbis{\deltaLambda(t)}{\Omega,\R^d}{\P} \Big)
\end{equation*}
uniformly for $t \in [0,T]$.
Using \eqref{eq:ExponentialStabilitySemi} and the variation-of-constants formula above, one may show that the following quantity
\begin{equation*}
W_{\beta} := \sup_{t \in [0,T]} \Big( e^{\beta t/2} \NormLbis{\deltaY(t)}{\Omega,\R^d}{\P} + e^{\beta(T-t)/2} \NormLbis{\deltaLambda(t)}{\Omega,\R^d}{\P} \Big),
\end{equation*}
obeys the estimate
\begin{equation*}
W_{\beta} \leq M \Big( \NormLbis{\deltaY(0)}{\Omega,\R^d}{\P} + \NormLbis{\deltaLambda(T)}{\Omega,\R^d}{\P} \Big) + \frac{4M}{\beta} \rho(\Delta) W_{\beta}.
\end{equation*}
Hence, upon choosing $\Delta_0 > 0$ in such a way that $4M \rho(\Delta_0) \leq \beta/2$, we get that $W_{\beta} \leq 2M (\|\deltaY(0)\| + \|\deltaLambda(T)\|)$, and therefore
\begin{equation}
\label{eq:FirstExponentialIneq}
\| \deltaY(t) \|_{L^2_{\P}(\Omega, \R^d)} + \| \deltaLambda(t) \|_{L^2_{\P}(\Omega, \R^d)} \leq 2M \Big( \| \deltaY(0) \|_{L^2_{\P}(\Omega, \R^d)} \, e^{-\beta t / 2} + \| \deltaLambda(T) \|_{L^2_{\P}(\Omega, \R^d)} \, e^{-\beta (T-t)/2} \Big)
\end{equation}
for all times $t \in [0,T]$, where $M,\beta > 0$ are the constants appearing in \eqref{eq:ExponentialStabilitySemi}. Recalling the change of variables \eqref{eq:ChangeVariablesFormula} and the expression of its inverse \eqref{eq:ChangeVariablesInv}, the estimate in \eqref{eq:FirstExponentialIneq} implies that 
\begin{equation}
\label{eq:SecondExponentialIneq}
\begin{aligned}
& \| \deltaX(t) \|_{L^2_{\P}(\Omega, \R^d)} + \| \deltaPsi(t) \|_{L^2_{\P}(\Omega, \R^d)} \\
& \leq c_0 \bigg( \| \deltaX(0) \|_{L^2_{\P}(\Omega, \R^d)} + \| \deltaPsi(0) \|_{L^2_{\P}(\Omega, \R^d)} + \| \deltaX(T) \|_{L^2_{\P}(\Omega, \R^d)} + \| \deltaPsi(T) \|_{L^2_{\P}(\Omega, \R^d)} \hspace{-0.05cm} \bigg) \Big( e^{-\beta t / 2} + e^{-\beta (T-t)/2} \Big) \\
& \leq c_0 \bigg( \| X^0 - \bar{X}^s \|_{L^2_{\P}(\Omega, \R^d)} + \| \nabla \widetilde{\varphi}(\bar{X}^s) + \bar{\Psi}^s \|_{L^2_{\P}(\Omega, \R^d)} \bigg) \Big( e^{-\beta t / 2} + e^{-\beta (T-t)/2} \Big) \\
& \hspace{0.45cm}  + c_0 \bigg( \| \deltaPsi(0) \|_{L^2_{\P}(\Omega, \R^d)} + \Big(1 + \| \nabla^2 \widetilde{\varphi}(\bar{X}^s) \|_{\Lpazo(L^2_{\P}(\Omega,\R^d))} \Big) \| \deltaX(T) \|_{L^2_{\P}(\Omega, \R^d)} + \rho(\Delta) \| \deltaX(T) \|_{L^2_{\P}(\Omega, \R^d)} \bigg) \\
& \hspace{12.25cm} \times \Big( e^{-\beta t / 2} + e^{-\beta (T-t)/2} \Big),
\end{aligned} 
\end{equation}
for some constant $c_0 > 0$ depending only on the operator norms of $\Ppazo,\Epazo \in \Lpazo(L^2_{\P}(\Omega,\R^d))$. 

Following the above considerations, establishing the exponential turnpike property boils down to deriving upper bounds on $\| \deltaX(T) \|_{L^2_{\P}(\Omega, \R^d)}$, $\| \deltaPsi(0) \|_{L^2_{\P}(\Omega, \R^d)}$ and $\| \deltaPsi(T) \|_{L^2_{\P}(\Omega, \R^d)}$ that are uniform in $T > 0$. To this end, recall first that under Hypothesis \ref{hyp:LT}-$(iii)$, there exists an operator $\Qpazo \in \Lpazo(L^2_{\P}(\Omega, \R^d))$ such that $\widetilde{\Apazo}^{\star} + \widetilde{\Cpazo}^{\star} \Qpazo$ generates an exponentially stable semigroup, and let $Z(\cdot) \in C^1([0,T],L^2_{\P}(\Omega,\R^d))$ be the unique solution of 
\begin{equation}
\label{eq:Ldef}
\dot{Z}(t) = -(\widetilde{\Apazo}^{\star} + \widetilde{\Cpazo}^{\star} \Qpazo) Z(t), \qquad
Z(T) = \deltaX(T).
\end{equation}
Denote by $(\Upazo(t))_{t \geq 0} \subset \Lpazo(L^2_{\P}(\Omega,\R^d))$ the semigroup generated by $-(\widetilde{\Apazo}^{\star} + \widetilde{\Cpazo}^{\star} \Qpazo)$. By Hypothesis \ref{hyp:LT}-$(iii)$, this semigroup is exponentially stable, hence there exist $M_\Qpazo,\gamma_\Qpazo > 0$ such that
\begin{equation*}
\| \Upazo(t) \|_{\Lpazo(L^2_{\P}(\Omega,\R^d))} \leq M_\Qpazo e^{-\gamma_\Qpazo t}
\end{equation*}
for all $t \geq 0$. In particular, $Z(t) = \Upazo(T-t)\deltaX(T)$ and there exists a constant $c_1 > 0$ independent of $T > 0$ such that
\begin{equation}
\label{eq:LIneqs}
\| Z(0) \|^2_{L^2_{\P}(\Omega, \R^d)} + \int_{0}^{T} \| Z(t) \|^2_{L^2_{\P}(\Omega, \R^d)} \leq c_1 \| \deltaX(T) \|^2_{L^2_{\P}(\Omega, \R^d)}. 
\end{equation}
Then, combining the dynamics of the first equation in \eqref{eq:PerturbedSystem} with the definition \eqref{eq:Ldef} of $Z(\cdot) \in C^1([0,T],L^2_{\P}(\Omega,\R^d))$, one further has that 
\begin{equation}
\label{eq:deltaXTEst1}
\begin{aligned}
\| \deltaX(T) \|^2_{L^2_{\P}(\Omega, \R^d)} & = \langle \deltaX(T) , Z(T) \rangle_{L^2_{\P}(\Omega,\R^d)} \\
& =  \langle \deltaX(0), Z(0) \rangle_{L^2_{\P}(\Omega, \R^d)} + \INTSeg{\big\langle \widetilde{\Hpazo}_{\Psi u} \widetilde{\Hpazo}_{uu}^{-1} \widetilde{\Hpazo}_{u \Psi} \deltaPsi(t), Z(t) \big\rangle_{L^2_{\P}(\Omega, \R^d)}}{t}{0}{T} \\
& \hspace{0.45cm} - \INTSeg{\big\langle \Qpazo^{\star} \widetilde{\Cpazo} \deltaX(t), Z(t) \big\rangle_{L^2_{\P}(\Omega, \R^d)}}{t}{0}{T} + \mathrm{o} \bigg( \INTSeg{\| (\deltaX(t), \deltaPsi(t)) \|_{L^2_{\P}(\Omega, \R^{2d})}^2}{t}{0}{T} \bigg).
\end{aligned}
\end{equation}
Next, it follows from the Cauchy-Schwarz inequality together with \eqref{eq:LIneqs} that there exist positive constants $c_2,c_3 > 0$ such that 
\begin{equation}
\label{eq:deltaXTEst2}
\begin{aligned}
& \bigg| \int_{0}^{T} \big\langle \widetilde{\Hpazo}_{\Psi u} \widetilde{\Hpazo}_{uu}^{-1} \widetilde{\Hpazo}_{u \Psi} \deltaPsi(t), Z(t) \big\rangle_{L^2_{\P}(\Omega, \R^d)} \, \mathrm{d}t \, \bigg| \\
& \hspace{0.45cm} \leq \bigg( \int_{0}^{T} \| \widetilde{\Hpazo}_{\Psi u} \widetilde{\Hpazo}_{uu}^{-1} \widetilde{\Hpazo}_{u \Psi} \deltaPsi(t) \|^{2}_{L^2_{\P}(\Omega, \R^d)} \, \mathrm{d}t \bigg)^{1/2} \left( \int_{0}^{T} \| Z(t) \|^{2}_{L^2_{\P}(\Omega, \R^d)} \, \mathrm{d}t \right)^{1/2} \\
&  \hspace{0.45cm} \leq c_2 \left( \int_{0}^{T} \| \widetilde{\Hpazo}_{uu}^{-1/2} \widetilde{\Hpazo}_{u \Psi} \deltaPsi(t) \|^2_{L^2_{\P}(\Omega, U)} \, \mathrm{d}t \right)^{1/2} \| \deltaX(T) \|_{L^2_{\P}(\Omega, \R^d)},
%
%&  \hspace{0.45cm} \leq c_2 \| \widetilde{\Hpazo}_{\Psi u} \widetilde{\Hpazo}_{uu}^{-1/2} \|_{\Lpazo \big(L^2_{\P}(\Omega, U),L^2_{\P}(\Omega, \R^d) \big)} \left( \int_{0}^{T} \| \widetilde{\Hpazo}_{uu}^{-1/2} \widetilde{\Hpazo}_{u \Psi} \deltaPsi(t) \, \mathrm{d}t \|^2_{L^2_{\P}(\Omega, U)} \, \mathrm{d}t \right)^{1/2}  \| \deltaX(T) \|_{L^2_{\P}(\Omega, \R^d)},
\end{aligned}
\end{equation}
and
\begin{equation}
\label{eq:deltaXTEst3}
\begin{aligned}
& \bigg| \int_{0}^{T} \big \langle \Qpazo^{\star} \widetilde{\Cpazo} \deltaX(t), Z(t) \big\rangle_{L^2_{\P}(\Omega, \R^d)} \, \mathrm{d}t \, \bigg| \\
& \hspace{0.45cm} \leq \| \Qpazo^{\star} \|_{\Lpazo(L^2_{\P}(\Omega, \R^d))} \left( \int_{0}^{T} \| \widetilde{\Cpazo} \deltaX(t) \|^{2}_{L^2_{\P}(\Omega, \R^d)} \, \mathrm{d}t \right)^{1/2} \left( \int_{0}^{T} \| Z(t) \|^{2}_{L^2_{\P}(\Omega, \R^d)} \, \mathrm{d}t \right)^{1/2} \\
& \hspace{0.45cm} \leq c_3 \left( \int_{0}^{T} \| \widetilde{\Cpazo} \deltaX(t) \|^{2}_{L^2_{\P}(\Omega, \R^d)} \, \mathrm{d}t \right)^{1/2} \| \deltaX(T) \|_{L^2_{\P}(\Omega, \R^d)}.
%
%& \hspace{0.45cm} \leq c_3 \| \Qpazo^{\star} \|_{\Lpazo(L^2_{\P}(\Omega, \R^d))} \left( \int_{0}^{T} \| \widetilde{\Cpazo} \deltaX(t) \|^{2}_{L^2_{\P}(\Omega, \R^d)} \, \mathrm{d}t \right)^{1/2} \| \deltaX(T) \|_{L^2_{\P}(\Omega, \R^d)}.
\end{aligned}
\end{equation}
Upon plugging both \eqref{eq:deltaXTEst2}-\eqref{eq:deltaXTEst3} in \eqref{eq:deltaXTEst1} while using \eqref{eq:LIneqs} to estimate the initial condition, we finally obtain that
\begin{multline}
\label{eq:deltaXEstFinal}
\| \deltaX(T) \|^2_{L^2_{\P}(\Omega, \R^d)} \leq c_4 \left( \|\deltaX(0)\|^2_{L^2_{\P}(\Omega, \R^d)} + \int_{0}^{T} \Big( \| \widetilde{\Hpazo}_{uu}^{-1/2} \widetilde{\Hpazo}_{u \Psi} \deltaPsi(t) \|^2_{L^2_{\P}(\Omega, U)} + \| \widetilde{\Cpazo} \deltaX(t) \|^2_{L^2_{\P}(\Omega, \R^d)} \Big) \, \mathrm{d}t \right) \\
+ \mathrm{o} \bigg( \INTSeg{\| (\deltaX(t), \deltaPsi(t)) \|_{L^2_{\P}(\Omega, \R^{2d})}^2}{t}{0}{T} \bigg)
\end{multline}
for a suitable constant $c_4 > 0$ that is still independent of $T > 0$. Applying similar arguments to the second equation in \eqref{eq:PerturbedSystem}, we likewise get
\begin{multline}
\label{eq:deltaPsiEstFinal}
\| \deltaPsi(0) \|^2_{L^2_{\P}(\Omega, \R^d)} \leq c_5 \left( \| \deltaPsi(T) \|^{2}_{L^2_{\P}(\Omega, \R^d)} + \int_{0}^{T} \left( \| \widetilde{\Hpazo}_{uu}^{-1/2} \widetilde{\Hpazo}_{u \Psi} \deltaPsi(t) \|^2_{L^2_{\P}(\Omega, U)} + \| \widetilde{\Cpazo} \deltaX(t) \|^{2}_{L^2_{\P}(\Omega, \R^d)} \right) \, \mathrm{d}t \right) \\
 + \mathrm{o} \bigg( \INTSeg{\| (\deltaX(t), \deltaPsi(t)) \|_{L^2_{\P}(\Omega, \R^{2d})}^2}{t}{0}{T} \bigg)
\end{multline}
for some constant $c_5 > 0$. At this stage, multiplying the first equation in \eqref{eq:PerturbedSystem} by $\deltaPsi(t)$ and the second one by $\deltaX(t)$ and subsequently integrating in time further yields 
\begin{equation*}
\begin{aligned}
& \int_{0}^{T}  \left( \| \widetilde{\Hpazo}_{uu}^{-1/2} \widetilde{\Hpazo}_{u \Psi} \deltaPsi(t) \|^2_{L^2_{\P}(\Omega, U)} + \| \widetilde{\Cpazo} \deltaX(t) \|^2_{L^2_{\P}(\Omega, \R^d)} \right) \, \mathrm{d}t \\
& \hspace{0.8cm} = \langle \deltaX(T) , \deltaPsi(T) \rangle_{L^2_{\P}(\Omega,\R^d)} - \langle \deltaX(0) , \deltaPsi(0) \rangle_{L^2_{\P}(\Omega,\R^d)} + \mathrm{o} \bigg( \INTSeg{\| (\deltaX(t), \deltaPsi(t)) \|_{L^2_{\P}(\Omega, \R^{2d})}^2}{t}{0}{T} \bigg) \\
& \hspace{0.8cm}\leq \sigma \Big( \| \deltaX(T) \|^2_{L^2_{\P}(\Omega, \R^d)} + \| \deltaPsi(0) \|^2_{L^2_{\P}(\Omega, \R^d)} \Big) + \dfrac{1}{4 \sigma} \bigg( \| \deltaX(0) \|^2_{L^2_{\P}(\Omega, \R^d)} + \| \delta\Psi(T) \|^2_{L^2_{\P}(\Omega, \R^d)} \bigg) \\
& \hspace{1.45cm} + \mathrm{o} \bigg( \INTSeg{\| (\deltaX(t), \deltaPsi(t)) \|_{L^2_{\P}(\Omega, \R^{2d})}^2}{t}{0}{T} \bigg)
\end{aligned}
\end{equation*}
for any real number $\sigma > 0$. The above estimate combined with \eqref{eq:deltaXEstFinal} and \eqref{eq:deltaPsiEstFinal} then implies that
\begin{equation}
\label{eq:NormIntegralEst}
\begin{aligned}
& \int_{0}^{T} \left( \| \widetilde{\Hpazo}_{uu}^{-1/2} \widetilde{\Hpazo}_{u \Psi} \deltaPsi(t) \|^2_{L^2_{\P}(\Omega, U)} + \| \widetilde{\Cpazo} \deltaX(t) \|^2_{L^2_{\P}(\Omega, \R^d)} \right) \, \mathrm{d}t \\
& \leq 2 c_6 \sigma \int_{0}^{T} \left( \| \widetilde{\Hpazo}_{uu}^{-1/2} \widetilde{\Hpazo}_{u \Psi} \deltaPsi(t) \|^2_{L^2_{\P}(\Omega, U)} + \| \widetilde{\Cpazo} \deltaX(t) \|^2_{L^2_{\P}(\Omega, \R^d)} \right) \, \mathrm{d}t \\
& \hspace{0.25cm} + \left( \dfrac{1}{4 \sigma} + 2 c_6 \sigma \right) \left( \| \deltaX(0) \|^2_{L^2_{\P}(\Omega, \R^d)} + \| \deltaPsi(T) \|^2_{L^2_{\P}(\Omega, \R^d)} \right) + \rho(\Delta) \INTSeg{\| (\deltaX(t), \deltaPsi(t)) \|_{L^2_{\P}(\Omega, \R^{2d})}^2}{t}{0}{T},
\end{aligned}
\end{equation}
where $c_6 := \max\{c_4,c_5\}$. Choosing finally $\sigma := 1/(4 c_6)$ in \eqref{eq:NormIntegralEst} and merging the latter with \eqref{eq:PerturbedBoundary}, \eqref{eq:deltaXEstFinal} and \eqref{eq:deltaPsiEstFinal}, we obtain at last 
\begin{equation}
\label{eq:BoundaryEstFinal}
\begin{aligned}
& \Big(1 - c_7 \|\nabla^2 \widetilde{\varphi}(\bar{X}^s) \|_{\Lpazo(L^2_{\P}(\Omega,\R^d))} \Big) \| \deltaX(T) \|_{L^2_{\P}(\Omega, \R^d)} + \| \deltaPsi(0) \|_{L^2_{\P}(\Omega, \R^d)} \\
& \hspace{3cm} \leq c_7 \bigg( \| X^0 - \bar{X}^s  \|_{L^2_{\P}(\Omega, \R^d)} + \| \nabla \widetilde{\varphi}(\bar{X}^s) + \bar{\Psi}^s \|_{L^2_{\P}(\Omega, \R^d)} \bigg)\\
& \hspace{3.45cm} + \rho(\Delta) \bigg( \NormLbis{\deltaX(T)}{\Omega,\R^d}{\P} +  \INTSeg{\| (\deltaX(t), \deltaPsi(t)) \|_{L^2_{\P}(\Omega, \R^{2d})}^2}{t}{0}{T} \bigg),
\end{aligned}
\end{equation}
where the constant $c_7 > 0$ is again independent of $T > 0$. Hence, choosing $\epsilon > 0$ in \eqref{eq:SmallnessLag} so that $c_7 \epsilon \leq 1/2$, we deduce from \eqref{eq:SecondExponentialIneq} and \eqref{eq:BoundaryEstFinal} that
\begin{equation}
\label{eq:TurnpikeEstimate}
\begin{aligned}
& \| \deltaX(t) \|_{L^2_{\P}(\Omega, \R^d)} + \| \deltaPsi(t) \|_{L^2_{\P}(\Omega, \R^d)} \\
& \leq \Bigg( c_8 \Big( \| X^0 - \bar{X}^s  \|_{L^2_{\P}(\Omega, \R^d)} + \| \nabla \widetilde{\varphi}(\bar{X}^s) + \bar{\Psi}^s \|_{L^2_{\P}(\Omega, \R^d)} \Big) \\
& \hspace{0.8cm} + \rho(\Delta) \Big( \NormLbis{\deltaX(T)}{\Omega,\R^d}{\P} +  \INTSeg{\| (\deltaX(t), \deltaPsi(t)) \|_{L^2_{\P}(\Omega, \R^{2d})}^2}{t}{0}{T} \Big) \Bigg) \Big( e^{-\beta t/2} + e^{-\beta(T-t)/2} \Big),
\end{aligned}
\end{equation}
for all times $t \in [0, T]$, where the constant $c_8 > 0$ is independent of $T > 0$. This, together with \eqref{eq:varU}, leads to a similar estimate for $\deltau(\cdot)$.

Finally, there remains to close the bootstrap \eqref{eq:AprioriEstLag}. By integrating \eqref{eq:TurnpikeEstimate} over $[0,T]$ and using \eqref{eq:BoundaryEstFinal} to control the endpoint terms, we obtain, for some constant $C>0$ independent of $T>0$ that
\begin{equation*}
\Delta \leq C \Big( \NormLbis{X^0 - \bar{X}^s}{\Omega,\R^d}{\P} + \NormLbis{\nabla \widetilde{\varphi}(\bar{X}^s) + \bar{\Psi}^s}{\Omega,\R^d}{\P} \Big) + C \rho(\Delta)\Delta.
\end{equation*}
Choosing $\epsilon > 0$ in \eqref{eq:SmallnessLag} so that $C\epsilon \leq \Delta_0/2$ and $C\rho(\Delta_0) \leq 1/2$ yields $\Delta \leq \Delta_0$, hence the a priori bound \eqref{eq:AprioriEstLag} is indeed satisfied. Then \eqref{eq:TurnpikeEstimate} gives the desired estimate \eqref{eq:LagrangianTurnpike}.
\end{proof}

%%%%%%%%%%%%%%%%%%%%%%%%%%%%%%%%%%%%%%%%%%%%%%%%%%%%%%%%%%%%%%%%%%%%%%%

\subsection{Turnpike in the Eulerian framework}
\label{subsection:EulerianTurnpike}

In this section, we establish exponential turnpike estimates for the Eulerian problem
\begin{equation*}
(\Ppazo_{\Erm}) ~~ \left\{
\begin{aligned}
\min_{(\mu,u_{\Erm})} & \bigg[ \INTSeg{\INTDom{L \Big( \mu(t),x,u_{\Erm}(t,x) \Big)}{\R^d}{\mu(t)(x)}}{t}{0}{T} + \varphi(\mu(T)) \bigg], \\
\textnormal{s.t.}~ & \left\{
\begin{aligned}
& \partial_t \mu(t) + \Div_x \Big( v(\mu(t),\cdot,u_{\Erm}(t,\cdot)) \mu(t) \Big) = 0, \\
& \mu(0) = \mu^0.
\end{aligned}
\right.
\end{aligned}
\right.
\end{equation*}
This results hinges upon the fact that it is possible to formulate intrinsic second-order hypotheses on the relaxed Hamiltonian, and then to show that they imply the Hilbertian hypotheses of Theorem \ref{thm:TurnpikeLagrangian} after lifting. In what follows given a stationary Eulerian Pontryagin triple $(\bar{\nu}^s,\bar{u}_{\Erm}^s) \in \Pcal_2(\R^{2d}) \times L^2_{\bar{\mu}_s}(\R^d,U)$, we consider the occupation measure
\begin{equation*}
\bar{\Bnu}^s := (\pi^1,\pi^2,\bar{u}_{\Erm}^s \circ \pi^1)_{\sharp}\bar{\nu}^s \in \Pcal_2(\R^{2d} \times U).
\end{equation*}
Recalling the definition of the relaxed Hamiltonian 
\begin{equation*}%\label{eq:RelaxedHamiltonian}
\Hcal(\Bnu) = \INTDom{H \big( \pi^1_{\sharp}\Bnu,x,p,u \big)}{\R^{2d} \times U}{\Bnu(x,p,u)}
\end{equation*}
presented in Corollary \ref{cor:EulerianKKTHamil}, we may represent its Wasserstein Hessian at $\bar{\Bnu}^s$ by the operator matrix
\begin{equation*}
\Hess \Hcal(\bar{\Bnu}^s) =
\begin{pmatrix}
\Hpazo_{xx} & \Hpazo_{xp} & \Hpazo_{xu} \\
\Hpazo_{px} & \Hpazo_{pp} & \Hpazo_{pu} \\
\Hpazo_{ux} & \Hpazo_{up} & \Hpazo_{uu}
\end{pmatrix},
\end{equation*}
and note that following Definition \ref{def:WassHess}, each block therein splits into
\begin{equation}
\label{eq:EulerianBlockSplitting}
\Hpazo_{\alpha\beta} = \Mpazo_{\alpha\beta} + \Kpazo_{\alpha\beta},
\end{equation}
where $\Mpazo_{\alpha\beta}$ is a multiplication operator and $\Kpazo_{\alpha\beta}$ is an integral operator. In particular, $\Kpazo_{\alpha\beta} = 0$ whenever $\alpha,\beta \neq x$, since nonlocal derivatives only appear at the level of the marginal $\pi^1_{\sharp} \bar{\Bnu}^s = \bar{\mu}^s$ following the computations detailed in \eqref{eq:HamiltonianKKT0}.

\begin{taggedhyp}{\textnormal{(ET)}} \hfill
\label{hyp:ET}
\begin{enumerate}
\item[$(i)$] The operator $\Hpazo_{uu} \in \Lpazo\big(L^2_{\bar{\Bnu}^s}(\R^{2d} \times U,U)\big)$ is boundedly invertible.
\item[$(ii)$] There exists $\Cpazo_{\mathrm{hor}} \in \Lpazo\big(L^2_{\bar{\Bnu}^s}(\R^{2d} \times U,\R^d)\big)$ such that
$\Hpazo_{xu}\Hpazo_{uu}^{-1}\Hpazo_{ux} - \Hpazo_{xx} = \Cpazo_{\mathrm{hor}}^{\star}\Cpazo_{\mathrm{hor}}$.
\item[$(iii)$] Denoting by 
\begin{equation*}
\Apazo_{\mathrm{hor}} := \Hpazo_{px} - \Hpazo_{pu}\Hpazo_{uu}^{-1}\Hpazo_{ux} \in \Lpazo(L^2_{\bar{\Bnu}^s}(\R^{2d} \times U,\R^d)),
\end{equation*}
the pairs $(\Apazo_{\mathrm{hor}},\Hpazo_{pu})$ and $(\Apazo_{\mathrm{hor}}^{\star},\Cpazo_{\mathrm{hor}}^{\star})$ are exponentially stabilizable. 
\item[$(iv)$] The multiplication operators $\Mpazo_{uu} \in \Lpazo\big(L^2_{\bar{\Bnu}^s}(\R^{2d} \times U,U)$ is boundedly invertible, and 
\begin{equation*}
\Qpazo_{\mathrm{vert}} := \Mpazo_{xu}\Mpazo_{uu}^{-1}\Mpazo_{ux} - \Mpazo_{xx} \in \Lpazo\big(L^2_{\bar{\Bnu}^s}(\R^{2d} \times U,\R^d)\big)
\end{equation*}
is nonnegative. Equivalently, there exists a multiplication operator $\Cpazo_{\mathrm{vert}} \in\Lpazo\big(L^2_{\bar{\Bnu}^s}(\R^{2d} \times U,\R^d)\big)$ such that
$\Qpazo_{\mathrm{vert}} = \Cpazo_{\mathrm{vert}}^{\star}\Cpazo_{\mathrm{vert}}$.
\item[$(v)$] Denoting by
\begin{equation*}
\Apazo_{\mathrm{vert}} := \Mpazo_{px} - \Mpazo_{pu}\Mpazo_{uu}^{-1}\Mpazo_{ux} \in \Lpazo\big(L^2_{\bar{\Bnu}^s}(\R^{2d} \times U,\R^d)\big)
\end{equation*}
the pairs $(\Apazo_{\mathrm{vert}},\Mpazo_{pu})$ and $(\Apazo_{\mathrm{vert}}^{\star},\Cpazo_{\mathrm{vert}}^{\star})$ are exponentially stabilizable as multiplication operators. Equivalently, the finite-dimensional matrix pairs obtained by evaluating their symbols at $\bar{\Bnu}^s$ are uniformly stabilizable and detectable.
\end{enumerate}
\end{taggedhyp}

\begin{rmk}[Horizontal versus vertical modes]
\label{rmk:HorizontalVertical}
Hypotheses \ref{hyp:ET}-$(i)$,$(ii)$ and $(iii)$ are the genuinely Eulerian, Wasserstein-intrinsic assumptions, which mimic the classical Lagrangian ones exposed in Hypotheses \ref{hyp:LT}. Their role is to control the part of the Hessian seen on the horizontal subspace %$L^2_{\P}(\Omega,\R^{2d} \times U;\sigma(\bar{Z}^s))$ 
generated by a stationary Lagrangian realization $\bar{Z}^s := (\bar{X}^s,\bar{\Psi}^s,\bar{u}_{\Lrm}^s)$ of $\bar{\Bnu}^s$. On the other hand, Hypotheses \ref{hyp:ET}-$(iv)$ and $(v)$ are additional pointwise conditions needed to control the lifted solutions on the orthogonal complement of $L^2_{\P}(\Omega,\R^{2d} \times U;\sigma(\bar{Z}^s))$ in the lifted Hilbert space. In practice, these vertical assumptions are much simpler to verify than global Riccati hypotheses, since they reduce to uniform pointwise finite-dimensional stabilizability and detectability of multiplication symbols. %This is precisely the mechanism that becomes transparent in linear-quadratic examples.
\end{rmk}

Let now $(\bar{X}^s,\bar{\Psi}^s) \in L^2_{\P}(\Omega,\R^{2d})$ be such that $(\bar{X}^s,\bar{\Psi}^s)_{\sharp}\P = \bar{\nu}^s$, set
$\bar{u}_{\Lrm}^s := \bar{u}_{\Erm}^s \circ \bar{X}^s$
and define $\bar{Z}^s := (\bar{X}^s,\bar{\Psi}^s,\bar{u}_{\Lrm}^s) \in L^2_{\P}(\Omega,\R^{2d} \times U)$. We shall write $\Pi^s := \mathbb{E}[\,\cdot\,\vert \sigma(\bar{Z^s})]$ for the conditional expectation with respect to $\sigma(\bar{Z}^s)$, which coincides with the orthogonal projection onto the horizontal subspace
\begin{equation*}
\mathsf{H}_{\mathrm{hor}}^s := L^2_{\P}(\Omega,\R^d;\sigma(\bar{Z^s})).
\end{equation*}
Likewise, we denote by $\Hsf_{\mathrm{vert}}^s := \ker \Pi^s$ the orthogonal complement of vertical directions, and recall that the lifted Hamiltonian $\widetilde{\Hpazo} : L^2(\Omega,\R^{2d} \times U) \to \R$ defined in \eqref{eq:HamiltonianLag} is such that $\widetilde{\Hpazo}(\bar{Z}^s) = \Hcal(\bar{\Bnu}^s)$. 

\begin{lem}[Block structure of the lifted Hamiltonian Hessian]
\label{lem:LiftedBlockStructure}
For every block $\widetilde{\Hpazo}_{\alpha\beta}$ in the matrix operator representation \eqref{eq:LiftedBlockStructureLag} of the lifted Hessian $\nabla^2 \widetilde{\Hpazo}(\bar{Z^s})$, there exist a multiplication operator $\widetilde{\Mpazo}_{\alpha\beta}$ and a bounded integral operator $\widetilde{\Kpazo}_{\alpha\beta}$ whose range is contained in $\Hsf_{\mathrm{hor}}^s$, such that
\begin{equation}
\label{eq:LiftedBlockStructure}
\widetilde{\Hpazo}_{\alpha\beta} = \widetilde{\Mpazo}_{\alpha\beta} + \widetilde{\Kpazo}_{\alpha\beta}\Pi^s.
\end{equation}
Moreover, the following holds for every indexes $\alpha,\beta \in \{x,p,u\}$. 
\begin{enumerate}
\item[$(a)$] $\widetilde{\Kpazo}_{\alpha\beta} = 0$ whenever $\alpha,\beta \neq x$;
\item[$(b)$] $\Hsf_{\mathrm{hor}}^s$ and $\Hsf_{\mathrm{vert}}^s$ are invariant under the action of each block $\widetilde{\Hpazo}_{\alpha\beta}$. 
\item[$(c)$] $(\widetilde{\Hpazo}_{\alpha\beta})_{\mid \Hsf_{\mathrm{hor}}^s} = \Cpazo_{\bar{Z^s}} \Hpazo_{\alpha \beta} \Cpazo_{\bar{Z^s}}^{\star}$, namely the restriction of $\widetilde{\Hpazo}_{\alpha\beta}$ to the horizontal space $\Hsf_{\mathrm{hor}}^s$  is unitarily conjugated to $\Hpazo_{\alpha\beta}$ through the composition operator $\Cpazo_{\bar{Z^s}}$. 
\item[$(d)$] $(\widetilde{\Hpazo}_{\alpha\beta})_{\mid \Hsf_{\mathrm{vert}}^s} = \widetilde{\Mpazo}_{\alpha \beta}$, namely the restriction of $\widetilde{\Hpazo}_{\alpha\beta}$ to the vertical space $\Hsf_{\mathrm{vert}}^s$ coincides with the multiplication operator $\widetilde{\Mpazo}_{\alpha\beta}$.
\end{enumerate}
\end{lem}

\begin{proof}
By Proposition \ref{prop:WassHess}, the lifted Hessian $\nabla^2 \widetilde{\Hpazo}(\bar{Z}^s)$ being the second derivative of a rearrangement-invariant functional, it is the sum of a multiplication operator and an integral operator. Moreover, it follows from Lemma \ref{lem:LiftedConjugation} that each block is of the form
\begin{equation*}
\widetilde{\Hpazo}_{\alpha \beta} = \widetilde{\Mpazo}_{\alpha \beta} + \Cpazo_{\bar{Z}^s} \Kpazo_{\alpha \beta} \Cpazo_{\bar{Z}^s}^{\star} 
\end{equation*}
where $\Kpazo_{\alpha \beta}$ appears in \eqref{eq:EulerianBlockSplitting}. Thence, the latter factorize through the conditional expectation onto the sigma-algebra generated by the lifted variable, which yields \eqref{eq:LiftedBlockStructure}. 

Concerning item $(a)$, the fact that the  kernel part vanishes when $\alpha,\beta \neq x$ follows directly from the structure of the Wasserstein Hessian of the relaxed Hamiltonian exposed in \eqref{eq:EulerianBlockSplitting}. Item $(b)$, namely the invariance of both $\Hsf_{\mathrm{hor}}^s$ and $\Hsf_{\mathrm{vert}}^s$ under the action of each block ,is immediate. The conjugation statement of item $(c)$ on $\Hsf_{\mathrm{hor}}^s$ is simply \eqref{eq:HessianUnitaryConjugation} from Lemma \ref{lem:LiftedConjugation} applied to each block, whereas the claim in item $(d)$ that the restriction to $\Hsf_{\mathrm{vert}}^s$ coincides with the multiplication part follows from straightforwardly from \eqref{eq:LiftedBlockStructure} the fact that $\Pi^s H = 0$ for every $H \in \Hsf_{\mathrm{vert}}^s$.
\end{proof}

These preliminary results being laid out, we are ready to state the crux of our approach, which allows to lift the hypotheses from the Eulerian problem to the Lagrangian one. 

\begin{prop}[From intrinsic Eulerian hypotheses to lifted Hilbertian hypotheses]
\label{prop:ETimpliesLT}
Assume that Hypotheses \ref{hyp:ET} hold. Then, every realization $\bar{Z}^s := (\bar{X}^s,\bar{\Psi}^s,\bar{u}_{\Lrm}^s) \in L^2_{\P}(\Omega,\R^{2d} \times U)$ of the relaxed stationary Eulerian triple $\bar{\Bnu}^s \in \Pcal_2(\R^{2d} \times U)$ satisfies Hypotheses \ref{hyp:LT}.
\end{prop}

\begin{proof}
By Lemma \ref{lem:LiftedBlockStructure}, all lifted operators split according to the direct sum decomposition 
\begin{equation*}
L^2_{\P}(\Omega,\R^{2d} \times U) = \Hsf_{\mathrm{hor}}^s \oplus \Hsf_{\mathrm{vert}}^s, 
\end{equation*}
and every block of $\nabla^2 \widetilde{\Hpazo}(\bar{Z^s})$ is block-diagonal with respect to the latter. On the horizontal space $\Hsf_{\mathrm{hor}}^s$, the relevant Lagrangian operators are unitarily conjugated to the intrinsic Eulerian blocks. Therefore, by Proposition \ref{prop:WassHess} and the invariance of bounded invertibility as well as exponential stabilizability and detectability under unitary equivalence, Hypotheses \ref{hyp:ET}-$(i)$, $(ii)$ and $(iii)$ directly imply that Hypotheses \ref{hyp:LT} hold when restricting the operators to the horizontal component.

On the vertical space $\Hsf_{\mathrm{vert}}^s$, all the integral operators vanish and only multiplication operators remain. Then, Hypotheses \ref{hyp:ET}-$(iv)$ and $(v)$ are exactly the finite-dimensional stabilizability and detectability assumptions needed for Hypotheses \ref{hyp:LT} -- namely the Schur complement, the Riccati factorization and the exponentially stabilizing feedbacks -- to hold on that vertical component. Since the lifted operators can be written as direct sums of their horizontal and vertical restrictions, the Hilbertian assumptions listed in Hypotheses \ref{hyp:LT} hold on the whole space.
\end{proof}

\begin{rmk}[An intrinsic tangent-space perspective]
\label{rmk:MetricTangent}
A more intrinsic formulation of the Eulerian second-order assumptions would naturally involve a relaxed tangent space to $\Pcal_2(\R^d)$, in the spirit of Gigli's metric tangent module \cite{Gigli2011} and of the recent approach developed in \cite{Bertucci2025}. Such a viewpoint is attractive, as it allows to consider all first-order directions, including the non-Monge ones which are outside of the reduced tangent space. Although it still possesses some linear-like structure, the resulting object is no longer a Hilbert space, which makes the whole Riccati theory and operator diagonalization approach quite arduous to carry out, and less transparent. That is why we opted to develop the present approach instead, relying on the Hilbertian liftings machinery.
\end{rmk}

\begin{thm}[Exponential turnpike in the Eulerian framework]
\label{thm:TurnpikeEulerian}
Let $(\bar{\nu}^s,\bar{u}_{\Erm}^s)$ be a stationary Eulerian Pontryagin triple satisfying Hypotheses \ref{hyp:ET}. %Assume in addition that $\bar{u}_{\Erm}^s$ stays at a positive distance from $\partial U$ on $\supp(\bar{\mu}^s)$. 

Then, there exist positive constants $\epsilon,\alpha,c > 0$ such that, under the smallness condition
\begin{equation}
\label{eq:SmallnessEul}
W_2(\mu^0,\bar{\mu}^s)
+ \NormLbis{\nabla_{\mu}\varphi(\bar{\mu}^s) \circ \pi^1 + \pi^2}{\R^{2d},\R^d}{\bar{\nu}^s}
+ \| \D_x \nabla_{\mu}\varphi(\bar{\mu}^s)\|_{C^0(\R^d)}
+ \|{\nabla_{\mu}^2 \varphi(\bar{\mu}^s)}\|_{C^0(\R^{2d})}
\leq \epsilon,
\end{equation}
every optimal Pontryagin triple $(\nu^*(\cdot),u_{\Erm}^*(\cdot))$ for $(\Ppazo_{\Erm})$ satisfies
\begin{equation}
\label{eq:EulerianTurnpike}
W_2(\nu^*(t),\bar{\nu}^s) + \inf_{\gamma \in \Gamma(\mu^*(t),\bar{\mu}^s)} \NormLbis{u_{\Erm}^*(t) \circ \pi^1 - \bar{u}_{\Erm}^s \circ \pi^2}{\R^d,\R^d}{\gamma}
\leq c \Big( e^{-\alpha t} + e^{-\alpha(T-t)} \Big)
\end{equation}
for all $t \in [0,T]$.
\end{thm}

\begin{proof}
Let $(\mu^*(\cdot),u_{\Erm}^*(\cdot)) \in \AC^2([0,T],\Pcal_2(\R^d)) \times \Ucal_{\Erm}[\mu^*(\cdot)]$ be an optimal trajectory-control pair for $(\Ppazo_{\Erm})$ and $\nu^*(\cdot) \in \AC^2([0,T],\Pcal_2(\R^{2d}))$ be the state-costate curve given by Theorem \ref{thm:PMPEulerian}. By Proposition \ref{prop:EulerianRealization}, there exists an optimal Pontryagin triple $(X^*(\cdot),\Psi^*(\cdot),u_{\Lrm}^*(\cdot)) \in \AC^2([0,T],L^2_{\P}(\Omega,\R^{2d}) \times \Ucal_{\Lrm}$ for $(\Ppazo_{\Lrm})$ such that
\begin{equation*}
(X^*(t),\Psi^*(t))_{\sharp}\P = \nu^*(t) \qquad \text{and} \qquad u_{\Lrm}^*(t,\omega) = u_{\Erm}^*(t,X^*(t,\omega))
\end{equation*}
for $\Lcal^1 \times \P$-almost every $(t,\omega) \in [0,T] \times \Omega$.
%Fix a realization $(\bar{X}^s,\bar{\Psi}^s)$ of $\bar{\nu}^s$ and set $\bar{u}_{\Lrm}^s = \bar{u}_{\Erm}^s \circ \bar{X}^s$. By \cite[Proposition 2.1]{Cavagnari2022}, after possibly changing the realization on the reference probability space, one may additionally impose $\NormLbis{X^0 - \bar{X}^s}{\Omega,\R^d}{\P} = W_2(\mu^0,\bar{\mu}^s)$.
Besides, following e.g. \cite[Lemma 2.1]{BertucciL2024} (see \cite{Cavagnari2025} for a more general variant), there exists a pair of random variables $(\bar{X}^s,\bar{\Psi}^s) \in L^2_{\P}(\Omega,\R^{2d})$ such that  $(\bar{X}^s,\bar{\Psi}^s)_{\sharp} \P = \bar{\nu}^s$ and 
\begin{equation*}
\NormLbis{X^0 - \bar{X}^s}{\Omega,\R^d}{\P} \leq W_2(\mu^0,\bar{\mu}^s) + \epsilon. 
\end{equation*}
Moreover, by Proposition \ref{prop:WassGrad}, we also have that 
\begin{equation*}
\NormLbis{\nabla \widetilde{\varphi}(\bar{X}^s) + \bar{\Psi}^s}{\Omega,\R^d}{\P}
= \NormLbis{\nabla_{\mu}\varphi(\bar{\mu}^s) \circ \pi^1 + \pi^2}{\R^{2d},\R^d}{\bar{\nu}^s},
\end{equation*}
whereas Proposition \ref{prop:WassHess} yields
\begin{equation*}
\Vert \nabla^2 \widetilde{\varphi}(\bar{X}^s) \Vert_{\Lpazo(L^2_{\P}(\Omega,\R^d))}
\leq \| \D_x \nabla_{\mu}\varphi(\bar{\mu}^s)\|_{C^0(\R^d)}
+ \|{\nabla_{\mu}^2 \varphi(\bar{\mu}^s)}\|_{C^0(\R^{2d})}. 
\end{equation*}
Hence, the Eulerian smallness assumption \eqref{eq:SmallnessEul} with $\epsilon > 0$ directly implies the Lagrangian one \eqref{eq:SmallnessLag} with $2 \epsilon > 0$. Besides, setting $\bar{u}_{\Lrm}^s := \bar{u}_{\Erm}^s \circ \bar{X}^s$ further entails $(\bar{X}^s,\bar{\Psi}^s,\bar{u}_{\Lrm}^s)_{\sharp}\P = \bar{\Bnu}^s$, and it thus follows from Proposition \ref{prop:WassGrad} that 
\begin{equation*}
\nabla \widetilde{\Hpazo}(\bar{X}^s,\bar{\Psi}^s,\bar{u}_{\Lrm}^s) = \nabla_{\Bnu} \Hcal(\bar{\Bnu}^s) \circ (\bar{X}^s,\bar{\Psi}^s,\bar{u}_{\Lrm}^s) = 0,
\end{equation*}
so that the triple $(\bar{X}^s,\bar{\Psi}^s,\bar{u}_{\Lrm}^s) \in L^2_{\P}(\Omega,\R^{2d} \times U)$ is stationary for $(\Ppazo^s_{\Lrm})$, and satisfies Hypotheses \ref{hyp:LT} as a consequence of Proposition \ref{prop:ETimpliesLT} above. Therefore, Theorem \ref{thm:TurnpikeLagrangian} applies, and yields
\begin{equation*}
\NormLbis{X^*(t)-\bar{X}^s}{\Omega,\R^d}{\P}
+ \NormLbis{\Psi^*(t)-\bar{\Psi}^s}{\Omega,\R^d}{\P}
+ \NormLbis{u_{\Lrm}^*(t)-\bar{u}_{\Lrm}^s}{\Omega,U}{\P}
\leq c \Big( e^{-\alpha t} + e^{-\alpha(T-t)} \Big)
\end{equation*}
for all $t \in [0,T]$. Denoting by $\gamma_t := (X^*(t),\bar{X}^s)_{\sharp}\P \in \Gamma(\mu^*(t),\bar{\mu}^s)$, there lastly remains to observe that 
\begin{equation*}
W_2(\nu^*(t),\bar{\nu}^s) \leq \NormLbis{(X^*(t),\Psi^*(t))-(\bar{X}^s,\bar{\Psi}^s)}{\Omega,\R^{2d}}{\P},
\end{equation*}
by construction, whereas
\begin{equation*}
\inf_{\gamma \in \Gamma(\mu^*(t),\bar{\mu}^s)} \NormLbis{u_{\Erm}^*(t) \circ \pi^1 - \bar{u}_{\Erm}^s \circ \pi^2}{\R^d,\R^d}{\gamma} \leq \NormLbis{u_{\Erm}^*(t) \circ \pi^1 - \bar{u}_{\Erm}^s \circ \pi^2}{\R^d,\R^d}{\gamma_t}
= \NormLbis{u_{\Lrm}^*(t)-\bar{u}_{\Lrm}^s}{\Omega,U}{\P}.
\end{equation*}
Combining the two previous estimates yields \eqref{eq:EulerianTurnpike} and closes the proof of Theorem \ref{thm:TurnpikeEulerian}.
\end{proof}

%\begin{rmk}[Boundary stationary controls]
%If the stationary control lies on $\partial U$, the maximizing feedback may become saturated and lose the smooth structure that underlies the previous turnpike argument. This is not merely technical. Even for one-dimensional constrained linear-quadratic problems, one can construct families of optimal trajectories for which no uniform exponential turnpike estimate with horizon-independent prefactor holds. The interiority assumption in Theorem \ref{thm:TurnpikeEulerian} precisely excludes this phenomenon.
%\end{rmk}

\section{Conclusion and perspectives}\label{section:Conclusion}

In this paper, we have established exponential turnpike theorems for a class of nonlinear deterministic meanfield optimal control problems, both in a lifted Lagrangian framework, and in an intrinsic Eulerian formulation over Wasserstein spaces. The article combines three ingredients: a systematic comparison between the lifted and Wasserstein first- and second-order differential structures, static first-order optimality conditions in the Lagrangian and Eulerian settings, and an operator-theoretic turnpike analysis driven by the Hessian of the stationary Hamiltonian. From the methodological viewpoint, one of the main messages is that the Eulerian second-order assumptions naturally split into a horizontal part, transferred to the lifted space by unitary conjugation, and a vertical part, which only involves multiplication operators and can therefore be checked pointwisely. This decomposition clarifies both the scope and the limitations of the usual intrinsic Wasserstein Hessians when singular perturbations are present in the lift. In our opinion, several interesting questions remain open.

\begin{enumerate}[wide, labelindent=0pt]
\item[$\diamond$] A first natural direction to investigate is to replace the present Hilbertian lifting procedure by a fully intrinsic second-order theory on a relaxed metric tangent space to $\Pcal_2(\R^d)$, in the spirit of recent developments on tangent structures beyond Monge directions and on regularised approximations of the squared Wasserstein distance, see for instance \cite{Bertucci2025, BertucciL2024}. \vspace{-0.15cm} 
\item[$\diamond$] A second direction would be to treat boundary stationary controls and genuinely saturated optimal arcs, for which the smooth Riccati diagonalization used here is no longer directly available. This would connect the present analysis with the broader constrained/singular turnpike literature and with recent constrained or singular mean-field models, see, e.g., \cite{CohenSun2024, HuLyu2025, TrelatZuazuaSurvey}. \vspace{-0.15cm} 
\item[$\diamond$] A third direction would be to revisit the static and dynamic arguments at the level of occupation measures or relaxed controls, in order to incorporate stronger local topologies on Eulerian controls and possibly recover more intrinsic compactness properties, see for example \cite{BahlaliMezerdiMezerdi2018, BouveretDumitrescuTankov2020, Henrion2024, Lasserre2008}. This might also allow to study the fully Eulerian turnpike problem, wherein the constraint of the static problem is the vanishing of the divergence instead of that of the full vector field. 
\end{enumerate}

Finally, another natural extension would be to move beyond deterministic first-order continuity dynamics and towards controlled McKean-Vlasov diffusions, common-noise models, and hierarchical particle / kinetic / hydrodynamic regimes. In this context, it would be very interesting to connect the present ODE-based analysis with the rapidly developing PDE-side turnpike theory for first- and second-order mean field games and mean field control systems, and more generally with the multiscale passage from interacting particle systems to kinetic or hydrodynamic descriptions, see \cite{BayraktarJian2026, Cecchin2026, CirantDeBernardi2025, CirantPorretta2021, HertyZhou2025, Paul2022}. \medskip

\noindent {\small \textbf{Acknowledgements:} The authors thank Alessandro Pinzi for helping fix a small mistake in the original proof of Theorem 4.8. B.B.-W. was partially supported by the French government under the France 2030 program, reference ANR-11-IDEX-0003 within the OI H-Code. D.S. was supported by 
a MCSA Cofund fellowship under the programme ``UNIPhD''.}

%%%%%%%%%%%%%%%%%%%%%%%%%%%%%%%%%%%%%%%%%%%%%%%%%%%%%%%%%%%%%%%%%%%%%%%
%					         NEW SECTION AHEAD    					  %
%%%%%%%%%%%%%%%%%%%%%%%%%%%%%%%%%%%%%%%%%%%%%%%%%%%%%%%%%%%%%%%%%%%%%%%

\addcontentsline{toc}{section}{Appendix}
\section*{Appendix}

\setcounter{Def}{0} 
\setcounter{section}{0}
\renewcommand{\thesection}{A} 
\renewcommand{\thesubsection}{A} 

\subsection{Proof of Proposition \ref{prop:WassHess}}
\label{section:AppendixHess}

\setcounter{Def}{0} \renewcommand{\thethm}{A.\arabic{Def}} 
\setcounter{equation}{0} \renewcommand{\theequation}{A.\arabic{equation}}

In this appendix, we provide an elementary proof of the second-order differentiability and Hessian representation formula for $\widetilde{\Phi} : L^2_{\P}(\Omega,\R^d) \to \R$. The latter is based on the following technical lemma, which is a minor adaptation of the arguments from \cite[Proposition 5.1]{SetValuedPMP}. 

\begin{lem}[Small-o estimates in weighted Wasserstein metric]
\label{lem:Small-o}
For every $F \in C^0_b(\R^d,\R^{d \times d})$ and any two measures $\mu,\nu \in \Pcal_2(\R^{2d})$, it holds that 
\begin{equation*}
\INTDom{\Big(F(x+s(y-x))-F(x) \Big)(x-y)}{\R^{2d}}{\gamma(x,y)} = o(W_{2,\gamma}(\mu,\nu))
\end{equation*}
for each $\gamma \in \Gamma(\mu,\nu)$ and all $s \in [0,1]$. 
\end{lem}

\begin{proof}
To begin with, note that by Prokhorov's theorem (see e.g. \cite[Theorem 5.1.3]{AGS}), there exists for each $\gamma \in \Gamma(\mu,\nu)$ and every $\epsilon >0$ a radius $R_{\epsilon} > 0$ such that 
\begin{equation*}
\gamma \Big(\R^{2d} \setminus B_{R_{\epsilon}}(0)^2 \Big) \leq \frac{\epsilon}{8 \NormC{F}{0}{\R^d}}, 
\end{equation*}
which implies in particular that 
\begin{equation*}
\begin{aligned}
& \INTDom{\Big|F(x+s(y-x))-F(x)\Big||x-y|}{\R^{2d}}{\gamma(x,y)} \\
& \hspace{0.45cm}  \leq \INTDom{\Big|F(x+s(y-x))-F(x)\Big||x-y|}{B_{R_{\epsilon}}(0)^2}{\gamma(x,y)} 
+ 2 \NormC{F}{0}{\R^d} \INTDom{|x-y|}{\R^{2d} \setminus B_{R_{\epsilon}}(0)^2}{\gamma(x,y)} \\
& \hspace{0.45cm}  \leq \INTDom{|F(x+s(y-x))-F(x)||x-y|}{B_{R_{\epsilon}}(0)^2}{\gamma(x,y)} + \frac{\epsilon}{4} W_{2,\gamma}(\mu,\nu)
\end{aligned}
\end{equation*}
by H\"older's inequality. Focusing now on the first term, observe that $F : B_{R_{\epsilon}}(0) \to \R$ is uniformly continuous, so that there exists $\eta_{\epsilon} > 0$ for which 
$|F(x+s(y-x))-F(x)| < \frac{\epsilon}{4}$
whenever $x,y \in B_{R_{\epsilon}}(0)$ satisfy $|x-y| < \eta_{\epsilon}$. This observation allows us to further infer that 
\begin{equation*}
\begin{aligned}
& \INTDom{|F(x+s(y-x))-F(x)||x-y|}{B_{R_{\epsilon}}(0)^2}{\gamma(x,y)} \\
& \hspace{0.45cm} \leq \frac{\epsilon}{4} \INTDom{|x-y|}{\R^{2d}}{\gamma(x,y)} + \INTDom{|F(x+s(y-x))-F(x)||x-y|}{B_{R_{\epsilon}}(0)^2 \, \cap\{ |x-y| > \eta_{\epsilon}\}}{\gamma(x,y)} \\
& \hspace{0.45cm}  \leq \frac{\epsilon}{4} W_{2,\gamma}(\mu,\nu) + \frac{4 R_{\epsilon} \NormC{F}{0}{\R^d}}{\eta_{\epsilon}^2} W_{2,\gamma}^2(\mu,\nu)
\end{aligned}
\end{equation*}
where we used H\"older's and Chebyshev's inequalities. Combining all the above, we finally obtain that 
\begin{equation*}
\INTDom{|F(x+s(y-x))-F(x)||x-y|}{\R^{2d}}{\gamma(x,y)} \leq \frac{\epsilon}{2} W_{2,\gamma}(\mu,\nu) + \frac{4 R_{\epsilon} \NormC{F}{0}{\R^d}}{\eta_{\epsilon}^2} W_{2,\gamma}^2(\mu,\nu). 
\end{equation*}
In particular, this implies that for each $\epsilon > 0$, there exists a $\delta := \epsilon \eta_{\epsilon}^2/(8R_{\epsilon} \NormC{F}{0}{\R^d}) > 0$ for which 
\begin{equation*}
\INTDom{|F(x+s(y-x))-F(x)||x-y|}{\R^{2d}}{\gamma(x,y)} \leq \epsilon W_{2,\gamma}(\mu,\nu)
\end{equation*}
whenever $W_{2,\gamma}(\mu,\nu) \leq \delta$, which precisely amounts to the latter being a small-o of $W_{2,\gamma}(\mu,\nu)$. 
\end{proof}

\begin{proof}[Proof of Proposition \ref{prop:WassHess}]
Similarly to what we did in the proof of Proposition \ref{prop:WassGrad}, consider the plan $\Bgamma := (X,G,H)_{\sharp} \P \in \Pcal_2(\R^{3d})$, and observe that
\begin{equation}
\label{eq:TaylorHessian0}
\begin{aligned}
\langle \nabla \widetilde{\Phi}(X+G),H \rangle_{L^2_{\P}(\Omega,\R^d)} & = \INTDom{\Big\langle \nabla_{\mu} \Phi((X+G)_{\sharp} \P) \circ (X+G)(\omega) , H(\omega)\Big\rangle}{\Omega}{\P(\omega)} \\
& = \INTDom{\big\langle \nabla_{\mu} \Phi((X+G)_{\sharp} \P)(x+g) , h \big\rangle}{\R^{3d}}{\Bgamma(x,g,h)}.
\end{aligned}
\end{equation}
Our goal now is to carry out a first-order expansion of the previous expression. To begin with, note that it follows from the usual Taylor theorem that
\begin{multline}
\label{eq:TaylorHessian1}
\nabla_{\mu} \Phi((X+G)_{\sharp} \P)(x+g) = \nabla_{\mu} \Phi((X+G)_{\sharp} \P)(x) + \D_x \nabla_{\mu}\Phi((X+G)_{\sharp} \P)(x)g \\
+ \INTSeg{\Big( \D_x \nabla_{\mu}\Phi((X+G)_{\sharp} \P)(x+sg)- \D_x \nabla_{\mu}\Phi((X+G)_{\sharp} \P)(x) \Big)g}{s}{0}{1}.
\end{multline}
for all $(x,g,h) \in \R^{3d}$. Since we assumed that $\D_x \nabla_{\mu} \Phi((X+G)_{\sharp} \P) \in C^0_b(\R^d,\R^{d \times d})$, it follows from Lemma \ref{lem:Small-o} above that 
\begin{equation}
\label{eq:TaylorHessian2}
\begin{aligned}
\INTDom{\INTSeg{\bigg\langle \Big( \D_x \nabla_{\mu}\Phi((X+G)_{\sharp} \P)(x+s g)- \D_x \nabla_{\mu}\Phi((X+G)_{\sharp} \P)(x) \Big)  g , h \bigg\rangle}{s}{0}{1}}{\R^{3d}}{\Bgamma(x,g,h)} = \mathrm{o} \Big( \NormLbis{G}{\Omega}{\P} \Big)
\end{aligned}
\end{equation}
and similarly 
\begin{equation}
\label{eq:TaylorHessian3}
\INTDom{\Big\langle \D_x \nabla_{\mu}\Phi((X+G)_{\sharp} \P)(x)  g , h \Big\rangle}{\R^{3d}}{\Bgamma(x,g,h)} = \INTDom{\Big\langle \D_x \nabla_{\mu}\Phi(X_{\sharp} \P)(x)  g , h \Big\rangle}{\R^{3d}}{\Bgamma(x,g,h)} + \mathrm{o} \Big( \NormLbis{G}{\Omega}{\P} \Big). 
\end{equation}
Thus, upon combining \eqref{eq:TaylorHessian0}, \eqref{eq:TaylorHessian1}, \eqref{eq:TaylorHessian2} and \eqref{eq:TaylorHessian3}, we arrive at the following intermediate expansion 
\begin{multline}
\label{eq:TaylorHessian5}
\langle \nabla \widetilde{\Phi}(X+G),H \rangle_{L^2_{\P}(\Omega,\R^d)} = \INTDom{\Big\langle \nabla_{\mu} \Phi((X+G)_{\sharp} \P)(x) , h \Big\rangle}{\R^{3d}}{\Bgamma(x,g,h)} \\
+ \INTDom{\Big\langle \D_x \nabla_{\mu}\Phi(X_{\sharp} \P)(x)  g , h \Big\rangle}{\R^{3d}}{\Bgamma(x,g,h)} + \mathrm{o} \Big( \NormLbis{G}{\Omega}{\P} \Big).
\end{multline}
At this stage, it can be deduced straightforwardly from the definition of Wasserstein differential that 
\begin{multline}
\label{eq:TaylorHessian6}
\nabla_{\mu} \Phi((X+G)_{\sharp}\P)(x)
= \nabla_{\mu} \Phi(X_{\sharp}\P)(x) + \INTDom{\nabla^2_{\mu} \Phi(X_{\sharp} \P)(x,y)v}{\R^{3d}}{\Bgamma(y,v,w)} \\
+ \INTSeg{\INTDom{\Big( \nabla^2_{\mu} \Phi \Big( (X+sG)_{\sharp} \P \Big)(x,y+sv) - \nabla^2_{\mu} \Phi(X_{\sharp} \P)(x,y) \Big)  v}{\R^{3d}}{\Bgamma(y,v,w)}}{s}{0}{1} 
\end{multline}
for all $x \in \R^d$, where it can be shown yet again by using Lemma \ref{lem:Small-o} that 
\begin{multline}
\label{eq:TaylorHessian7}
\INTDom{\INTDom{\INTSeg{\bigg\langle \Big( \nabla^2_{\mu} \Phi \Big( (X+sG)_{\sharp} \P \Big)(x,y+sv) - \nabla^2_{\mu} \Phi(X_{\sharp} \P)(x,y) \Big) v , h \bigg\rangle}{s}{0}{1}}{\R^{3d}}{\Bgamma(y,v,w)}}{\R^{3d}}{\Bgamma(x,g,h)}  \\
= \mathrm{o} \Big( \NormLbis{G}{\Omega}{\P} \Big). 
\end{multline}
Whence, we finally get by merging \eqref{eq:TaylorHessian5}, \eqref{eq:TaylorHessian6} and \eqref{eq:TaylorHessian7} that 
\begin{equation*}
\begin{aligned}
\langle \nabla \widetilde{\Phi}(X+G),H \rangle_{L^2_{\P}(\Omega,\R^d)} & = \INTDom{\langle \nabla_{\mu} \Phi(X_{\sharp} \P)(x) , h \rangle}{\R^{3d}}{\Bgamma(x,g,h)} + \INTDom{\Big\langle \D_x \nabla_{\mu}\Phi(X_{\sharp} \P)(x)  g , h \Big\rangle}{\R^{3d}}{\Bgamma(x,g,h)} \\
& \hspace{0.45cm} + \INTDom{\INTDom{\Big\langle \nabla^2_{\mu} \Phi(X_{\sharp} \P)(x,y)v , h \Big\rangle}{\R^{3d}}{\Bgamma(x,g,h)}}{\R^{3d}}{\Bgamma(y,v,w)} + \mathrm{o} \Big( \NormLbis{G}{\Omega}{\P} \Big) \\
& = \langle \nabla \widetilde{\Phi}(X),H \rangle_{L^2_{\P}(\Omega,\R^d)} + \INTDom{\big\langle \D_x \nabla_{\mu}\Phi(X_{\sharp} \P)(X(\omega))  G(\omega) , H(\omega) \big\rangle}{\Omega}{\P(\omega)} \\
& \hspace{0.45cm} + \INTDom{\INTDom{\Big\langle \nabla^2_{\mu} \Phi(X_{\sharp} \P)(X(\omega),X(\theta)) G(\theta), H(\omega) \Big\rangle}{\Omega}{\P(\theta)}}{\Omega}{\P(\omega)} + \mathrm{o} \Big( \NormLbis{G}{\Omega}{\P} \Big)
\end{aligned}
\end{equation*}
which due to our working assumptions finally yields that $\widetilde{\Phi} :L^2_{\P}(\Omega,\R^d) \to \R$ is twice continuously Fréchet differentiable and that its Hessian is given by \eqref{eq:HessianFormula}. 
\end{proof}

%%%%%%%%%%%%%%%%%%%%%%%%%%%%%%%%%%%%%%%%%%%%%%%%%%%%%%%%%%%%%%%%%%%%%%%%
%							 BILBIOGRAPHY AHEAD		                   %
%%%%%%%%%%%%%%%%%%%%%%%%%%%%%%%%%%%%%%%%%%%%%%%%%%%%%%%%%%%%%%%%%%%%%%%%

{\footnotesize 
\paragraph*{{\footnotesize Acknowledgements:}} 
}

\bibliographystyle{alpha}
\small
{\footnotesize
\bibliography{../../ControlWassersteinBib}
}

\end{document}